\documentclass[10pt,a4paper]{article}
\usepackage{xcolor}
\usepackage{mdframed}
\usepackage[utf8]{inputenc}
\usepackage{amsmath, amsthm, amssymb}
\usepackage{tikz-cd}
\usepackage{enumerate}
\usepackage{appendix}
\usepackage[hidelinks]{hyperref}
\usepackage{color}

\definecolor{mygray}{gray}{0.75}

\newtheorem{tvrz}{Proposition}[section]
\newtheorem{lemma}[tvrz]{Lemma}
\newtheorem{theorem}[tvrz]{Theorem}
\newtheorem{cor}[tvrz]{Corollary}
\theoremstyle{definition}
\newtheorem{definice}[tvrz]{Definition}
\theoremstyle{remark}
\newtheorem{rem}[tvrz]{Remark}
\theoremstyle{definition}
\newtheorem{mdexample}[tvrz]{Example}
\newenvironment{example}%
{\begin{mdframed}[topline=false, rightline=false, bottomline=false, linewidth=0.2em, linecolor=mygray, innerleftmargin=0.5em, innerrightmargin=0]\begin{mdexample}}%
{\end{mdexample}\end{mdframed}}

\def\^{\wedge}
\def\<{\langle}
\def\>{\rangle}

\def\X{\mathfrak{X}}
\def\g{\mathfrak{g}}

\def\q{\mathfrak{q}}
\def\dr{\mathrm{d}}
\def\h{\mathfrak{h}}

\def\R{\mathbb{R}}

\def\D{\mathrm{D}}
\def\F{\mathcal{F}}

\def\U{\mathcal{U}}

\def\Z{\mathbb{Z}}

\def\gTM{\mathbb{T}M}
\def\fA{\mathbf{A}}
\def\ol{\overline}
\def\fPsi{\mathbf{\Psi}}

\def\btr{\blacktriangleright}
\def\fg{\mathbf{g}}

\def\T{\mathcal{T}}

\def\frR{\mathfrak{R}}
\def\frk{\mathfrak{k}}

\def\fk{\mathbf{k}}
\def\fK{\mathbf{K}}

\def\cD{\nabla}

\def\fP{\mathbf{P}}

\def\f1{\mathbf{1}}

\def\tI{\text{I}}

\def\gm{\mathbf{G}}

\def\~{\widetilde}

\newcommand{\Li}[1]{\mathcal{L}_{#1}}

\def\dra{\dashrightarrow}
\def\rat{\rightarrowtail}

\def\cCA{\text{\textbf{CAlg}}}
\def\cLA{\text{\textbf{LAlg}}}
\def\cMan{\text{\textbf{Man}}}
\def\dbl{[\![}
\def\dbr{]\!]}

\DeclareMathOperator{\Lie}{Lie}
\DeclareMathOperator{\Df}{Df}

\DeclareMathOperator{\im}{im}

\DeclareMathOperator{\gSO}{SO}

\DeclareMathOperator{\Ric}{Ric}

\DeclareMathOperator{\End}{End}

\DeclareMathOperator{\Ad}{Ad}

\DeclareMathOperator{\ad}{ad}

\DeclareMathOperator{\Div}{div}
\DeclareMathOperator{\an}{an}
\DeclareMathOperator{\In}{in}
\DeclareMathOperator{\gr}{gr}
\DeclareMathOperator{\rk}{rk}
\DeclareMathOperator{\supp}{supp}

\begin{document}
\begin{flushright}
\today
\end{flushright}
\vspace{0.7cm}
\begin{center}

\baselineskip=13pt {\Large \bf{Hitchhiker's Guide to Courant Algebroid Relations}\\}
 \vskip0.5cm
 {\it Dedicated to all people carrying a towel.}  
 \vskip0.7cm
 {\large{ Jan Vysoký$^{1}$}}\\
 \vskip0.6cm
$^{1}$\textit{Faculty of Nuclear Sciences and Physical Engineering, Czech Technical University in Prague\\ Břehová 7, Prague 115 71, Czech Republic, jan.vysoky@fjfi.cvut.cz}\\
\vskip0.3cm
\end{center}

\begin{abstract}
Courant algebroids provide a useful mathematical tool (not only) in string theory. It is thus important to define and examine their morphisms. To some extent, this was done before using an analogue of canonical relations known from symplectic geometry. However, it turns out that applications in physics require a more general notion. We aim to provide a self-contained and detailed treatment of Courant algebroid relations and morphisms. A particular emphasis is placed on providing enough motivating examples. In particular, we show how Poisson--Lie T-duality and Kaluza--Klein reduction of supergravity can be interpreted as Courant algebroid relations compatible with generalized metrics (generalized isometries). 
\end{abstract}


\section{Introduction}
Recently, generalized geometry and Courant algebroids came to prominence as a very useful tool to understand certain aspects of string theory. Their relevance was recognized already in the founding paper \cite{Hitchin:2004ut} and the famous letters \cite{Severa:2017oew}. Courant algebroids can be used to describe current algebras of $\sigma$-models \cite{Alekseev:2004np}, various aspects of T-duality \cite{Baraglia:2013wua, Cavalcanti:2011wu, Garcia-Fernandez:2016ofz, Grana:2008yw}, geometrical approach to (exceptional, heterotic) supergravity \cite{Coimbra:2011nw, Coimbra:2012af, 2013arXiv1304.4294G, Strickland-Constable:2019qti} and Poisson--Lie T-duality \cite{Severa:2015hta,Severa:2016prq, Severa:2017kcs, Severa:2018pag}. See also recent attempts to describe global geometry of double field theory using para-Hermitian manifolds \cite{Freidel:2017yuv, Freidel:2018tkj,Svoboda:2018rci,Vaisman:2012px}, and our contributions to the subject\footnote{Here should be a pile of self-citations. Instead, we refer to relevant works in the body of the paper.}. Of course, this is by no means a comprehensive list of references. However, it should testify to the importance of proper understanding of mathematical apparatus of  Courant algebroids. 

The first axiomatic definition of Courant algebroids appeared in \cite{liu1997manin}, generalizing an example introduced in \cite{courant}. We use a more recent formulation which was introduced in \cite{Severa:2017oew} and proved to be equivalent to the original one in \cite{1999math.....10078R}. For a complete history of Courant algebroids, we refer to the nice review paper \cite{Kosmann_Schwarzbach_2013}.

In principle, Courant algebroids are just vector bundles equipped with some additional structures naturally generalizing the notion of a quadratic Lie algebra (i.e. equipped with a non-degenerate symmetric bilinear and $\ad$-invariant form). One thus expects that their morphisms can be defined as vector bundle maps preserving these additional structures in some ``obvious way''. This can be easily done only for vector bundle maps over diffeomorphisms. Such an assumption poses some serious limitations, e.g. it works only for Courant algebroids over diffeomorphic base manifolds. Can one generalize this to vector bundle maps over arbitrary smooth maps? Note that the answer is not so straightforward even for a more known case of Lie algebroids. 
 
In fact, finding the correct notion of Courant algebroid morphism proved to be quite an elusive objective. First, none of the founding mathematical papers even raise the question. To my knowledge, the only explicit definition for arbitrary base map appeared in a relatively unknown\footnote{Thanks to B. Jurčo for making me aware of it.} paper \cite{popescu1999generalized} as an example of a morphism of so called \textit{generalized algebroids}. This is what we call a \textit{classical Courant algebroid morphism}, see Subsection \ref{subsec_graph}. Note that they use the original skew-symmetric version of Courant algebroid introduced in \cite{liu1997manin}.

Now, recall that for a pair of symplectic manifolds $(M_{1},\omega_{1})$ and $(M_{2},\omega_{2})$, one says that a smooth map $\varphi: M_{1} \rightarrow M_{2}$ is symplectic, iff $\omega_{1} = \varphi^{\ast}(\omega_{2})$. To ensure that the induced pullback map $\varphi^{\ast}: C^{\infty}(M_{2}) \rightarrow C^{\infty}(M_{1})$ intertwines the respective Poisson brackets, one considers only diffeomorphisms. However, the resulting \textit{symplectic category} has too few arrows. Therefore, it was suggested in \cite{weinstein1982symplectic} to consider a larger set of morphisms consisting of Lagrangian submanifolds\footnote{Often, they are also called \textit{canonical relations} or \text{Lagrangian correspondences}.} of $M_{1} \times \ol{M}_{2}$, where $\ol{M}_{2}$ usually denotes the symplectic manifold $(M_{2},-\omega_{2})$. On the set level, there is a straightforward composition rule of such morphisms. However, the resulting set may fail to be a submanifold, hence not all arrows can be composed. Nevertheless, it is still very useful to work with this generalized definition, called \textit{symplectic ``category''} by A. Weinstein, see also \cite{bates1997lectures}. 

Motivated by this idea and the fact that Courant algebroids \textit{are} symplectic (super)manifolds \cite{roytenberg2002structure}, Courant algebroid morphisms were defined in the unpublished manuscript \cite{alekseevxu} as Dirac structures in the product Courant algebroid $E_{1} \times \ol{E}_{2}$ supported on a graph of a smooth base manifold map. Note that only Courant algebroids equipped with a fiber-wise metric of a split signature are considered. In \cite{bursztyn2008courant}, they dropped this assumption and considered Courant algebroid morphisms to be maximally isotropic involutive subbundles. However, as we show in Example \ref{ex_RHcomQRe0}, such morphisms do not compose even on the level of linear algebra. More generally, one can consider Dirac structures in $E_{1} \times \ol{E}_{2}$ supported on an arbitrary submanifold, which were called \textit{Courant algebroid relations} in \cite{li2014dirac}. This particular paper contains a more thorough discussion of conditions on two Courant algebroid relations to compose to a Courant algebroid relation and we have used it as our main reference. See also \cite{li2009courant, meinrenken2017dirac}. Note that in all these papers, they consider only Courant algebroids with a fiber-wise metric of a split signature. 

Let us now summarize the main reasons leading us to write this paper. 

(i) People mostly considered Courant algebroids with a fiber-wise metric of a split signature. This contains some prominent examples, e.g. exact Courant algebroids. However, there are cases relevant for applications in physics, where the signature is more general, e.g. heterotic Courant algebroids. At first glance, one may solve this by replacing Lagrangian subbundles by maximally isotropic ones. However, the composed relation may fail to maximally isotropic. Fortunately, it turns out that there is no real argument for keeping the maximality requirement and it seems more natural to work with arbitrary isotropic involutive subbundles.  

(ii) Except for \cite{li2014dirac}, most of the papers do not elaborate on detailed conditions for two Courant algebroid relations to compose to a new Courant algebroid relation. Our intention is to fill all the gaps, providing a self-contained careful treatment of this topic. In particular, we take the liberty to examine in detail the rather intriguing nature of involutive subbundles of Courant algebroids supported on an arbitrary submanifold of the base. In fact, some of their aspects are very different to the seemingly similar case of Lie algebroids.

(iii) We put an emphasis on examples. In particular, one can show that our Courant algebroid relations contain the classical Courant algebroid morphisms of \cite{popescu1999generalized}. In Example \ref{ex_compfails}, we explicitly demonstrate how a composition of two Courant algebroid relations fails to be a Courant algebroid relation. We show that several situations relevant for applications in physics naturally fit into the geometrical framework introduced in this article. 

We are aware of the fact that Courant algebroids can be interpreted as degree $2$ symplectic NQ manifolds. There should be a correspondence of Courant algebroid relations discussed here and some generalization of Weinstein's symplectic ``category'' suitable for differential graded manifolds. However, we keep this discussion for a future endeavor.

\textbf{The paper is organized as follows:}

In Section \ref{sec_involutive}, we build our way to the definition of (almost) involutive structures in Courant algebroids. For a given Courant algebroid, these are subbundles supported on some submanifold of its base, compatible in some sense with all the additional structures i.e. the fiber-wise metric, the anchor and the  bracket. We discuss consequences of all requirements in detail. The pinnacles of this section are Definition \ref{def_involutive} and Proposition \ref{tvrz_invongen}.

Section \ref{sec_relations} deals with the pivotal notion of this paper, Courant algebroid relations. After giving the definition and some basic examples, we focus on the intriguing problem of their composition. One has to discuss geometrical conditions ensuring that the composed relation is a well-behaved subbundle, resulting in the notion of clean composition of relations. We prove the theorem claiming that under these assumptions, the composition is again a Courant algebroid relation. By its nature, this section can be sometimes quite technical, hence tedious. Essential notions can be thus found in Definition \ref{def_rel}, Definition \ref{def_composition}, Proposition \ref{tvrz_RcircR} and Theorem \ref{thm_composition}

Section \ref{sec_examples} is fully devoted to bring enough interesting examples of Courant algebroid relations. In Subsection \ref{subsec_graph}, we consider the relation obtained as a graph of a vector bundle morphism $\F$ over an arbitrary base map. Examining the conditions imposed on $\F$, we recover the definition which appeared in \cite{popescu1999generalized}, calling it a \textit{classical Courant algebroid morphism} in Definition \ref{def_classicalCA}. We provide an example where the subbundle $\gr(\F)$ is not maximally isotropic and an example of two Courant algebroid relations which \textit{cannot} be composed.

In Subsection \ref{subsec_Dorfman}, a natural functor from the ``category'' of Lie algebroids (and their relations) to the ``category'' of Courant algebroids is constructed. In Theorem \ref{thm_Df}, we call it the \textit{Dorfman functor} since for the tangent bundle $TM$, one obtains the standard Dorfman bracket on the \textit{generalized tangent bundle} $\gTM := TM \oplus T^{\ast}M$. In Example \ref{ex_paraHermitian}, it is shown how the Dorfman functor appears naturally in para-Hermitian geometry. 

Subsection \ref{subsec_reduction} was one of our main motivations to deal with Courant algebroid relations. We show that the (somewhat simplified) reduction procedure \cite{Bursztyn2007726} of equivariant Courant algebroids can be naturally interpreted as a Courant algebroid morphism. In Proposition \ref{tvrz_invreduction}, one can use this to obtain conditions on reductions of involutive structures (in particular Dirac structures). 

Finally, in Subsection \ref{subsec_tworeduced}, one can show that there is a canonical Courant algebroid morphism between two Courant algebroids obtained via the reduction procedure by a Lie group action and its restriction to any closed subgroup. In particular, we show how Poisson--Lie T-duality can be interpreted as a Courant algebroid relation between two reduced Courant algebroids. 

A concept of generalized metric and its interplay with Courant algebroids proved to be useful for applications in physics. Naturally, one would like to impose some its compatibility with relations of Courant algebroids. This is done in Section \ref{sec_geniso}, resulting in the notion of generalized isometry. We examine this definition for concrete examples of previous sections. In particular, we show that Poisson--Lie T-duality can be viewed as a generalized isometry of the $\sigma$-model backgrounds encoded using the generalized metrics.

In final Section \ref{sec_connections}, we recall Courant algebroid connections and show how their compatibility with Courant algebroid relations can be imposed. It is then shown what happens with induced torsion and curvature tensors. Interestingly, generalized Riemann tensors are related only for torsion-free connections, as we demonstrate on a simple counter-example.

Finally, some of the necessary linear algebra statements were moved to Appendix \ref{sec_supplement}. 
\section{Involutive structures on Courant algebroids} \label{sec_involutive}
First, let us very briefly recall the (modern) definition of Courant algebroids as it was given by Roytenberg in his thesis \cite{1999math.....10078R} and by Ševera and Weinstein in \cite{Severa:2017oew,Severa:2001qm}. 
\begin{definice} \label{def_courant}
A \textbf{Courant algebroid} is a $4$-tuple $(E,\rho,\<\cdot,\cdot\>,[\cdot,\cdot])$, where
\begin{enumerate}[(i)]
\item $q: E \rightarrow M$ is a real vector bundle of a finite rank.
\item $\rho: E \rightarrow TM$ is a vector bundle map over the identity called the \textbf{anchor};
\item $\<\cdot,\cdot\>$ is a fiber-wise metric on $E$, sometimes also denoted as $g_{E}$;
\item $[\cdot,\cdot]$ is an $\R$-bilinear bracket subject to the axioms:
\begin{enumerate}[C1)]
\item $[\psi,f \psi'] = f [\psi,\psi'] + \Li{\rho(\psi)}(f) \psi'$; 
\item $[\psi,[\psi',\psi'']] = [[\psi,\psi'],\psi''] + [\psi',[\psi,\psi'']]$;
\item $\Li{\rho(\psi)}\<\psi',\psi''\> = \< [\psi,\psi'],\psi''\> + \<\psi',[\psi,\psi'']\>$;
\item $\< [\psi,\psi], \psi'\> = \frac{1}{2} \Li{\rho(\psi')} \<\psi,\psi\>$;
\end{enumerate}
for all sections $\psi,\psi',\psi'' \in \Gamma(E)$ and $f \in C^{\infty}(M)$.
\end{enumerate}
\end{definice}
Note that one can show that C1) and C2) together imply that $\rho([\psi,\psi']) = [\rho(\psi),\rho(\psi')]$ for all $\psi,\psi' \in \Gamma(E)$. Some authors add this as an additional axiom. There is an induced bundle map $\rho^{\ast}: T^{\ast}M \rightarrow E$ determined uniquely by the equation 
\begin{equation} \label{eq_rhoastdef} \< \rho^{\ast}(\alpha), \psi \> = \< \alpha, \rho(\psi)\>, \end{equation}
for all $\alpha \in \Omega^{1}(M)$ and $\psi \in \Gamma(E)$. It automatically satisfies $\rho \circ \rho^{\ast} = 0$ and induces a canonical $\R$-linear map $\D: C^{\infty}(M) \rightarrow \Gamma(E)$ given by composition $\D = \rho^{\ast} \circ \dr$. We usually abuse the notation and use $\<\cdot,\cdot\>$ also for the pairing of vector fields and $1$-forms. Note that we have no intentions to work with the skew-symmetric version of Courant algebroids \cite{liu1997manin}.

Let us henceforth assume that $E$ is a vector bundle over $M$ carrying a Courant algebroid structure. When we start talking about a vector bundle equipped with a fiber-wise metric, the first question should be ``what is its signature?''. This inquiry has a good meaning. Indeed, the signature of any fiber-wise metric has to be locally constant, see \cite{2013arXiv1307.2171D}, hence constant on any connected component of the base $M$. 

Without the loss of (too much) generality, we may thus \textit{assume that the signature of $\<\cdot,\cdot\>$ is constant} and denote it as $(p,q)$. In general, there are no restrictions on its possible values for Courant algebroids and we keep it \textit{completely arbitrary} throughout this entire paper. 

\begin{definice}
We say that a subset $L \subseteq E$ is a \textbf{subbundle of $E$ supported on a submanifold $S \subseteq M$}, if $L$ is a vector subbundle of the restricted vector bundle $E_{S}$.
\end{definice}
\begin{rem}
We always consider only embedded submanifolds $S \subseteq M$, though not necessarily closed ones. One can be more adventurous and allow also for immersed submanifolds. However, this brings some unnecessary technical difficulties, e.g. sections of $E_{S}$ which cannot be extended to (not even local) sections of $E$. Since we do not need this generality for any of our examples, we gladly avoid this treachery.
\end{rem} 
The restricted vector bundle $E_{S}$ comes equipped with a fiber-wise metric $\<\cdot,\cdot\>$ naturally induced from $E$ and denoted by the same symbol. For any subbundle $L \subseteq E$ supported on a submanifold $S \subseteq M$, one may then construct an \textbf{orthogonal complement} $L^{\perp}$ with respect to the induced metric, hence defining a new subbundle of $E$ supported on $S$. Recall that $\rk(L^{\perp}) = \rk(E) - \rk(L)$ and there is a canonical identification $(L^{\perp})^{\perp} = L$. 

\begin{definice}
Let $L$ be a subbundle of $E$ supported on $S$. We say that 
\begin{enumerate}[(i)]
\item $L$ is \textbf{isotropic} with respect to $\<\cdot,\cdot\>$, if $L \subseteq L^{\perp}$; 
\item $L$ is \textbf{coisotropic} with respect to $\<\cdot,\cdot\>$, if $L^{\perp} \subseteq L$.
\end{enumerate}
\end{definice}
Note that these conditions are fiber-wise, that is $L$ is (co)isotropic, iff its fiber $L_{s}$ is (co)isotropic in the quadratic vector space $(E_{s}, \<\cdot,\cdot\>)$ for each $s \in S$. For any subbundle $L \subseteq E$ over $S$, let $\Gamma^{0}(L)$ denote the set of its sections isotropic with respect to the fiber-wise metric $\<\cdot,\cdot\>$ on $E_{S}$, that is 
\begin{equation}
\Gamma^{0}(L) = \{ \sigma \in \Gamma(L) \; | \; \< \sigma,\sigma\> = 0\}.
\end{equation}
Note that for general $L$, this is not a vector subspace of $\Gamma(L)$. We then have the following characterization of isotropic subbundles:
\begin{lemma} \label{lem_isosections}
A subbundle $L$ of $E$ supported on $S$ is isotropic, iff $\Gamma(L) = \Gamma^{0}(L)$. 
\end{lemma}
\begin{rem}
We often use the following extension property: Let $E$ be any vector bundle over $M$ and and let $S \subseteq M$ be any embedded submanifold. Let $\sigma \in \Gamma(E_{S})$ be a smooth section of the restricted vector bundle $E_{S}$. Then there exists an open neighborhood $U$ of $S$ in $M$ and a local section $\psi \in \Gamma_{U}(E)$, such that $\psi|_{S} = \sigma$. If $S$ is closed, one can choose $U = M$. See Exercise 10-9. in \cite{lee2012introduction}. This statement is \textit{not true} for immersed submanifolds.
\end{rem}
Now, recall that a subspace of a quadratic vector space is called maximally isotropic, if it is isotropic and not properly contained in any isotropic subspace. \begin{definice}
A subbundle $L$ of $E$ supported on $S$ is called \textbf{maximally isotropic}, if its fiber $L_{s}$ is maximally isotropic in $(E_{s},\<\cdot,\cdot\>)$ for every $s \in S$. 
\end{definice}
\begin{lemma} \label{lem_maxiso}
Let $L$ be an isotropic subbundle of $E$ supported on $S$. Then the following three statements are equivalent:
\begin{enumerate}[(i)]
\item $L$ is maximally isotropic;
\item $\rk(L) = \min \{ p,q\}$; 
\item $\Gamma(L) = \Gamma^{0}(L^{\perp})$; 
\end{enumerate}
\end{lemma}
This statement follows easily from the similar linear algebra statements, see e.g. \cite{lam2005introduction}. Note that any isotropic subbundle $L$ satisfies the inequality $\rk(L) \leq \min\{p,q\}$ and thus $\rk(L^{\perp}) \geq \max \{p,q\}$. This shows that for $p \neq q$, one \textit{cannot impose}\footnote{Technically, one can. However, there are no such subbundles.} the condition $L = L^{\perp}$. For $p=q$, this requirement is equivalent to the maximal isotropy, and such subbundles are called \textbf{Lagrangian}. We will strictly use this term just for split signatures.

Previous paragraphs in a sense establish the compatibility of subbundles with the metric $\<\cdot,\cdot\>$. Let us now focus on their interrelation with the anchor map. By $\Gamma(E;L)$, we shall denote the submodule of sections which take values in the subbundle $L$ when restricted to $S$, that is 
\begin{equation} \label{eq_gammaEL}
\Gamma(E;L) = \{ \psi \in \Gamma(E) \; | \; \psi|_{S} \in \Gamma(L) \}.
\end{equation}
If $U \subseteq M$ is an open subset, by $\Gamma_{U}(E;L)$ we denote a subset of $\Gamma_{U}(E)$ of local sections taking values in $L$ when restricted to $U \cap S$. 
\begin{definice}
Let $L$ be a subbundle of $E$ supported on $S$. We say that $L$ is \textbf{compatible with the anchor}, if $\rho(L) \subseteq TS$. 
\end{definice}
\begin{lemma}
Let $L$ be a subbundle of $E$ supported on $S$ and compatible with the anchor. Then for any $\psi \in \Gamma(E;L)$ and any $f \in C^{\infty}(M)$, one has 
\begin{equation} \label{eq_rhocomp1}
\Li{\rho^{!}(\psi|_{S})}(f|_{S}) = \Li{\rho(\psi)}(f)|_{S},
\end{equation}
where $\rho^{!}: L \rightarrow TS$ is the vector bundle map over $1_{S}$ induced by $\rho$. More generally, for any $\sigma \in \Gamma(L)$ and any $f \in C^{\infty}(M)$, one has the relation
\begin{equation} \label{eq_rhocomp2}
\Li{\rho^{!}(\sigma)}(f|_{S}) = \< \D{f}|_{S}, \sigma\>. 
\end{equation}
\end{lemma}
\begin{proof}
It suffices to prove (\ref{eq_rhocomp2}) as (\ref{eq_rhocomp1}) is obtained by setting $\sigma = \psi|_{S}$. Now, the right-hand side of (\ref{eq_rhocomp2}) can be rewritten using (\ref{eq_rhoastdef}) and the definition of $\D$ as 
\begin{equation}
\< \D{f}|_{S},\sigma \> = \< (\dr{f})|_{S}, \rho^{!}(\sigma) \>,
\end{equation}
where $\rho^{!}: E_{S} \rightarrow (TM)_{S}$ is the vector bundle map induced by $\rho$. But by assumption, we have $\rho^{!}(\sigma) \in \X(S)$. We may thus replace the restriction $(\dr{f})|_{S}$ by the pullback $i^{\ast}(\dr{f}) = \dr(f|_{S})$, where $i: S \rightarrow M$ denotes the embedding of $S$ into $M$. We can then write
\begin{equation}
\< (\dr{f})|_{S}, \rho^{!}(\sigma) \> = \<\dr(f|_{S}), \rho^{!}(\sigma) \> = \Li{\rho^{!}(\sigma)}(f|_{S}).
\end{equation}
This finishes the proof.
\end{proof}
There is a useful equivalent reformulation of the compatibility condition.
\begin{tvrz} \label{tvrz_rhocompequivalent}
Let $L$ be a subbundle of $E$ supported on $S$. Then $L$ is compatible with the anchor, iff for every $f \in C^{\infty}(M)$ such that $f|_{S} = 0$, one has $\D{f}|_{S} \in \Gamma^{0}(L^{\perp})$. 
\end{tvrz}
\begin{proof}
The only if part follows immediately from (\ref{eq_rhocomp2}), together with the fact that $\D{f} \in \Gamma^{0}(E)$ for any $f \in C^{\infty}(M)$. For the if part, let $s \in S$ be an arbitrary point and let $e \in L_{s}$. We have to prove that $\rho(e) \in T_{s}S$. We can extend $e$ to a section $\sigma \in \Gamma(L)$ and then to a section $\psi \in \Gamma_{U}(E;L)$ on some neighborhood $U$ of $S$ satisfying $\psi|_{S} = \sigma$. To prove the claim, it suffices to show that $(\Li{\rho(\psi)}(f))(s) = 0$ for any function $f \in C^{\infty}(M)$ such that $f|_{S} = 0$. But $(\Li{\rho(\psi)}(f))(s) = \< \D{f}|_{S},\sigma \>(s) = 0$, where we have used the assumption in the last step.
\end{proof}
The Leibniz rule C1) of Definition \ref{def_courant} suggests that the compatibility of $L$ with the anchor may have some implications for the bracket. This is indeed so, as shows the following proposition:
\begin{tvrz} \label{tvrz_companchcons}
Let $L$ be a subbundle of $E$ supported on $S$ and compatible with the anchor. Let $\psi \in \Gamma(E;L)$ and $\psi' \in \Gamma(E)$ be a section vanishing on $S$, that is $\psi'|_{S} = 0$. Then also $[\psi,\psi']|_{S} = 0$. Consequently, one has $[\psi',\psi]|_{S} \in \Gamma^{0}(L^{\perp})$. 
\end{tvrz}
\begin{proof}
Let $s \in S$. Pick a local frame $(\psi_{\mu})_{\mu=1}^{\rk(E)}$ for $E$ over some neighborhood $U$ of $s$. On $U$, we may write $\psi' = f^{\mu} \psi_{\mu}$. By assumption $f^{\mu}|_{S'} = 0$, where $S' = U \cap S$. Leibniz rule C1) and the consequence of the anchor compatibility (\ref{eq_rhocomp1}) gives the equation
\begin{equation}
[\psi,\psi']|_{S'} = f^{\mu}|_{S'} [\psi,\psi_{\mu}]|_{S'} + \Li{\rho^{!}(\psi|_{S'})}(f^{\mu}|_{S'}) \psi_{\mu}|_{S'} = 0.
\end{equation}
We have thus proved that $[\psi,\psi']|_{U \cap S} = 0$. As $s \in S$ was arbitrary, this proves that $[\psi,\psi']|_{S} = 0$. To prove the second claim, note that the already proved statement together with C4) imply
\begin{equation}
[\psi',\psi]|_{S} = -[\psi,\psi']|_{S} + \D\<\psi,\psi'\>|_{S} = \D \<\psi,\psi'\>|_{S}.
\end{equation}
But we have $\<\psi,\psi'\>|_{S} = \<\psi|_{S},\psi'|_{S}\> = 0$ as $\psi'|_{S} = 0$. Hence by Proposition \ref{tvrz_rhocompequivalent}, we obtain $[\psi',\psi]|_{S} \in \Gamma^{0}(L^{\perp})$. This finishes the proof.
\end{proof}
This sorts out the relation of $L$ with the anchor $\rho$. Finally, we may examine the compatibility of $L$ with the Courant algebroid bracket $[\cdot,\cdot]$. 
\begin{definice}
Let $L$ be a subbundle of $E$ supported on $S$. We say that $L$ is \textbf{involutive}, if for any $\psi,\psi' \in \Gamma(E;L)$, one has $[\psi,\psi'] \in \Gamma(E;L)$.

We say that $L$ is \textbf{locally involutive on an open subset $U \subseteq M$}, if for any $\psi,\psi' \in \Gamma_{U}(E;L)$, one has $[\psi,\psi'] \in \Gamma_{U}(E;L)$.
\end{definice}
Due to the local nature of the bracket $[\cdot,\cdot]$, one expects the global and local involutivity of $L$ to be closely related. More importantly, it suffices to verify the involutivity locally. 
\begin{lemma} \label{lem_invlocinv}
Let $L$ be a subbundle of $E$ supported on $S$. Then $L$ is involutive, if and only if it is locally involutive on every open subset $U \subseteq M$. 

Moreover, if $\{ U_{\alpha} \}_{\alpha \in I}$ is any open cover of $S$ and $L$ is locally involutive on $U_{\alpha}$ for every $\alpha \in I$, then $L$ is involutive.
\end{lemma}
\begin{proof}
If $L$ is locally involutive on every open subset $U \subseteq M$, it is involutive. Conversely, let $U \subseteq M$ be a given open set. Let $\psi,\psi' \in \Gamma_{U}(E;L)$. We must prove that $[\psi,\psi'] \in \Gamma_{U}(E;L)$. Pick an arbitrary $s \in U \cap S$ and its precompact neighborhood $V \subseteq M$ such that $\ol{V} \subseteq U$. Let $\eta \in C^{\infty}(M)$ be a bump function satisfying $\supp(\eta) \subseteq U$ and $\eta|_{\ol{V}} = 1$. Define global sections $\phi = \eta \psi$ and $\phi' = \eta \psi'$. It follows that $\phi,\phi' \in \Gamma(E;L)$. Hence by assumption, $[\phi,\phi'] \in \Gamma(E;L)$. Since $\phi|_{V} = \psi|_{V}$ and $\phi'|_{V} = \phi'|_{V}$, one has $[\psi,\psi'](s) = [\phi,\phi'](s) \in L_{s}$. As $s \in U \cap S$ was arbitrary, we have just proved that $[\psi,\psi'] \in \Gamma_{U}(E;L)$. 

To prove the second claim, suppose $L$ is locally involutive on every set $U_{\alpha}$ of an open cover $\{U_{\alpha} \}_{\alpha \in I}$ of $S$. We shall prove that $L$ is involutive. Let $\psi,\psi' \in \Gamma(E;L)$. Pick any $s \in S$. Hence $s \in U_{\alpha}$ for some $\alpha \in I$. By assumption, we have $[\psi,\psi']|_{U_{\alpha}} = [\psi|_{U_{\alpha}}, \psi'|_{U_{\alpha}}] \in \Gamma_{U_{\alpha}}(E;L)$. In particular, we have $[\psi,\psi'](s) \in L_{s}$. As $s \in S$ was arbitrary, this proves that $[\psi,\psi'] \in \Gamma(E;L)$. \end{proof}
\begin{example} \label{ex_involutive}
At this point, let us come with some trivial and not very interesting examples. 

First, for every submanifold $S \subseteq M$, the restricted vector bundle $E_{S}$ forms a subbundle of $E$ supported on $S$. One has $E^{\perp}_{S} = 0_{S}$, where $0_{S}$ denotes the image of the zero section of $E_{S}$. Hence $E_{S}$ is coisotropic. In general, it is not compatible with the anchor and obviously, it is involutive. 

On the other hand, the zero section $0_{S}$ is an isotropic subbundle of $E$ supported on $S$. One has $0_{S}^{\perp} = E_{S}$. Trivially, it is compatible with the anchor. It follows immediately from Proposition \ref{tvrz_companchcons} that $0_{S}$ is involutive. 
\end{example}
We have thus given an example of an involutive subbundle which is not compatible with the anchor. Similarly, the second example is an involutive subbundle, whose orthogonal complement is not compatible with the anchor. Interestingly, those are the only such cases. This observation appeared in \cite{li2014dirac}. We have the following statement:
\begin{tvrz} \label{tvrz_invcons}
Let $L$ be an involutive subbundle of $E$ supported on $S$.
\begin{enumerate}[(i)]
\item Suppose $L \neq E_{S}$. Then $L$ is compatible with the anchor. 
\item Suppose $L \neq 0_{S}$. Then $L^{\perp}$ is compatible with the anchor.
\item Let $\psi \in \Gamma(E;L)$ and $\phi \in \Gamma(E;L^{\perp})$. Then $[\psi,\phi] \in \Gamma(E;L^{\perp})$.
\end{enumerate}
\end{tvrz}
\begin{proof}$(i)$ Let $s \in S$ be an arbitrary point and let $e \in L_{s}$. We have to prove that $\rho(e) \in T_{s}S$. One can extend $e$ to a section $\sigma \in \Gamma(L)$ and then to a section $\psi \in \Gamma_{U}(E;L)$ on some neighborhood $U$ of $S$ satisfying $\psi|_{S} = \sigma$. To prove the claim, it suffices to show that $(\Li{\rho(\psi)}f)(s) = 0$ for every $f \in C^{\infty}(M)$ satisfying $f|_{S} = 0$. As $L \neq E_{S}$, there exists a non-zero subspace $L'_{s} \subseteq E_{s}$ such that $E_{s} = L_{s} \oplus L'_{s}$. Choose a local section $\psi' \in \Gamma_{U}(E)$ such that $\psi'(s) \in L'_{s}$ and $\psi'(s) \neq 0$. Note that $(f\psi')|_{S} = 0$, hence $f\psi' \in \Gamma_{U}(E;L)$. As $L$ is locally involutive on $U$ by Lemma \ref{lem_invlocinv}, we have $[\psi,f\psi'] \in\Gamma_{U}(E;L)$. Using the Leibniz rule C1), we obtain
\begin{equation}
[\psi,f\psi'](s) = f(s) [\psi,\psi'](s) + (\Li{\rho(\psi)}f)(s) \psi'(s) = (\Li{\rho(\psi)}f)(s) \psi'(s).
\end{equation}
The left-hand side is $L_{s}$, whereas the right-hand side is in $L'_{s}$. This implies $(\Li{\rho(\psi)}f)(s) = 0$. 

$(iii)$ Let $\psi \in \Gamma(E;L)$ and $\phi \in \Gamma(E;L^{\perp})$ be arbitrary. Let $\sigma \in \Gamma(L)$ be arbitrary. We have to show that $\< [\psi,\phi]|_{S}, \sigma \> = 0$. There is a neighborhood $U$ of $S$ and $\psi' \in \Gamma_{U}(E;L)$ satisfying $\psi'|_{S} = \sigma$. Using C3) and (\ref{eq_rhocomp1}), we obtain the equation
\begin{equation}
\Li{\rho^{!}(\psi|_{S})}( \<\phi,\psi'\>|_{S}) = \< [\psi,\phi]|_{S}, \sigma \> + \< \phi|_{S}, [\psi,\psi']|_{S} \>.
\end{equation}
But by assumption, we have $\<\phi,\psi'\>|_{S} = 0$ and $[\psi,\psi'] \in \Gamma_{U}(E;L)$, hence $\<\phi|_{S}, [\psi,\psi']|_{S}\> = 0$. 

$(ii)$ The proof is the same as in $(i)$, except that we choose $L'_{s}$ to be a non-zero complement of $L_{s}^{\perp}$ and instead of the involutivity of $L$ we use the already proved statement $(iii)$.
\end{proof}
So far, we have considered general involutive subbundles. It turns out that there are good reasons to consider only isotropic ones. 
\begin{tvrz}
Suppose $L$ is an involutive subbundle of $E$ supported on a submanifold $S$, which is not isotropic. Then for any $f \in C^{\infty}(M)$, one has $\D{f}|_{V} \in \Gamma_{V}(L)$, where $V \subseteq S$ is the open subset $V = \{ s \in S \; | \; L_{s} \text{ is not isotropic in } (E_{s},\<\cdot,\cdot\>) \}$. In other words, the restriction $L_{V}$ must contain the image $\rho^{\ast}( (T^{\ast}M)_{V})$. 
\end{tvrz}
\begin{proof}
It is easy to see that $V$ is an open subset of $S$. Let $s \in V$ be arbitrary. There is thus an element $e \in L_{s}$ with $\<e,e\> \neq 0$. Extend $e$ to a section $\sigma \in \Gamma(L)$ and find a local section $\psi \in \Gamma_{U}(E;L)$ on some neighborhood $U$ of $S$, such that $\psi|_{S} = \sigma$. Let $f \in C^{\infty}(M)$ be an arbitrary function. Then also $f\psi \in \Gamma_{U}(E;L)$. A combination of C1) and C4) yields
\begin{equation}
[f\psi,\psi](s) = f(s) [\psi,\psi](s) - (\Li{\rho(\psi)}(f)) \psi(s) + \<e,e\> \D{f}(s).
\end{equation}
As $L$ is involutive, all terms except for the last one live in $L_{s}$. As $\<e,e\> \neq 0$, we have also $\D{f}(s) \in L_{s}$. Since $s \in V$ was arbitrary, we have proved that $\D{f}|_{V} \in \Gamma_{V}(L)$. 
\end{proof}
Observations made in previous paragraphs lead us to the main definition of this section.
\begin{definice} \label{def_involutive}
Let $(E,\rho,\<\cdot,\cdot\>,[\cdot,\cdot])$ be a Courant algebroid. A subbundle $L$ of $E$ supported on $S$ is called an \textbf{almost involutive structure supported on $S$}, if
\begin{enumerate}[(i)]
\item $L$ is isotropic;
\item $L^{\perp}$ is compatible with the anchor $\rho$.
\end{enumerate}
One deletes the adjective almost when $L$ is involutive. An (almost) involutive structure $L$ is called an \textbf{(almost) Dirac structure supported on $S$}, if $L$ is maximally isotropic.
\end{definice}
\begin{rem} \label{rem_maxisoanchcomp}
Note that for maximally isotropic $L$, the assumption of the compatibility of $L^{\perp}$ with the anchor follows from a weaker assumption. Indeed, if $L$ is compatible with the anchor, every $f \in C^{\infty}(M)$ with $f|_{S} = 0$ satisfies $\D{f}|_{S} \in \Gamma^{0}(L^{\perp})$ by Proposition \ref{tvrz_rhocompequivalent}. By Lemma \ref{lem_maxiso} $(iii)$, we have $\D{f}|_{S} \in \Gamma^{0}(L)$. Proposition \ref{tvrz_rhocompequivalent} then gives the compatibility of $L^{\perp}$ with the anchor.
\end{rem}
\begin{rem}
Compare this definition to the one of a \textbf{Lie subalgebroid}. See e.g. Definition 4.3.14 of \cite{Mackenzie}. There is one very significant difference. In general, there is no bracket $[\cdot,\cdot]_{L}$ induced naturally on $\Gamma(L)$. Indeed, the obvious idea would be to define $[\sigma,\sigma']_{L} := [\psi,\psi']|_{S}$, where $\psi,\psi' \in \Gamma_{U}(E;L)$ are some local sections defined on a neighborhood $U$ of $S$ and satisfying $\psi|_{S} = \sigma$ and $\psi'|_{S} =\sigma'$. However, the second statement of Proposition \ref{tvrz_companchcons} shows that this construction can (and usually will) depend on the extension of $\sigma$.
\end{rem}
\begin{rem}
In the literature \cite{bursztyn2008courant,li2009courant,li2014dirac}, people usually considered just (almost) Dirac structures, most of the times only with the split signature (i.e. Lagrangian). As we shall demonstrate in the following, there is no real reason to consider just maximally isotropic subbundles. Note that recently, in \cite{Ikeda:2018rwe,Severa:2019ddq}, involutive structures of non-maximal dimension (supported on the entire base) were considered for generalized tangent bundle and called \textbf{small Dirac structures}.
\end{rem}

Having an almost involutive structure, it would be useful if one could verify its (local) involutivity only on some subset of the module $\Gamma_{U}(E;L)$, which can be in general quite big. This is ensured by the compatibility of $L^{\perp}$ with the anchor.
\begin{lemma} \label{lem_doesnotdepend}
Let $L$ be an almost involutive structure over $S$. Let $\psi,\psi' \in \Gamma(E;L)$ and $\phi,\phi' \in \Gamma(E;L)$ be sections such that $\psi|_{S} = \phi|_{S}$ and $\psi'|_{S} = \phi'|_{S}$. 

Then $[\psi,\psi']|_{S} \in \Gamma(L)$, iff $[\phi,\phi']|_{S} \in \Gamma(L)$.
\end{lemma}
\begin{proof}
By definition, both $L$ and $L^{\perp}$ are compatible with the anchor. The result then follows immediately from Proposition \ref{tvrz_companchcons}.
\end{proof}
\begin{tvrz} \label{tvrz_invongen}
Let $L$ be an almost involutive structure supported on $S$. 

Suppose that for every $s \in S$, there is a local frame $(\sigma_{\mu})_{\mu=1}^{\rk(L)}$ for $L$ over its neighborhood $V \subseteq S$ together with a collection $\{ \psi_{\mu} \}_{\mu=1}^{\rk(L)} \subseteq \Gamma_{U}(E)$, where $V \subseteq U$ and $\psi_{\mu}|_{V} = \sigma_{\mu}$. Moreover, assume that $[\psi_{\mu},\psi_{\nu}] \in \Gamma_{U}(E;L)$ for all $\mu,\nu \in \{1,\dots,\rk(L)\}$. 

Then $L$ is an involutive structure supported on $S$. 
\end{tvrz}
\begin{proof}
Let $\psi,\psi' \in \Gamma(E;L)$. Pick an arbitrary point $s \in S$. Let $(\sigma_{\mu})_{\mu=1}^{\rk(L)}$ be a local frame over an open set $V \subseteq S$ given by the assumption. Let $\sigma = \psi|_{V}$ and $\sigma' = \psi'|_{V}$ be the respective restrictions in $\Gamma_{V}(L)$. Then $\sigma = f^{\mu} \sigma_{\mu}$ and $\sigma' = g^{\nu} \sigma_{\nu}$ for unique functions $f^{\mu},g^{\nu} \in C^{\infty}(V)$. There exists a set $W$ open in $M$, such that $V \subseteq W \subseteq U$, together with smooth extensions $\hat{f}^{\mu}, \hat{g}^{\nu} \in C^{\infty}(W)$ of $f^{\mu}, g^{\nu}$. One may choose $W$ so that $W \cap S = V$. 

Now, define $\phi,\phi' \in \Gamma_{W}(E;L)$ by $\phi = \hat{f}^{\mu}\psi_{\mu}$ and $\phi' = \hat{g}^{\nu} \psi_{\nu}$. By construction, we have $\phi|_{V} = \sigma$ and $\phi'|_{V} = \sigma'$. Moreover, using C1) and C4), one has 
\begin{equation}
[\phi,\phi'] = \hat{f}^{\mu} \hat{g}^{\nu} [\psi_{\mu},\psi_{\nu}] + \Li{\rho(\phi)}(\hat{g}^{\nu}) \psi_{\nu} - \Li{\rho(\phi')}(\hat{f}^{\mu}) \psi_{\mu} + \<\psi_{\mu},\phi'\> \D{\hat{f}^{\mu}}.
\end{equation}
Restricting both sides to $V$, the first term on the right is in $\Gamma_{V}(L)$ by assumption, the next two due to $\psi_{\mu}|_{V} = \sigma_{\mu}$ and the last one vanishes as $\<\psi_{\mu},\phi'\>|_{V} = \<\sigma_{\mu}, \sigma' \> = 0$, since $L$ is isotropic. Thus $[\phi,\phi'] \in \Gamma_{W}(E;L)$. But $\psi|_{V} = \phi|_{V} = \sigma$ and $\psi'|_{V} = \phi'|_{V} = \sigma'$. By Lemma \ref{lem_doesnotdepend} applied to the almost involutive structure $L_{V}$  in the Courant algebroid $E_{W}$, we get $[\psi,\psi'] \in \Gamma_{W}(E;L)$. In particular, we have $[\psi,\psi'](s) \in L_{s}$. As $s \in S$ was arbitrary, this proves that $L$ is involutive. 
\end{proof}
\begin{example} \label{ex_LADorfman}
Let $(A,a,[\cdot,\cdot]_{A})$ be a Lie algebroid over $M$. Let $\dr^{A}: \Omega^{\bullet}(A) \rightarrow \Omega^{\bullet + 1}(A)$ be the corresponding coboundary operator and $\Li{}^{A}: \Omega^{\bullet}(A) \rightarrow \Omega^{\bullet}(A)$ be the Lie derivative. See e.g. \cite{mackenzie1994}. There is a canonical Courant algebroid structure on the double $E = A \oplus A^{\ast}$, where $\rho(X,\xi) = a(X)$, the fiber-wise metric $\<\cdot,\cdot\>$ is the canonical pairing of $A$ and $A^{\ast}$, and $[\cdot,\cdot]$ is the \textbf{Dorfman bracket} defined by the formula
\begin{equation}
[(X,\xi),(Y,\eta)] = ( [X,Y]_{A}, \Li{X}^{A}\eta - i_{Y}(\dr^{A}\xi)),
\end{equation}
for all $(X,\xi),(Y,\eta) \in \Gamma(A \oplus A^{\ast})$. Let $K$ be a \textbf{Lie subalgebroid} of $A$ over $S \subseteq M$, that is
\begin{enumerate}[(i)]
\item $K$ is a subbundle of $A$ supported on $S$;
\item $a(K) \subseteq TS$;
\item $[X,Y]_{A} \in \Gamma(A;K)$ for all $X,Y \in \Gamma(A;K)$.
\end{enumerate}
One can form the annihilator subbundle $\an(K) \subseteq A^{\ast}_{S}$ and let $L := K \oplus \an(K)$. We claim that $L$ is a Dirac structure in $E$ supported on $S$. 

It follows from the definition of $\an(K)$ that $L$ is isotropic. Note that $\<\cdot,\cdot\>$ has a split signature $(\rk(A),\rk(A))$. In fact, as $\rk(L) = \rk(K) + \rk(\an(K)) = \rk(A)$, we see that $L$ is maximally isotropic. Next, one has $\rho(L) = a(K) \subseteq TS$. Using Remark \ref{rem_maxisoanchcomp}, one finds that also $\rho(L^{\perp}) \subseteq TS$. We conclude that $L$ is an almost Dirac structure. It remains to prove the involutivity. The only non-trivial part is to argue that for $X,Y \in \Gamma(A;K)$ and $\xi,\eta \in \Gamma(A^{\ast};\an(K))$, one has 
\begin{equation} \Li{X}^{A}\eta - i_{Y}(\dr^{A}\xi) \in \Gamma(A^{\ast};\an(K)).
\end{equation}
For every section $Z \in \Gamma(A;K)$, one has 
\begin{equation} 
\begin{split}
\<\Li{X}^{A}\eta, Z \>|_{S} = & \ \Li{a(X)}\<\eta,Z\>|_{S} - \< \eta|_{S}, [X,Z]_{A}|_{S}\> \\
= & \ \Li{a^{!}(X|_{S})}( \<\eta,Z\>|_{S}) - \< \eta|_{S}, [X,Z]_{A}|_{S} \> = 0,
\end{split}
\end{equation}
where we have used the analogue of (\ref{eq_rhocomp1}) for the first term and the involutivity of $K$. This proves that $\Li{X}^{A}\eta \in \Gamma(A^{\ast};\an(K))$. The calculation for the second term is analogous. We conclude that $K \oplus \an(K)$ is a Dirac structure in $A \oplus A^{\ast}$.
\end{example}
We finish this part with an observation useful in the following section.
\begin{tvrz} \label{tvrz_product}
Let $(E_{1},\rho_{1},\<\cdot,\cdot\>_{1},[\cdot,\cdot]_{1})$ and $(E_{2},\rho_{2},\<\cdot,\cdot\>_{2},[\cdot,\cdot]_{2})$ be a pair of Courant algebroids. Let $L_{1} \subseteq E_{1}$ and $L_{2} \subseteq E_{2}$ be a pair of involutive structures supported on $S_{1}$ and $S_{2}$, respectively. Then $L_{1} \times L_{2}$ is an involutive structure in the product Courant algebroid $E_{1} \times E_{2}$ supported on the submanifold $S_{1} \times S_{2}$.
\end{tvrz}
\begin{proof}
Let $\psi_{1} \in \Gamma(E_{1})$ and $\psi_{2} \in \Gamma(E_{2})$. By $(\psi_{1},\psi_{2}) \in \Gamma(E_{1} \times E_{2})$ we denote the corresponding pullback section. Let $\pi_{i}: M_{1} \times M_{2} \rightarrow M_{i}$ be the projections, $i \in \{1,2\}$. The Courant algebroid structures $\rho$ and $\<\cdot,\cdot\>$ are defined on pullback sections
\begin{equation}
\rho(\psi_{1},\psi_{2}) := (\rho_{1}(\psi_{1}), \rho_{2}(\psi_{2})), \;\; \< (\psi_{1},\psi_{2}), (\psi'_{1},\psi'_{2}) \> := \<\psi_{1}, \psi'_{1}\>_{1} \circ \pi_{1} + \< \psi_{2}, \psi'_{2} \>_{2} \circ \pi_{2},
\end{equation}
for all $\psi_{1},\psi'_{1} \in \Gamma(E_{1})$ and $\psi_{2},\psi'_{2} \in \Gamma(E_{2})$. The bracket $[\cdot,\cdot]$ is given by 
\begin{equation} \label{eq_prodbracket}
[(\psi_{1},\psi_{2}), (\psi'_{1},\psi'_{2}) ] := ( [\psi_{1},\psi'_{1}]_{1}, [\psi_{2},\psi'_{2}]_{2}).
\end{equation}
On general sections, all operations are defined by $C^{\infty}$-linearity and axioms C1) and C4). It is straightforward to prove that $L_{1} \times L_{2}$ is an almost involutive structure. Consequently, we may employ Proposition \ref{tvrz_invongen} and prove the involutivity only on the sections of the form $(\psi_{1},\psi_{2})$ and $(\psi'_{1},\psi'_{2})$, where $\psi_{1},\psi'_{1} \in \Gamma(E_{1};L_{1})$ and $\psi_{2},\psi'_{2} \in \Gamma(E_{2};L_{2})$. Plugging into (\ref{eq_prodbracket}) and using the involutivity of $L_{1}$ and $L_{2}$, we obtain that $[(\psi_{1},\psi_{2}),(\psi'_{1},\psi'_{2})] \in \Gamma(E_{1} \times E_{2} ; L_{1} \times L_{2})$. Hence $L_{1} \times L_{2}$ is an involutive structure in $E_{1} \times E_{2}$ supported on $S_{1} \times S_{2}$. 
\end{proof}
\begin{rem} \label{rem_prodnotworking}
For general signatures of $E_{1}$ and $E_{2}$, this statement \textbf{does not hold} if we replace the word ``involutive'' by ``Dirac''. This was one of the biggest motivations to drop the requirement of maximal isotropy. To see the issue, let $(p_{1},q_{1})$ and $(p_{2},q_{2})$ be the signatures of $\<\cdot,\cdot\>_{1}$ and $\<\cdot,\cdot\>_{2}$, respectively. If both $L_{1}$ and $L_{2}$ are Dirac structures, we get $\rk(L_{1} \times L_{2}) = \min\{p_{1},q_{1}\} + \min\{p_{2},q_{2}\}$. The signature of $\<\cdot,\cdot\>$ is $(p_{1}+p_{2},q_{1}+q_{2})$. This shows that $L_{1} \times L_{2}$ is a Dirac structure, iff 
\begin{equation}
\min\{p_{1},q_{1}\} + \min\{p_{2},q_{2}\} = \min\{p_{1}+p_{2},q_{1}+q_{2}\}.
\end{equation}
But this is simply not true in general. Note that for split signatures, everything works fine. 
\end{rem}
\begin{rem}
Note that our definition of involutive structure in general \textit{does not} include the case $L = 0_{S}$, as $0_{S}^{\perp} = E_{S}$ is usually not compatible with the anchor. This is on purpose, to avoid some unnecessary issues. For example, Proposition \ref{tvrz_product} would not stand. Indeed, if $L_{1} \neq 0_{S_{1}}$ and $L_{2} = 0_{S_{2}}$, then $(L_{1} \times 0_{S_{2}})^{\perp} = L_{1}^{\perp} \times E_{S_{2}}$ is in general not compatible with the anchor, hence in view or Proposition \ref{tvrz_invcons}, it cannot be involutive.
\end{rem}
\section{Relations and their compositions} \label{sec_relations}
Now, let us turn our attention to the main subject of this paper. It is supposed to generalize the concept of Courant algebroid relations introduced in \cite{li2014dirac}, in particular to work well for Courant algebroids of arbitrary signatures. In view of Remark \ref{rem_prodnotworking}, we drop the condition of maximality to ensure that the Courant algebroid relations can be composed at least on the level of linear algebra. 

First, recall the following usual convention. If $(E,\rho,\<\cdot,\cdot\>,[\cdot,\cdot])$ is a Courant algebroid, by $\ol{E}$ one denotes the Courant algebroid $(E,\rho,-\<\cdot,\cdot\>,[\cdot,\cdot])$. It allows one to make the following definition.
\begin{definice} \label{def_rel}
Let $(E_{1},\rho_{1},\<\cdot,\cdot\>_{1},[\cdot,\cdot]_{1})$ and $(E_{2},\rho_{2},\<\cdot,\cdot\>_{2},[\cdot,\cdot]_{2})$ be a pair of Courant algebroids over base manifolds $M_{1}$ and $M_{2}$, respectively. 

By a \textbf{Courant algebroid relation from $E_{1}$ to $E_{2}$} we mean an involutive structure $R \subseteq E_{1} \times \ol{E}_{2}$ supported on a submanifold $S \subseteq M_{1} \times M_{2}$. We will use the notation $R: E_{1} \dra E_{2}$. 

If $S$ happens to be a graph of a smooth map $\varphi: M_{1} \rightarrow M_{2}$, $S = \gr(\varphi)$, we say that $R$ is a \textbf{Courant algebroid morphism from $E_{1}$ to $E_{2}$ over $\varphi$} and write $R: E_{1} \rat E_{2}$. 
\end{definice}
\begin{rem}
There are some remarks in order. Compared to \cite{li2014dirac}, we do not require $R$ to be maximally isotropic (or as in their case, Lagrangian). We still use the same name for the structure, though. We could have decorated it with some adjectives like ``generalized'' or ``weak'' but no real confusion should arise. 
\end{rem}

\begin{lemma}
Let $R: E_{1} \dra E_{2}$ be a Courant algebroid relation. By a \textbf{transpose relation} $R^{T}: E_{2} \dra E_{1}$, we mean $R$ viewed as a subset of $E_{2} \times \ol{E}_{1}$. 

Then $R^{T}$ is also a Courant algebroid relation. The transpose of a Courant algebroid morphism over $\varphi$ is a Courant algebroid morphism, iff $\varphi$ is a diffeomorphism. 
\end{lemma}
\begin{proof}
Let $E = E_{1} \times \ol{E}_{2}$. The product Courant algebroid $E_{2} \times \ol{E}_{1}$ can be identified with the Courant algebroid $\ol{E}$. Then $R^{T}$ is an involutive structure in $E_{2} \times \ol{E}_{1}$, iff $R$ is an involutive structure in $\ol{E}$. But that is obvious. The second statement is clear. 
\end{proof}

\begin{example} \label{ex_CArel}
Let us give some trivial examples. The interesting ones have their dedicated section. See also \cite{li2009courant} for some canonical examples of (Lagrangian) Courant algebroid relations.

(i) Let $E$ be any Courant algebroid over $M$. Let $\Delta(E) \subseteq E \times \ol{E}$ be the diagonal embedding of $E$ into the Cartesian product. It is a vector subbundle of $E \times \ol{E}$ supported on $\Delta(M)$. It is easily seen to be maximally isotropic and $\rho(\Delta(E)) = \Delta(\rho(E)) \subseteq \Delta(TM) = T(\Delta(M))$. The orthogonal complement $\Delta(E)^{\perp}$ is automatically compatible with the anchor using Remark \ref{rem_maxisoanchcomp}. Hence $\Delta(E)$ is an almost Dirac structure. 

Proposition \ref{tvrz_invongen} says that it is sufficient to verify the involutivity on sections of the form $(\psi,\psi), (\psi',\psi') $, where $\psi,\psi' \in \Gamma(E)$. But $[(\psi,\psi),(\psi',\psi')] = ([\psi,\psi']_{E},[\psi,\psi']_{E})$ is a section which takes values in $\Delta(E)$ when restricted to $\Delta(M)$. Thus $\Delta(E)$ is a Courant algebroid relation. In fact, we have $\Delta(M) = \gr(1_{M})$, hence $\Delta(E): E \rat E$ is a Courant algebroid morphism over $1_{M}$. 

(ii) Let $E$ be a Courant algebroid over $M$. Every involutive structure $L \subseteq E$ supported on $S \subseteq M$ defines a pair of Courant algebroid relations $L \times \{0\}: E \dra \{0\}$ and $\{0\} \times L: \{0\} \dra E$ supported on $S \times \{\ast\}$ and $\{\ast \} \times S$, respectively. Here $\{0\} \rightarrow \{\ast\}$ is a trivial vector bundle over the singleton equipped with the trivial Courant algebroid structure. Clearly, both these correspondences are one-to-one. This observation will allow one to define pullbacks and pushforwards of involutive structures along Courant algebroid relations. 
\end{example}
\begin{definice}
Let $R: E_{1} \dra E_{2}$ be a Courant algebroid relation supported on $S$. Let $\psi_{1} \in \Gamma(E_{1})$ and $\psi_{2} \in \Gamma(E_{2})$. We say that $\psi_{1}$ and $\psi_{2}$ are \textbf{$R$-related} and write $\psi_{1} \sim_{R} \psi_{2}$, if $(\psi_{1},\psi_{2}) \in \Gamma(E_{1} \times \ol{E}_{2}; R)$. For any $f_{1} \in C^{\infty}(M_{1})$ and $f_{2} \in C^{\infty}(M_{2})$, we write $f_{1} \sim_{S} f_{2}$, if $f_{1}(s_{1}) = f_{2}(s_{2})$ for all $(s_{1},s_{2}) \in S$. 
\end{definice}
\begin{lemma} \label{lem_relatedcons}
Let $R: E_{1} \dra E_{2}$ be a Courant algebroid relation over $S$. Let $\psi_{1},\phi_{1} \in \Gamma(E_{1})$ and $\psi_{2},\phi_{2} \in \Gamma(E_{2})$ be sections satisfying $\psi_{1} \sim_{R} \psi_{2}$ and $\phi_{1} \sim_{R} \phi_{2}$. Then
\begin{enumerate}[(i)]
\item $[\psi_{1},\phi_{1}]_{1} \sim_{R} [\psi_{2},\phi_{2}]_{2}$; 
\item $\<\psi_{1},\phi_{1}\>_{1} \sim_{S} \<\psi_{2},\phi_{2}\>_{2}$.
\end{enumerate}
\end{lemma}
\begin{proof}
Both statements follow immediately from definitions.
\end{proof}
For our future needs, let us define a concept of $R$-related covariant tensors. Note that by covariant tensors on a vector bundle $E$, we mean $C^{\infty}$-multilinear maps from $\Gamma(E)$ to $C^{\infty}(M)$, not ordinary tensor fields on the total space $E$. 
\begin{definice} \label{def_tensRrelated}
Let $t_{1} \in \T_{k}(E_{1})$ and $t_{2} \in \T_{k}(E_{2})$ be two covariant $k$-tensors on $E_{1}$ and $E_{2}$, respectively. We say that $t_{1}$ and $t_{2}$ are \textbf{$R$-related} and and write $t_{1} \sim_{R} t_{2}$, if for all $(s_{1},s_{2}) \in S$ and all $k$-tuples $(e_{1}^{(i)},e_{2}^{(i)}) \in R_{(s_{1},s_{2})}$, $i \in \{1,\dots,k\}$, one has 
\begin{equation}
(t_{1})_{s_{1}}(e_{1}^{(1)}, \dots, e_{1}^{(k)}) = (t_{2})_{s_{2}}(e_{2}^{(1)}, \dots, e_{2}^{(k)})
\end{equation}
\end{definice}
This point-wise definition can be slightly reformulated directly in terms of the product structure $E_{1} \times \ol{E}_{2}$. We leave its proof to the reader.
\begin{lemma} \label{lem_tensRrelated}
Let $t_{1} \in \T_{k}(E_{1})$ and $t_{2} \in \T_{k}(E_{2})$. Define $t \in \T_{k}(E_{1} \times \ol{E}_{2})$ to be a difference of pullbacks $t = p_{1}^{\ast}(t_{1}) - p_{2}^{\ast}(t_{2})$, where $p_{i}: E_{1} \times \ol{E}_{2} \rightarrow E_{i}$ are the projections, $i \in \{1,2\}$. 

More precisely, let $(\psi_{1}^{(1)},\dots,\psi_{1}^{(k)})$ and $(\psi_{2}^{(1)},\dots,\psi_{2}^{(k)})$ be $k$-tuples of sections of $E_{1}$ and $E_{2}$, respectively. Then $t$ is defined by the formula
\begin{equation}
t((\psi_{1}^{(1)},\psi_{2}^{(1)}), \dots, (\psi_{1}^{(k)},\psi_{2}^{(k)})) := t_{1}(\psi_{1}^{(1)},\dots,\psi_{1}^{(k)}) \circ \pi_{1} - t_{2}(\psi_{2}^{(1)},\dots,\psi_{2}^{(k)}) \circ \pi_{2},
\end{equation}
and extended by $C^{\infty}$-linearity. $\pi_{i}: M_{1} \times M_{2} \rightarrow M_{i}$ are the projections, $i \in \{1,2\}$. 

Then $t_{1} \sim_{R} t_{2}$, iff $t(\psi^{(1)},\dots,\psi^{(k)})|_{S} = 0$ for all $\psi^{(1)},\dots,\psi^{(k)} \in \Gamma(E_{1} \times \ol{E}_{2};R)$. 
\end{lemma}

\begin{rem} \label{rem_contravariant}
There is a suitable definition also for contravariant tensors. Indeed, one just replaces $R \subseteq E$ with a subbundle $R^{\dagger} \subseteq E^{\ast}$ defined as $R^{\dagger} = \mathcal{C}(\an(R))$, where $\mathcal{C}(\xi_{1},\xi_{2}) = (\xi_{1},-\xi_{2})$ for all $(\xi_{1},\xi_{2}) \in E^{\ast}$. Note that for Courant algebroids, one can be tempted to \textit{not} distinguish among covariant and contravariant tensors, since we can identify them using the fiber-wise metric. However, for a general Courant algebroid relation $R$, $t_{1} \sim_{R} t_{2}$ \textit{does not} imply $t_{1}^{\sharp} \sim_{R} t_{2}^{\sharp}$ (where $\sharp$ indicates the ``raising of all indices''). Observe that if $R$ is Lagrangian, $R = R^{\perp}$, this issue does not arise. 
\end{rem}
Naturally, the most important question is whether the Courant algebroid relations can be composed. This all is motivated by a brilliant idea \cite{weinstein1982symplectic}. See also \cite{bates1997lectures} or \cite{guillemin2013semi}. For Courant algebroids, everything what follows is just a slight modification of \cite{li2014dirac}. We add some technical details and point out differences here and there. Some of the statements are based on the personal communication with E. Meinrenken. Note that an uninterested reader may skim through definitions and jump directly to Theorem \ref{thm_composition} which states the main result of this section. 

\begin{definice} \label{def_composition}
Let $R: E_{1} \dra E_{2}$ and $R': E_{2} \dra E_{3}$ be a pair of Courant algebroid relations, where $(E_{i},\rho_{i},\<\cdot,\cdot\>_{i},[\cdot,\cdot]_{i})$ is a Courant algebroid over $M_{i}$, $i \in \{1,2,3\}$. The \textbf{composition $R' \circ R$ of $R$ and $R'$} is a subset of $E_{1} \times \ol{E}_{3}$ defined as
\begin{equation}
R' \circ R = \{ (e_{1},e_{3}) \in E_{1} \times \ol{E}_{3} \; | \; (e_{1},e_{2}) \in R \text{ and } (e_{2},e_{3}) \in R' \text{ for some $e_{2} \in E_{2}$} \}.
\end{equation}
\end{definice}
It is a well-known issue that this is not in general a smooth subbundle of $E_{1} \times \ol{E}_{3}$. In the next section, we will find explicit examples where the composition fails to be a submanifold. In fact, this is the only obstruction for it to be a vector bundle supported on a composition $S' \circ S$ (defined by the same formula). Indeed, this follows immediately from the astounding theorem in \cite{grabowski2009higher}:
\begin{theorem}[\textbf{Grabowski-Rotkiewicz}]\label{thm_GR} Let $q: E \rightarrow M$ be a vector bundle. Let $L \subseteq E$ be a subset closed under the operation $e \mapsto \lambda e$ for all $\lambda \in \R$. Then $L$ is a subbundle supported on a submanifold $S \subseteq M$, iff $L$ is a submanifold. 
\end{theorem}
$R' \circ R$ is easily seen to be closed under the scalar multiplication, hence the observation follows. However, even if $R' \circ R$ is a smooth submanifold, it is still not enough to prove that it is a Courant algebroid relation.  In the remainder of this section, let us use the shorthand notation
\begin{align}
E & = E_{1} \times \ol{E}_{2} \times E_{2} \times \ol{E}_{3}, \; \; E' = E_{1} \times \ol{E}_{3}, \\
M & = M_{1} \times M_{2} \times M_{2} \times M_{3}, \; \; M' = M_{1} \times M_{3}.
\end{align}

We will now establish the conditions under which one can prove that $R' \circ R$ is an involutive structure in $E_{1} \times \ol{E}_{3}$. First, observe that one can write $R' \circ R = p(R' \diamond R)$, where 
\begin{equation} \label{eq_RdiamondR}
R' \diamond R = (R \times R') \cap (E_{1} \times \Delta(E_{2}) \times \ol{E}_{3}) 
\end{equation}
and $p: E \rightarrow E'$ is the projection (on the first and the fourth factor of the Cartesian product). This is a subset closed under scalar multiplication, hence possibly a vector bundle supported on a subset $S' \diamond S$ defined as the intersection of $S \times S'$ and $M_{1} \times \Delta(M_{2}) \times M_{3}$. Note that $S' \circ S = \pi(S' \diamond S)$, where $\pi: M \rightarrow M'$ is the projection. We will need the following notion to proceed:
\begin{definice}
Let $S,S' \subseteq M$ be a pair of submanifolds. One says that $S$ and $S'$ \textbf{intersect cleanly in $M$}, if $S \cap S'$ is a submanifold of $M$ and for each $m \in S \cap S'$, one has
\begin{equation} \label{eq_cleancondition}
T_{m}(S \cap S') = T_{m}S \cap T_{m}S'
\end{equation}
Note that the inclusion $\subseteq$ follows automatically. 
\end{definice}
\begin{rem} \label{rem_clean}
The importance of the condition (\ref{eq_cleancondition}) lies in the following useful property. One can prove, see e.g. Proposition C.3.1 in \cite{hormander2007analysis}, that if $S$ and $S'$ intersect cleanly in $M$, then for each $m \in S \cap S'$, there is a coordinate chart $(U,\phi)$ around $m$, such that $\phi(U \cap S) = \phi(U) \cap V$ and $\phi(U \cap S') = \phi(U) \cap V'$ for a pair of linear subspaces $V,V' \subseteq \R^{n}$. In particular, one has $\phi(U \cap (S \cap S')) = \phi(U) \cap (V \cap V')$, that is $S$ and $S'$ look locally as a pair of intersecting vector subspaces. Note that if $S$ and $S'$ are transverse submanifolds, that is $T_{m}S + T_{m}S' = T_{m}M$ for all $m \in S \cap S'$, they automatically intersect cleanly in $M$.
\end{rem}
Now, one can impose two equivalent conditions on the relations $R$ and $R'$ to ensure that $R' \diamond R$ is a vector subbundle of $E$ supported on $S' \diamond S$. 
\begin{tvrz} \label{tvrz_RdiamondR}
Let $R: E_{1} \dra E_{2}$ and $R': E_{2} \dra E_{3}$ be a pair of Courant algebroid relations over $S$ and $S'$, respectively. Then the following two conditions are equivalent:
\begin{enumerate}[(i)]
\item $R \times R'$ and $E_{1} \times \Delta(E_{2}) \times \ol{E}_{3}$ intersect cleanly in $E$;
\item $S \times S'$ and $M_{1} \times \Delta(M_{2}) \times M_{3}$ intersect cleanly in $M$ and the dimension of the vector subspace $(R' \diamond R)_{m}$ is the same for all $m \in S' \diamond S$.
\end{enumerate}
Both these conditions ensure that $R' \diamond R$ is a subbundle of $E$ over a submanifold $S' \diamond S$. 
\end{tvrz} 
\begin{proof}
First, the condition $(i)$ ensures that $R' \diamond R$ is a submanifold of $E$. It is closed under the scalar multiplication, hence by Theorem \ref{thm_GR}, it is a subbundle of $E$ supported on $S' \diamond S$ (which is automatically a submanifold of $M$). Let us show that it also implies $(ii)$. Clearly, the dimension of $(R' \diamond R)_{m}$ is the same for all $m \in S' \diamond S$ and the intersection of $S \times S'$ and $M_{1} \times \Delta(M_{2}) \times M_{3}$ is a submanifold. It thus remains to prove that the condition (\ref{eq_cleancondition}) holds. One only has to prove the inclusion $\supseteq$. If $x \in T_{m}(S \times S') \cap T_{m}(M_{1} \times \Delta(M_{2}) \times M_{3})$ for $m \in S' \diamond S$, it can be viewed as a vector tangent both to $R \times R'$ and $E_{1} \times \Delta(E_{2}) \times \ol{E}_{3}$ at $0_{m}$ (if we identify the base $M$ with the image of the zero section). Hence by assumption, it is tangent to $R' \diamond R$ and consequently to $S' \diamond S$. This shows that $S \times S'$ and $M_{1} \times \Delta(M_{2}) \times M_{3}$ intersect cleanly. 

Conversely, suppose that $(ii)$ holds. In particular, $S' \diamond S$ is a submanifold of $M$. We can write $R' \diamond R = L \cap L' \subseteq E_{S' \diamond S}$, where $L = (R \times R')_{S' \diamond S}$ and $L' = (E_{1} \times \Delta(E_{2}) \times \ol{E}_{3})_{S' \diamond S}$ are the respective restricted vector bundles over $S' \diamond S$.  Now, $L$ and $L'$ are subbundles over the same base and it is a well-known fact that their intersection is a subbundle, iff the dimension of $(L \cap L')_{m}$ is constant for all $m \in S' \diamond S$. Hence $R' \diamond R$ is a subbundle of $E$ supported on $S' \diamond S$. In particular, it is a submanifold of $E$. To prove $(i)$, we only have to show that (\ref{eq_cleancondition}) holds.

It is not difficult to show that $L$ and $L'$ intersect cleanly in $E_{S' \circ S}$. One can now argue that as $S \times S'$ and $M_{1} \times \Delta(M_{2}) \times M_{3}$ intersect cleanly in $M$, there holds the equality
\begin{equation}
T_{e}(L) \cap T_{e}(L') = T_{e}(R \times R') \cap T_{e}(E_{1} \times \Delta(E_{2}) \times \ol{E}_{3}),
\end{equation}
for all $e \in R' \diamond R$. This shows that $R \times R'$ and $E_{1} \times \Delta(E_{2}) \times \ol{E}_{3}$ intersect cleanly in $E$.
\end{proof}
By imposing one of the equivalent conditions $(i)$ and $(ii)$ of Proposition \ref{tvrz_RdiamondR}, we ensure that $R' \circ R$ is the image of the vector subbundle $R' \diamond R$ under the vector bundle map $p: E \rightarrow E'$ over $\pi: M \rightarrow M'$. Similarly to the previous proposition, we will now impose two equivalent conditions making it into a subbundle of $E'$ supported on $S' \circ S$. 
\begin{tvrz} \label{tvrz_RcircR}
Let $R: E_{1} \dra E_{2}$ and $R': E_{2} \dra E_{3}$ be a pair of Courant algebroid relations over $S$ and $S'$, respectively. Suppose they satisfy the conditions in Proposition \ref{tvrz_RdiamondR}. Then the following conditions are equivalent:
\begin{enumerate}[(i)]
\item $R' \circ R$ is a submanifold of $E'$, such that the induced map $p: R' \diamond R \rightarrow R' \circ R$ becomes a smooth surjective submersion;
\item $S' \circ S$ is a submanifold of $M'$, such that the induced map $\pi: S' \diamond S \rightarrow S' \circ S$ becomes a smooth surjective submersion.  The rank of the linear map $p_{m}: (R' \diamond R)_{m} \rightarrow (R' \circ R)_{\pi(m)}$ is the same for all $m \in S' \diamond S$. 
\end{enumerate}
Both these conditions ensure that $R' \circ R$ is a subbundle of $E'$ supported on $S' \circ S$ and $p: R' \diamond R \rightarrow R' \circ R$ becomes a fiber-wise surjective vector bundle map over $\pi: S' \diamond S \rightarrow S' \circ S$. If any of the two equivalent conditions occurs, we say that \textbf{$R$ and $R'$ compose cleanly}.
\end{tvrz}
\begin{proof}
First, assume that $(i)$ holds. As $R' \circ R$ is closed under the scalar multiplication, by Theorem \ref{thm_GR} it is a vector bundle supported on a submanifold $S' \circ S$. The map $p: R' \diamond R \rightarrow R' \circ R$ is easily seen to be a fiber-wise surjective vector bundle map over $\pi: S' \diamond S \rightarrow S' \circ S$. This already implies that the base map is a submersion and the induced linear map $p_{m}: (R' \diamond R) \rightarrow (R' \circ R)_{\pi(m)}$ has the same rank for all $m \in S' \diamond S$. Hence $(ii)$ follows.

Conversely, if $(ii)$ holds, we have the induced smooth map $p: R' \diamond R \rightarrow E'_{S' \circ S}$ which can be viewed as a vector bundle map over the surjective submersion $\pi: S' \diamond S \rightarrow S' \circ S$. By assumption, it has a constant rank. Using the local section property of $\pi$ and the  standard statement for vector bundle maps over the identity (see e.g. Theorem 10.34 of \cite{lee2012introduction}), it follows that $R' \circ R = p(R' \diamond R)$ is a subbundle of $E'$ supported on $S' \circ S$, and $p: R' \diamond R \rightarrow R' \circ R$ becomes a fiber-wise surjective vector bundle map over a surjective submersion $\pi$. Such a map is always a surjective submersion and the condition $(i)$ follows. This finishes the proof. 
\end{proof}
\begin{rem}
Note that in \cite{li2014dirac}, their cleanly composing relations satisfy slightly weaker conditions. However, it will bring one to the realm of non-injectively immersed non-Hausdorff submanifolds where the night is long and full of (t)errors. We do not intend to go there.
\end{rem}
\begin{example} \label{ex_CAmorphcomp} Let $R: E_{1} \rat E_{2}$ and $R': E_{2} \rat E_{3}$ be a pair of Courant algebroid morphisms over $\varphi$ and $\varphi'$, respectively. It is easy to see that 
\begin{equation}
S' \diamond S = \{ (m_{1}, \varphi(m_{1}), \varphi(m_{1}), (\varphi' \circ \varphi)(m_{1})) \; | \; m_{1} \in M_{1} \}.
\end{equation}
This is a submanifold diffeomorphic to $M_{1}$ and one can quickly argue that (\ref{eq_cleancondition}) holds, hence $S \times S'$ and $M_{1} \times \Delta(M_{2}) \times M_{3}$ intersect cleanly in $M$. Moreover, one has $S' \circ S = \gr(\varphi' \circ \varphi)$ and $\pi: S' \diamond S \rightarrow S' \circ S$ is a diffeomorphism. This shows that Courant algebroid morphisms always meet the conditions on $S$ and $S'$ in Proposition \ref{tvrz_RdiamondR} $(ii)$ and Proposition \ref{tvrz_RcircR} $(ii)$. Consequently, to verify that $R$ and $R'$ compose cleanly, one only has to verify the constant rank requirements of both propositions.
\end{example}
One can now formulate the main theorem regarding the compositions of relations. We will then need some technical lemmas in order to prove it. 
\begin{theorem} \label{thm_composition}
Let $R: E_{1} \dra E_{2}$ and $R': E_{2} \dra E_{3}$ be a pair of relations over $S$ and $S'$, respectively. Suppose $R$ and $R'$ compose cleanly.

Then $R' \circ R$ is an involutive structure supported on $S' \circ S$, hence defines a Courant algebroid relation $R' \circ R: E_{1} \dra E_{3}$. 
\end{theorem}
\begin{lemma} \label{lem_cleantwosec}
Let $q: E \rightarrow M$ be a vector bundle, and let $S$ and $S'$ be two submanifolds which intersect cleanly in $M$. Suppose we have a pair of sections $\sigma \in \Gamma(E_{S})$ and $\sigma' \in \Gamma(E_{S'})$ which coincide on the intersection, i.e. $\sigma|_{S \cap S'} = \sigma'|_{S \cap S'}$.

Then there exists an open set $U \subseteq M$ containing $S \cap S'$, together with a local section $\psi \in \Gamma_{U}(E)$ which extends both $\sigma$ and $\sigma'$, that is $\psi|_{U \cap S} = \sigma|_{U \cap S}$ and $\psi'|_{U \cap S'} = \sigma'|_{U \cap S'}$. 

If both $S$ and $S'$ are closed, one may choose $U = M$. 
\end{lemma}
\begin{proof}
For each $m \in S \cap S'$, let $\mu: W \times \R^{k} \rightarrow q^{-1}(W)$ be a local trivialization chart for $E$, where $W$ a neighborhood of $m$ and $k = \rk(E)$. Manifolds $W \cap S$ and $W \cap S'$ intersect cleanly in $W$, so we may find an open neighborhood $U_{(m)} \subseteq W$ of $m$ together with a chart $\phi: U_{(m)} \rightarrow \R^{n}$ described in Remark \ref{rem_clean}, where $n = \dim(M)$. Hence $\phi(U_{(m)} \cap S) = \phi(U_{(m)}) \cap V$ and $\phi(U_{(m)} \cap S') = \phi(U_{(m)}) \cap V'$ for a pair of subspaces $V,V' \subseteq \R^{n}$. The sections $\sigma$ and $\sigma'$ then define a pair of smooth maps $\hat{\sigma}: \phi(U_{(m)}) \cap V \rightarrow \R^{k}$ and $\hat{\sigma}': \phi(U_{(m)}) \cap V' \rightarrow \R^{k}$, such that $\sigma(x) = \mu(x, \hat{\sigma}(\phi(x)))$ for all $x \in U_{(m)} \cap S$. The map $\hat{\sigma}'$ is defined similarly. 	

These two maps define a single smooth map $\hat{\sigma} \cup \hat{\sigma}': \phi(U_{(m)}) \cap (V \cup V') \rightarrow \R^{k}$ as they coincide on the intersection $\phi(U_{(m)}) \cap (V \cap V')$. But $\phi(U_{(m)}) \cap (V \cup V')$ is a closed submanifold of $\phi(U_{(m)})$ and there is thus a smooth extension of $\hat{\sigma} \cup \hat{\sigma}'$ to a smooth map $\hat{\chi}: \phi(U_{(m)}) \rightarrow \R^{k}$. It can be used to define a smooth local section $\psi_{(m)} \in \Gamma_{U_{(m)}}(E)$ via the formula $\psi_{(m)}(x) = \mu(x, \hat{\chi}(\phi(x)))$ for all $x \in U_{(m)}$. By construction, $\psi_{(m)}$ coincides with $\sigma$ on $U_{(m)} \cap S$ and with $\sigma'$ on $U_{(m)} \cap S'$. 

Set $U = \bigcup_{m \in S \cap S'} U_{(m)}$ and let $\{ \rho_{(m)} \}_{m \in S \cap S'}$ be a partition of unity subordinate to the open cover $\U = \{ U_{(m)} \}_{m \in S \cap S'}$ of $U$. Define $\psi \in \Gamma_{U}(E)$ by the usual formula $\psi = \sum_{m \in S \cap S'} \rho_{(m)} \psi_{(m)}$. Then $\psi|_{U \cap S} = \sigma|_{U \cap S}$ and $\psi'|_{U \cap S'} = \sigma'|_{U \cap S'}$. This finishes the main part of the proof. 

Now, if $S$ and $S'$ are closed, we may add two more open sets to $\U$, namely $U_{1} = M - S$ and $U_{2} = M - S'$ and consider two more local sections $\psi_{1} \in \Gamma_{U_{1}}(E)$ and $\psi_{2} \in \Gamma_{U_{2}}(E)$, constructed as follows. The set $U_{1} \cap S'$ is a closed embedded submanifold of $U_{1}$, and we have a local section $\sigma'|_{U_{1} \cap S'} \in \Gamma(E_{U_{1} \cap S'})$ Let $\psi_{1} \in \Gamma_{U_{1}}(E)$ be its extension to $U_{1}$. Similarly, $\psi_{2}$ is an extension of $\sigma|_{U_{2} \cap S}$ to $U_{2}$. Let $\{ \rho_{(m)} \}_{m \in S \cap S'} \cup \{ \rho_{1}, \rho_{2} \}$ be a partition of unity subordinate to the open cover $\U \cup \{ U_{1},U_{2} \}$ of $M$ and set $\psi = \sum_{m \in S \cap S'} \rho_{(m)} \rho_{(m)} + \rho_{1} \psi_{1} + \rho_{2}\psi_{2}$. This is a global section of $E$ which restricts to $\sigma$ on $S$ and to $\sigma'$ on $S'$. 
\end{proof}
\begin{lemma} \label{lem_Qrelation}
Consider the subset $Q \subseteq E \times \ol{E}'$ defined as 
\begin{equation}
Q = \{ (e, p(e)) \; | \; e \in E_{1} \times \Delta(E_{2}) \times \ol{E}_{3} \}.
\end{equation}
Then $Q$ defines a Courant algebroid relation $Q: E \dra E'$ supported on a submanifold $P$, where 
\begin{equation}
P = \{ (m,\pi(m)) \; | \; m \in M_{1} \times \Delta(M_{2}) \times M_{3} \}.
\end{equation}
Moreover, for any $\psi \in \Gamma(E)$ and $\psi' \in \Gamma(E')$, one has $\psi \sim_{Q} \psi'$, if and only if $\psi - p^{\ast}(\psi') \in \Gamma(E; 0_{M_{1}} \times \Delta(E_{2}) \times 0_{M_{3}})$, where $p^{\ast}(\psi')$ is the obvious pullback of $\psi'$ to a section of $E$. 
\end{lemma}
\begin{proof}
Examining $Q$ a little bit closer, one notices that it is a subbundle $\Delta(E_{1} \times E_{2} \times E_{3})$ in disguise. Hence by Example \ref{ex_CArel} (i), it is a Dirac structure supported on $P$. This proves that it defines a Courant algebroid relation $Q: E \dra E'$. The statement about $Q$-related sections follows immediately from the definition.
\end{proof}
Now comes the crucial part of the proof. Here one uses the assumption that $S \times S'$ and $M_{1} \times \Delta(M_{2}) \times M_{3}$ intersect cleanly. To simplify the notation, let us for the purpose of next statements write $C = E_{1} \times \Delta(E_{2}) \times \ol{E}_{3}$, $C^{\perp} = 0_{M_{1}} \times \Delta(M_{2}) \times 0_{M_{3}}$ and $A = M_{1} \times \Delta(M_{2}) \times M_{3}$. Let us emphasize once more that the next lemma is a more detailed explanation of the statement used without proof in \cite{li2014dirac} and its proof is based on e-mails exchanged with E. Meinrenken.
\begin{lemma} \label{lem_lift}
Let $R: E_{1} \dra E_{2}$ and $R': E_{2} \dra E_{3}$ be a pair of Courant algebroid relations over $S$ and $S'$, respectively. Assume that $R$ and $R'$ compose cleanly. 

Let $\psi' \in \Gamma(E', R' \circ R)$. Then there is an open set $U \subseteq M$, such that $S' \diamond S \subseteq U$, together with a local section $\psi \in \Gamma_{U}(E; R \times R')$ satisfying $\psi \sim_{Q} \psi'$. 
\end{lemma}
\begin{proof}
First, consider the restriction $\psi'|_{S' \circ S} \in \Gamma(R' \circ R)$. As $p: R' \diamond R \rightarrow R' \circ R$ is a fiber-wise surjective vector bundle map, there is a section $\sigma \in \Gamma(R' \diamond R)$, such that $p(\sigma(m)) = \psi'(\pi(m))$ for all $m \in S' \diamond S$. In particular, it follows that $\sigma - p^{\ast}(\psi')|_{S' \diamond S}$ takes values in the subbundle $C^{\perp}$. 

As $S' \diamond S$ is a submanifold of $A$, there is an open set $V \subseteq M$ containing $S' \diamond S$ and a local section $\sigma' \in \Gamma_{V \cap A}(C^{\perp})$ such that $\sigma'|_{S' \diamond S} = \sigma - p^{\ast}(\psi')|_{S' \diamond S}$. Next, consider the section
\begin{equation}
\sigma'' := \sigma' + p^{\ast}(\psi')|_{V \cap A}.
\end{equation}
It follows that $\sigma'' \in \Gamma_{V \cap A}(C)$ is a local section, such that $\sigma'' - p^{\ast}(\psi')|_{V \cap A} \in \Gamma_{V \cap A}(C^{\perp})$. Moreover, by construction, $\sigma''|_{S' \diamond S} = \sigma$.

On the other hand, there is an open set $W \subseteq M$ containing $S \diamond S'$, together with a local section $\tau \in \Gamma_{W \cap (S \times S')}(R \times R')$ which satisfies $\tau|_{S' \diamond S} = \sigma$. We may choose $W = V$ by taking their intersection if necessary.

We thus have a pair of sections $\sigma'' \in \Gamma_{V \cap A}(E_{V})$ and $\tau \in \Gamma_{V \cap (S \times S')}(E_{V})$, such that $\sigma''|_{S' \diamond S} = \tau|_{S' \diamond S}$. By assumption $V \cap A$ and $V \cap (S \times S')$ intersect cleanly in $V$. Hence by Lemma \ref{lem_cleantwosec}, there is an open subset $U \subseteq V$ containing $S' \diamond S$, together with a local section $\psi \in \Gamma_{U}(E)$ extending both $\sigma''$ and $\tau$. It follows that $\psi$ is the section we were looking for.
\end{proof}
\begin{lemma} \label{lem_RcircROG}
Let $R: E_{1} \dra E_{2}$ and $R': E_{2} \dra E_{3}$ be a pair of Courant algebroid relations over $S$ and $S'$, respectively. Assume that $R$ and $R'$ compose cleanly.

Then $R'^{\perp} \diamond R^{\perp} := (R^{\perp} \times R'^{\perp}) \cap C$ is a subbundle of $E$ over $S' \diamond S$ and $p(R'^{\perp} \diamond R^{\perp}) = (R' \circ R)^{\perp}$. 
\end{lemma}
\begin{proof}
The fact that $p(R'^{\perp} \diamond R^{\perp}) = (R' \circ R)^{\perp}$ follows from linear algebra, see Proposition \ref{tvrz_coiso} and Example \ref{ex_coisored}. To prove that $R'^{\perp} \diamond R^{\perp}$ is a subbundle over $S' \diamond S$, see the proof of Proposition \ref{tvrz_RdiamondR}, it suffices to prove that the dimension of $(R'^{\perp} \diamond R^{\perp})_{m}$ is the same for all $m \in S' \diamond S$. We have
\begin{equation}
\dim((R'^{\perp} \diamond R^{\perp})_{m}) =  \rk((R' \circ R)^{\perp}) + \dim( (R^{\perp} \times R'^{\perp})_{m} \cap C^{\perp}_{m}), 
\end{equation}
using the nullity-rank theorem and the fact that $(R' \circ R)^{\perp}$ is a subbundle. Next, one has 
\begin{equation}
\begin{split}
\dim((R^{\perp} \times R'^{\perp})_{m} \cap C^{\perp}_{m}) = & \ \rk(E) - \dim((R \times R')_{m} + C_{m}) \\
= & \ \rk(E) - \rk(R \times R') - \rk(C) + \dim((R \times R')_{m} \cap C_{m}) \\
= & \ \rk(E) - \rk(R \times R') - \rk(C) + \rk(R' \diamond R).
\end{split}
\end{equation}
The right-hand does not depend on $m$ and the proof is finished. 
\end{proof}
\begin{rem}
We see that if $R$ and $R'$ compose cleanly, then so do $R^{\perp}$ and $R'^{\perp}$ (although, strictly speaking, they are not involutive structures). In particular, one has $R'^{\perp} \circ R^{\perp} = (R' \circ R)^{\perp}$. 
\end{rem}
\begin{proof}[\textbf{Proof of Theorem \ref{thm_composition}}] It is easy to see that $R' \circ R$ is an isotropic subbundle.

Next, let us prove the involutivity of $(R' \circ R)^{\perp}$ with the anchor using Lemma \ref{lem_RcircROG}. Let $(e_{1},e_{3}) \in (R' \circ R)^{\perp}$. There is thus some $e_{2} \in E_{2}$, such that $e := (e_{1},e_{2},e_{2},e_{3}) \in (R^{\perp} \times R'^{\perp}) \cap C$. As $R$ and $R'$ are involutive structures, we have $\rho(R^{\perp} \times R'^{\perp}) \subseteq T(S \times S')$. The subbundle $C$ is also compatible with the anchor, $\rho(C) \subseteq TA$. As $S \times S'$ and $A$ intersect cleanly in $M$, we have $\rho(e) \in T(S' \diamond S)$. Since $S' \circ S = \pi(S' \diamond S)$, we have $T(\pi)(T(S' \diamond S)) \subseteq T(S' \circ S)$. Finally, one gets $\rho'(e_{1},e_{3}) = T(\pi)( \rho(e)) \in T(S' \circ S)$. This proves that $(R' \circ R)^{\perp}$ is compatible with the anchor.

Now, let us prove the involutivity. Let $\psi',\phi' \in \Gamma(E';R' \circ R)$. We have to show that $[\psi',\phi'] \in \Gamma(E'; R' \circ R)$. Let $\psi,\phi \in \Gamma_{U}(E;R \times R')$ be the corresponding sections provided by Lemma \ref{lem_lift}. It follows from Proposition \ref{tvrz_product} that the subbundle $R \times R'$ is involutive. Hence $[\psi,\phi] \in \Gamma_{U}(E;R \times R')$. As $Q: E \dra E'$ is a Courant algebroid relation by Lemma \ref{lem_Qrelation}, we get $[\psi,\phi] \sim_{Q} [\psi',\phi']$ by Lemma \ref{lem_relatedcons}. Note that this automatically implies that $[\psi,\phi] \in \Gamma_{U}(E;C)$. Consequently, $[\psi,\phi] \in \Gamma_{U}(E;R' \diamond R)$. Finally, this implies that $[\psi',\phi'] \in \Gamma(E; R' \circ R)$. Hence $R' \circ R$ is involutive and we conclude that it defines a Courant algebroid relation $R' \circ R: E_{1} \dra E_{3}$ supported on $S' \circ S$. 
\end{proof}
\begin{rem} \label{rem_compDiracfail}
Note that if $R$ and $R'$ are Dirac structures, the composition $R' \circ R$ may fail to be maximally isotropic. This is because the product $R \times R'$ is usually not maximally isotropic, see Remark \ref{rem_prodnotworking}. See also Appendix \ref{sec_supplement} for a more detailed discussion. 
\end{rem}
\begin{tvrz} \label{tvrz_CArelCAT}
The composition of relations is associative. Suppose $R: E_{1} \dra E_{2}$, $R': E_{2} \dra E_{3}$ and $R'': E_{3} \dra E_{4}$ is a triple of Courant algebroid relations. Then 
\begin{equation} \label{eq_ass}
R'' \circ (R' \circ R) = (R''\circ R') \circ R.
\end{equation}
This works already on the set level.

Moreover, let $\Delta(E_{1}): E_{1} \dra E_{1}$ and $\Delta(E_{2}): E_{2} \dra E_{2}$ be the relations from Example \ref{ex_CArel} (i). Then both pairs ($\Delta(E_{1})$,$R)$ and $(R,\Delta(E_{2}))$ compose cleanly and 
\begin{equation} \label{eq_identity}
R \circ \Delta(E_{1}) = \Delta(E_{2}) \circ R = R. 
\end{equation}
\end{tvrz}
\begin{proof}
Equations (\ref{eq_ass}) and (\ref{eq_identity}) can be proved easily from the definition of $\circ$. Let us prove that $\Delta(E_{1})$ and $R$ compose cleanly, the other statement is analogous. First, we have
\begin{equation}
R \diamond \Delta(E_{1}) = \{ (e_{1},e_{1},e_{1},e_{2}) \; | \; (e_{1},e_{2}) \in R \}.
\end{equation}
This set can be identified with a graph of a smooth map $(e_{1},e_{2}) \mapsto (e_{1},e_{1})$ from $R$ to $E_{1} \times \ol{E}_{1}$, which is an embedded submanifold of $(E_{1} \times \ol{E}_{1}) \times R$, hence of $E_{1} \times \ol{E}_{1} \times E_{1} \times \ol{E}_{2}$. It remains to prove the equation (\ref{eq_cleancondition}) which is straightforward. This establishes the condition $(i)$ of Proposition \ref{tvrz_RdiamondR}. Now, the restriction $p: R \diamond \Delta(E_{1}) \rightarrow R \circ \Delta(E_{1}) \equiv R$ maps $(e_{1},e_{1},e_{1},e_{2})$ to $(e_{1},e_{2})$, hence in fact, it defines a diffeomorphism. This proves the condition $(i)$ of  Proposition \ref{tvrz_RcircR}. Hence $\Delta(E_{1})$ and $R$ compose cleanly. This finishes the proof.
\end{proof}
\begin{rem}
This proposition shows that Courant algebroids together with their relations form a ``category'' where not all morphisms may be composed. Let us denote it as $\ol{\cCA}$. Note that we will find an actual category of Courant algebroids in the following section.
\end{rem}
\begin{rem}
Compositions of relations are fully compatible with Definition \ref{def_tensRrelated}. Indeed, if $t_{1} \sim_{R} t_{2}$ and $t_{2} \sim_{R'} t_{3}$ for $t_{i} \in \T_{k}(E_{i})$, $i \in \{1,2,3\}$, it is easy to see that $t_{1} \sim_{R' \circ R} t_{3}$. 

For contravariant tensors (see Remark \ref{rem_contravariant}), this requires one to prove that whenever $R$ and $R'$ compose cleanly, then so do $R^{\dagger}$ and $R'^{\dagger}$ and $(R' \circ R)^{\dagger} = R'^{\dagger} \circ R^{\dagger}$. The rest is then analogous to the covariant case. 
\end{rem}
To conclude this section, let us elaborate a little bit on pullbacks and pushforwards of involutive structures using the viewpoint of Example \ref{ex_CArel} (ii).
\begin{definice} \label{def_pushpull}
Let $R: E_{1} \dra E_{2}$ be a Courant algebroid relation supported on $S$. 
\begin{enumerate}[(i)]
\item Let $L_{2} \subseteq E_{2}$ be an involutive structure supported on $S_{2} \subseteq M_{2}$. Suppose that the relations $R: E_{1} \dra E_{2}$ and $L_{2} \times \{0\}: E_{2} \dra \{0\}$ compose cleanly. Then the \textbf{pullback involutive structure} $R^{\ast}(L_{2}) \subseteq E_{1}$ is uniquely determined by the formula
\begin{equation}
R^{\ast}(L_{2}) \times \{0\} = (L_{2} \times \{0\}) \circ R.
\end{equation}
\item Let $L_{1} \subseteq E_{1}$ be an involutive structure supported on $S_{1} \subseteq M_{1}$. Suppose that the relations $\{0\} \times L_{1}: \{0\} \dra E_{1}$ and $R: E_{1} \dra E_{2}$ compose cleanly. Then the \textbf{pushforward involutive structure} $R_{\ast}(L_{1}) \subseteq E_{2}$ is uniquely determined by the formula
\begin{equation}
\{0\} \times R_{\ast}(L_{1}) = R \circ (\{0\} \times L_{1}).
\end{equation}
\end{enumerate}
\end{definice}
Note that the resulting subbundles define involutive structures due to Theorem \ref{thm_composition}. 
\section{Examples and applications} \label{sec_examples}
In this section, we come up with a couple of prominent examples. We discuss when they fail to be Courant algebroid relations in the sense of \cite{li2014dirac}. See also \cite{li2009courant} for additional cases obtained via certain canonical constructions of new Courant algebroids.
\subsection{Graph of a bundle map} \label{subsec_graph}
Let $(E_{1},\rho_{1},\<\cdot,\cdot\>_{1},[\cdot,\cdot]_{1})$ and $(E_{2},\rho_{2},\<\cdot,\cdot\>_{2},[\cdot,\cdot]_{2})$ be a pair of Courant algebroids over $M_{1}$ and $M_{2}$, respectively. Let $\F: E_{1} \rightarrow E_{2}$ be a vector bundle map over $\varphi: M_{1} \rightarrow M_{2}$. Then its graph $\gr(\F) \subseteq E_{1} \times \ol{E}_{2}$ is a subbundle supported on $\gr(\varphi) \subseteq M_{1} \times M_{2}$. Let us now analyze the conditions making $\gr(\F): E_{1} \rat E_{2}$ into a Courant algebroid morphism over $\varphi$. 

First, note that for $\psi_{1} \in \Gamma(E_{1})$ and $\psi_{2} \in \Gamma(E_{2})$, we have $\psi_{1} \sim_{\gr(\F)} \psi_{2}$, if for all $m_{1} \in M_{1}$, one has
$\F(\psi_{1}(m_{1})) = \psi_{2}(\varphi(m_{1}))$. Such sections are usually called $\F$-related and one writes $\psi_{1} \sim_{\F} \psi_{2}$. One can also write this as a commutativity of the diagram
\begin{equation}
\begin{tikzcd}
E_{1} \arrow{r}{\F} & E_{2} \\
M_{1} \arrow{u}{\psi_{1}} \arrow{r}{\varphi} & M_{2} \arrow{u}{\psi_{2}}
\end{tikzcd}.
\end{equation}
For $f_{1} \in C^{\infty}(M_{1})$ and $f_{2} \in C^{\infty}(M_{2})$, we have $f_{1} \sim_{\gr(\varphi)} f_{2}$, iff $f_{1} = f_{2} \circ \varphi$. Now, let us turn our attention to the isotropy condition. 
\begin{lemma} \label{lem_grF1}
Let $\F: E_{1} \rightarrow E_{2}$ be a vector bundle map over $\varphi: M_{1} \rightarrow M_{2}$. Then the subbundle $\gr(\F)$ is isotropic, iff for all $m_{1} \in M_{1}$ and all $e_{1},e'_{1} \in (E_{1})_{m_{1}}$, one has 
\begin{equation} \label{eq_Fisotropic}
\< \F(e_{1}), \F(e'_{1}) \>_{2} = \<e_{1},e'_{1}\>_{1}. 
\end{equation}
In particular, $\F$ has to be fiber-wise injective. Let $(p_{1},q_{1})$ and $(p_{2},q_{2})$ be the signature of $\<\cdot,\cdot\>_{1}$ and $\<\cdot,\cdot\>_{2}$, respectively. Then $\gr(\F)$ is maximally isotropic, iff either $p_{1} = p_{2}$ or $q_{1} = q_{2}$. 
\end{lemma}
\begin{proof}
It is easy easy to see that (\ref{eq_Fisotropic}) is equivalent to the isotropy of $\gr(\F)$ in $E_{1} \times \ol{E}_{2}$. Note that this is the main motivation for the sign flip of $\ol{E}_{2}$. Any $\F$ satisfying (\ref{eq_Fisotropic}) must be fiber-wise injective. Indeed, if $\F(e_{1}) = 0$, we find $\<e_{1},e'_{1}\>_{1} = 0$ for all $e'_{1} \in (E_{1})_{m_{1}}$. As $\<\cdot,\cdot\>_{1}$ is non-degenerate, this forces $e_{1} = 0$. The signature of the fiber-wise metric on $E_{1} \times \ol{E}_{2}$ is $(p_{1}+q_{2}, q_{1} + p_{2})$. It follows that $\gr(\F)$ is maximally isotropic, iff
\begin{equation}
p_{1} + q_{1} = \min \{p_{1} + q_{2}, q_{1} + p_{2} \}.
\end{equation}
This equation holds, iff either $p_{1} = p_{2}$ or $q_{1} = q_{2}$. 
\end{proof}
\begin{lemma} \label{lem_grF2}
Let $\F: E_{1} \rightarrow E_{2}$ be a vector bundle map over $\varphi: M_{1} \rightarrow M_{2}$. Then the subbundle $\gr(\F)$ is compatible with the anchor, iff the following diagram commutes:
\begin{equation} \label{eq_Fanchcomp}
\begin{tikzcd}
E_{1} \arrow{r}{\F} \arrow{d}{\rho_{1}} & E_{2} \arrow{d}{\rho_{2}} \\
TM_{1} \arrow{r}{T(\varphi)} & TM_{2}
\end{tikzcd}.
\end{equation}
The subbundle $\gr(\F)^{\perp}$ is compatible with the anchor, iff for all $f_{2} \in C^{\infty}(M_{2})$, the sections $\D_{1}(f_{2} \circ \varphi)$ and $\D_{2}f_{2}$ are $\F$-related, that is for all $m_{1} \in M_{1}$, one has 
\begin{equation} \label{eq_Fanchcomp2}
\F((\D_{1}(f_{2} \circ \varphi))(m_{1})) = (\D_{2}f_{2})(\varphi(m_{1})).
\end{equation}
\end{lemma}
\begin{proof}
Let $(e_{1}, \F(e_{1})) \in \gr(\F)$. We have $\rho(e_{1},\F(e_{1})) = (\rho_{1}(e_{1}), \rho_{2}(\F(e_{1})))$. We have $T(\gr(\varphi)) = \gr(T(\varphi))$. The compatibility with the anchor is thus equivalent to $\rho_{2}(\F(e_{1})) = T(\varphi)(\rho_{1}(e_{1}))$ for all $e_{1} \in E_{1}$. But this is precisely the commutativity of (\ref{eq_Fanchcomp}). 

Next, let us identify the subbundle $\gr(\F)^{\perp}$. Let $g_{1}: E_{1} \rightarrow E_{1}^{\ast}$ and $g_{2}: E_{2} \rightarrow E_{2}^{\ast}$ denote the vector bundle isomorphisms induced by the respective metrics. For each $m_{1} \in M_{1}$, one may construct a linear map $\F^{\ast}_{m_{1}}: (E_{2})_{\varphi(m_{1})} \rightarrow (E_{1})_{m_{1}}$ defined as $\F_{m_{1}}^{\ast} = g_{1}^{-1} \circ (\F_{m_{1}})^{T} \circ g_{2}$. Then
\begin{equation}
\gr(\F)^{\perp}_{(m_{1},\varphi(m_{1}))} = \gr( \F^{\ast}_{m_{1}}) = \{ (\F^{\ast}_{m_{1}}(e_{2}),e_{2}) \; | \; e_{2} \in (E_{2})_{\varphi(m)} \}.
\end{equation}
Now, the subbundle $\gr(\F)^{\perp}$ is compatible with the anchor, iff every $m_{1} \in M_{1}$ and every $e_{2} \in (E_{2})_{\varphi(m_{1})}$ satisfies the condition $\rho_{2}(e_{2}) = (T_{m_{1}} \varphi)( \rho_{1}( \F^{\ast}_{m_{1}}(e_{2})))$. Using the non-degeneracy of $\<\cdot,\cdot\>_{2}$, this can be equivalently rewritten as an equation
\begin{equation} \label{eq_Fanchcomp3}
\rho_{2}^{\ast}(\alpha_{2}) = \F_{m_{1}}( \rho_{1}^{\ast}( \varphi^{\ast}(\alpha_{2}))),
\end{equation}
for all $\alpha_{2} \in T^{\ast}_{\varphi(m_{1})}M_{2}$. Now, plugging $\alpha_{2} = (df_{2})_{\varphi(m_{1})}$ for $f_{2} \in C^{\infty}(M_{2})$ gives (\ref{eq_Fanchcomp2}). Conversely, using (\ref{eq_Fanchcomp2}) on local coordinate functions of $M_{2}$ around $\varphi(m_{1})$ proves (\ref{eq_Fanchcomp3}) for the basis, hence for every $\alpha_{2}$ by linearity. This finishes the proof.
\end{proof}
\begin{example} Let $\varphi: M_{1} \rightarrow M_{2}$ be any smooth map. Let $E_{1} = 0_{M_{1}}$ be a trivial vector bundle over $M_{1}$. Then there is the unique vector bundle map $0: 0_{M_{1}} \rightarrow E_{2}$ over $\varphi$. We have $\gr(0) = 0_{\gr(\varphi)}$. As we have already noted, this is not necessarily an involutive structure in $0_{M_{1}} \times \ol{E}_{2}$. To see it explicitly, (\ref{eq_Fanchcomp3}) forces the condition $\rho_{2}^{\ast}(\alpha_{2}) = 0$ for all $\alpha_{2} \in T^{\ast}_{\varphi(m_{1})} M_{2}$. This in turn makes the anchor $\rho_{2}$ to vanish at all points of $\varphi(M_{1})$. On the other hand, this proves that the trivial subbundle may define an involutive structure. Consider e.g. a constant map $\varphi$, whose image is a point where $\rho_{2}$ vanishes. 
\end{example}
Let us turn our attention to the involutivity. This is always the tricky one. Before the actual formulation of the main statement, note that any vector bundle map $\F: E_{1} \rightarrow E_{2}$ over $\varphi$ induces a unique vector bundle map $\F^{!}: E_{1} \rightarrow \varphi^{!}(E_{2})$ over the identity, where $\varphi^{!}(E_{2}) \rightarrow M_{1}$ denotes the pullback of $E_{2}$ along $\varphi$. Every local section $\psi_{2} \in \Gamma_{U}(E_{2})$ induces the pullback section $\psi_{2}^{!} \in \Gamma_{\varphi^{-1}(U)}(\varphi^{!}(E_{2}))$. In particular, if $(\psi_{\mu})_{\mu=1}^{\rk(E_{2})}$ is a local frame for $E_{2}$ over $U \subseteq M_{2}$, then $(\psi_{\mu}^{!})_{\mu=1}^{\rk(E_{2})}$ forms a local frame for $\varphi^{!}(E_{2})$ over $\varphi^{-1}(U)$.  
\begin{theorem} \label{thm_grF}
Let $\F: E_{1} \rightarrow E_{2}$ be a vector bundle map over $\varphi: M_{1} \rightarrow M_{2}$. Suppose $\gr(\F)$ defines an almost involutive structure, that is assume that the conditions on $(\F,\varphi)$ derived in Lemma \ref{lem_grF1} and Lemma \ref{lem_grF2} are met. 

Then the involutivity of $\gr(\F)$ is equivalent to the following condition: 

Let $\psi_{1},\psi'_{1} \in \Gamma(E_{1})$ be any two sections. Let $(\psi_{\mu})_{\mu=1}^{\rk(E_{2})}$ be any local frame for $E_{2}$ over $U$. By construction, there are unique smooth functions $f^{\mu},g^{\nu} \in C^{\infty}(\varphi^{-1}(U))$, such that on $\varphi^{-1}(U)$, one can write $\F^{!}(\psi_{1}) = f^{\mu} \psi_{\mu}^{!}$, $\F^{!}(\psi'_{1}) = g^{\nu} \psi_{\nu}^{!}$. Then on $\varphi^{-1}(U)$, the equation
\begin{equation} \label{eq_Finvol}
\F^{!}([\psi_{1},\psi'_{1}]_{1}) = f^{\mu} g^{\nu} [\psi_{\mu},\psi_{\nu}]_{2}^{!} + \Li{\rho_{1}(\psi_{1})}(g^{\nu}) \psi_{\nu}^{!} - \Li{\rho_{1}(\psi'_{1})}(f^{\mu}) \psi_{\mu}^{!} + \<\psi_{\mu}^{!},\F^{!}(\psi'_{1})\>_{2} \F^{!}( \D_{1}f^{\mu}), 
\end{equation}
must hold true, where on the right-hand side, $\<\cdot,\cdot\>_{2}$ is the pullback fiber-wise metric on $\varphi^{!}(E_{2})$. 
\end{theorem}
\begin{proof}
Suppose $\gr(\F)$ is involutive. Let $\psi_{1},\psi'_{1} \in \Gamma(E_{1})$, $(\psi_{\mu})_{\mu=1}^{\rk(E_{2})}$ and $f^{\mu},g^{\nu} \in C^{\infty}(\varphi^{-1}(U))$ be defined as above. Write $\hat{U} := \varphi^{-1}(U) \times U$. Define $\psi,\psi' \in \Gamma_{\hat{U}}(E_{1} \times \ol{E}_{2}; \gr(\F))$ as 
\begin{equation} \label{eq_Finvol2}
\psi(m_{1},m_{2}) := ( \psi_{1}(m_{1}), f^{\mu}(m_{1}) \psi_{\mu}(m_{2})), \; \; \psi'(m_{1},m_{2}) := ( \psi'_{1}(m_{1}), g^{\nu}(m_{1}) \psi_{\nu}(m_{2})),
\end{equation}
for all $(m_{1},m_{2}) \in \hat{U}$. By assumption, one has $[\psi,\psi'] \in \Gamma_{\hat{U}}(E_{1} \times \ol{E}_{2}; \gr(\F))$ as well. Plugging $\psi$ and $\psi'$ into the bracket of $E_{1} \times \ol{E}_{2}$ now gives 
\begin{equation} \label{eq_Finvol3}
\begin{split}
[\psi,\psi'](m_{1},\varphi(m_{1})) = \big( & [\psi_{1},\psi'_{1}]_{1}(m_{1}) - \< \psi_{\mu}^{!}, \F^{!}(\psi'_{1}) \>(m_{1}) (\D_{1}f^{\mu})(m_{1}), \\
& f^{\mu}(m_{1}) g^{\nu}(m_{1}) [\psi_{\mu},\psi_{\nu}]_{2}(\varphi(m_{1})) + (\Li{\rho_{1}(\psi_{1})}(g^{\nu}))(m_{1}) \psi_{\nu}(\varphi(m_{1})) \\
& - (\Li{\rho_{1}(\psi'_{1})}(f^{\mu}))(m_{1}) \psi_{\mu}(\varphi(m_{1})) \big).
\end{split}
\end{equation}
Now, the second component of the right-hand side must be an $\F$-image of the first component. But this gives precisely the condition (\ref{eq_Finvol}) evaluated at $m_{1} \in \varphi^{-1}(U)$. 

Conversely, suppose that the condition in the statement of the theorem holds. To prove the involutivity, by Lemma \ref{lem_invlocinv} it suffices to cover $\gr(\varphi)$ by open sets where $\gr(\F)$ is locally involutive. Let $(m_{1},\varphi(m_{1})) \in \gr(\varphi)$. Pick any local frame $(\psi_{\mu})_{\mu=1}^{\rk(E_{2})}$ for $E_{2}$ on $U \subseteq M_{2}$ containing the point $\varphi(m_{1})$. We will argue that $\gr(\F)$ is locally involutive on $\hat{U} := \varphi^{-1}(U) \times U$. The most general elements of $\Gamma_{\hat{U} \cap \gr(\varphi)}(\gr(\F))$ can be obtained via the restriction of sections $\psi,\psi' \in \Gamma_{\hat{U}}(E_{1} \times \ol{E}_{2})$ defined by the formula analogous to (\ref{eq_Finvol2}), where $\psi_{1},\psi'_{1} \in \Gamma_{\varphi^{-1}(U)}(E_{1})$. Now, it is clear that (\ref{eq_Finvol}) must hold also for local sections $\psi_{1},\psi'_{1} \in \Gamma_{\varphi^{-1}(U)}(E_{1})$. By the calculation in the previous paragraph, it is equivalent to the condition $[\psi,\psi'] \in \Gamma_{\hat{U}}(E_{1} \times \ol{E}_{2}, \gr(\F))$. Hence $\gr(\F)$ is locally involutive on $\hat{U}$. As $(m_{1},\varphi(m_{1}))$ was an arbitrary point of $\gr(\varphi)$, this proves the claim. 
\end{proof}
\begin{definice} \label{def_classicalCA}
Let $\F: E_{1} \rightarrow E_{2}$ be a vector bundle map over $\varphi: E_{1} \rightarrow E_{2}$. We say that $\F$ is a \textbf{classical Courant algebroid morphism over $\varphi$}, if $\gr(\F): E_{1} \rat E_{2}$ is a Courant algebroid morphism over $\varphi$. Equivalently, this means that $(\F,\varphi)$ satisfies the conditions of Lemma \ref{lem_grF1}, Lemma \ref{lem_grF2} and Theorem \ref{thm_grF}. 
\end{definice}
\begin{rem} \label{rem_classicalCA}
The notion of (classical) Courant algebroid morphism seems to be relatively unknown. However, it was developed already twenty years ago as an example in \cite{popescu1999generalized}. 

Note that there is also a ``pedestrian approach'' how to find the condition (\ref{eq_Finvol}). Indeed, one may try to generalize the notion of Lie algebroid morphism in the sense of Mackenzie, see Section 4.3 of \cite{Mackenzie}. Their conditions are the commutativity of (\ref{eq_Fanchcomp}) together with the condition (\ref{eq_Finvol}) without the last term. For Lie algebroids, the right-hand side of this equation does not depend on the choice of the local frame $(\psi_{\mu})_{\mu=1}^{\rk(E_{2})}$. This is no longer true for Courant algebroids due to the axiom C4). However, this can be saved by adding the last term together with the condition (\ref{eq_Fanchcomp2}).
\end{rem}

\begin{rem} \label{rem_graphoverdiff}
It follows immediately from Lemma \ref{lem_relatedcons} that if $\psi_{1} \sim_{\F} \psi_{2}$, $\phi_{1} \sim_{\F} \phi_{2}$ and $\F$ is a classical Courant algebroid morphism, then $[\psi_{1},\phi_{1}]_{1} \sim_{\F} [\psi_{2},\phi_{2}]_{2}$. Moreover, if $\varphi$ is a diffeomorphism, one can define $\F(\psi_{1}) = \F \circ \psi_{1} \circ \varphi^{-1}$ for each $\psi_{1} \in \Gamma(E_{1})$. The complicated involutivity condition (\ref{eq_Finvol}) is then equivalent to the usual equation
\begin{equation} \label{eq_Finvol4}
\F([\psi_{1},\psi'_{1}]_{1}) = [\F(\psi_{1}),\F(\psi'_{1})]_{2},
\end{equation}
which has to hold for all $\psi_{1},\psi'_{1} \in \Gamma(E_{1})$. 
\end{rem}

\begin{tvrz}
Let $\F: E_{1} \rightarrow E_{2}$ and $\F': E_{2} \rightarrow E_{3}$ be a pair of classical Courant algebroid morphisms over $\varphi: M_{1} \rightarrow M_{2}$ and $\varphi': M_{2} \rightarrow M_{3}$, respectively.

Then the Courant algebroid relations $\gr(\F)$ and $\gr(\F')$ compose cleanly and 
\begin{equation} \label{eq_grFcomp} \gr(\F') \circ \gr(\F) = \gr(\F' \circ \F). \end{equation}
In particular, the composition $\F' \circ \F$ is a classical Courant algebroid morphism over $\varphi' \circ \varphi$. Courant algebroids together with classical Courant algebroid morphisms thus form the category \textup{$\cCA$}.
\end{tvrz}
\begin{proof}
It is straightforward to show (\ref{eq_grFcomp}). Using the same arguments as in Example \ref{ex_CAmorphcomp}, one can show that $\gr(\F)$ and $\gr(\F')$ satisfy the condition $(i)$ of both Proposition \ref{tvrz_RdiamondR} and Proposition \ref{tvrz_RcircR}, hence they compose cleanly. It then follows immediately from Theorem \ref{thm_composition} that $\gr(\F' \circ \F): E_{1} \rat E_{3}$ is a Courant algebroid morphism over $\varphi' \circ \varphi$ and consequently,  $\F' \circ \F$ is a classical Courant algebroid morphism. The final statement follows from Proposition \ref{tvrz_CArelCAT} and the observation that in fact $\Delta(E) = \gr(1_{E})$. This finishes the proof.  
\end{proof}
\begin{example}
It is not that easy to come up with some non-trivial example of a classical Courant algebroid morphism. Let $(E,\rho,\<\cdot,\cdot\>,[\cdot,\cdot])$ be a transitive Courant algebroid over $M$, that is $\rho$ is fiber-wise surjective. We thus have a sequence (not necessarily exact):
\begin{equation}
\begin{tikzcd}
0 \arrow{r} & T^{\ast}M \arrow{r}{\rho^{\ast}} & E \arrow{r}{\rho} & TM \arrow{r} & 0.
\end{tikzcd}
\end{equation}
One can always construct its splitting $\sigma: TM \rightarrow E$, such that $\sigma(TM) \subseteq E$ is isotropic. Define $H \in \Omega^{3}(M)$ by $H(X,Y,Z) = -\< [\sigma(X),\sigma(Y)], \sigma(Z) \>$, for all $X,Y,Z \in \X(M)$. This $3$-form is not necessarily closed. There is a canonical Lie algebroid bracket $[\cdot,\cdot]_{L}$ on the sections of $L = E / \im(\rho^{\ast})$ and a vector bundle map $\sigma_{L}: TM \rightarrow L$ obtained by composing $\sigma$ with the quotient map $E \rightarrow L$. One can show that $dH = 0$, if and only if $\sigma_{L}([X,Y]) = [\sigma_{L}(X),\sigma_{L}(Y)]_{L}$ for all $X,Y \in X(M)$. See e.g. \cite{Baraglia:2013wua} for details. Hence suppose that this is the case. 

Let $\gTM = TM \oplus T^{\ast}M$ be the standard Courant algebroid equipped with the $H$-twisted Dorfman bracket $[(X,\xi),(Y,\eta)] = ([X,Y],\Li{X}\eta - i_{Y}(d\xi) - H(X,Y,\cdot))$. Define $\F: \gTM \rightarrow E$ as $\F(X,\xi) = \sigma(X) + \rho^{\ast}(\xi)$. We claim that this is a classical Courant algebroid morphism over $1_{M}$. We have $\rho(\F(X,\xi)) = X$, hence (\ref{eq_Fanchcomp}) commutes, and $\F((0,df)) = \rho^{\ast}(df) = \D{f}$, hence (\ref{eq_Fanchcomp2}) stands true. The isotropy of $\sigma(TM)$ implies (\ref{eq_Fisotropic}) and we conclude that $\gr(\F)$ is an almost involutive structure. As $1_{M}$ is a diffeomorphism, it remains to verify the condition (\ref{eq_Finvol4}). But this follows from Proposition 3.2 in \cite{Baraglia:2013wua}, where we have $F = 0$ due to our condition on $\sigma_{L}$ above. 

Note that there is a canonical pairing $\<\cdot,\cdot\>_{\frk}$ on the kernel $\frk \subseteq L$ of the Lie algebroid anchor of $L$. It is easy to see that $\gr(\F)$ is a Dirac structure, iff  it is either positive or negative definite. In the split signature setting of \cite{li2014dirac}, it is a Courant algebroid morphism, iff $\frk = 0$. This happens precisely when $E$ is an exact Courant algebroid, and in this case, $\F$ is just the Ševera isomorphism $\gTM \cong E$. 
\end{example}
To conclude the discussion of this class of Courant algebroid relations, let us discuss the matter of pullback and pushforward involutive structures in the sense of Definition \ref{def_pushpull}. 
\begin{tvrz} \label{tvrz_pull}
Let $L_{2} \subseteq E_{2}$ be an involutive structure supported on $S_{2}$. Let $\F: E_{1} \rightarrow E_{2}$ be a classical Courant algebroid morphism over $\varphi: M_{1} \rightarrow M_{2}$. 

Then $\gr(\F)^{\ast}(L_{2})$ is equal to the inverse image $\F^{-1}(L_{2})$ and it defines a pullback involutive structure, iff either of the following two conditions holds:
\begin{enumerate}[(i)]
\item $\F^{-1}(L_{2})$ is a submanifold and $T_{e_{1}}( \F^{-1}(L_{2})) = (T_{e_{1}}\F)^{-1}( T_{\F(e_{1})} L_{2})$ for all $e_{1} \in \F^{-1}(L_{2})$;
\item $\varphi^{-1}(S_{2})$ is a submanifold, $T_{m_{1}}(\varphi^{-1}(S_{2})) = (T_{m_{1}}\varphi)^{-1}(T_{\varphi(m_{1})}S_{2})$ for all $m_{1} \in \varphi^{-1}(S_{2})$, and the subspace $\F_{m_{1}}^{-1}( (L_{2})_{\varphi(m_{1})})$ has the same dimension for all $m_{1} \in \varphi^{-1}(S_{2})$. 
\end{enumerate}
\end{tvrz}
\begin{proof}
The observation $\gr(\F)^{\ast}(L_{2}) = \F^{-1}(L_{2})$ follows easily from the definition. It defines a pullback involutive structure, if $\gr(\F)$ and $L_{2} \times \{0\}$ compose cleanly. We will now argue that the condition $(i)$ is equivalent to the conditions $(i)$ in Proposition \ref{tvrz_RdiamondR} and \ref{tvrz_RcircR}, whereas $(ii)$ corresponds to the conditions $(ii)$ of the same propositions. 

As $\F^{-1}(L_{2}) \times \{0\} = (L_{2} \times \{0\}) \circ \gr(F)$, the subset $\F^{-1}(L_{2})$ must be a submanifold of $E_{1}$. Let us construct the subset $(L_{2} \times \{0\}) \diamond \gr(\F)$. By definition, one has
\begin{equation}
\begin{split}
(L_{2} \times \{0\}) \diamond \gr(\F) = & \ (\gr(\F) \times (L_{2} \times \{0\})) \cap (E_{1} \times \Delta(E_{2}) \times \{0\}) \\
= & \ \big( (\gr(\F) \times L_{2}) \cap (E_{1} \times \Delta(E_{2})) \big) \times \{0 \} \\
= & \ \{ (e_{1}, \F(e_{1}), \F(e_{1})) \; | \; e_{1} \in \F^{-1}(L_{1}) \} \times \{ 0 \}. 
\end{split}
\end{equation}
The set in the first factor is the graph of the smooth map $e_{1} \mapsto (\F(e_{1}),\F(e_{1}))$ from $\F^{-1}(L_{2})$ to $E_{2} \times E_{2}$, hence a submanifold of $\F^{-1}(E_{1}) \times \ol{E}_{2} \times E_{2}$, hence of $E_{1} \times \ol{E}_{2} \times E_{2}$. This proves that $(L_{2} \times \{0\}) \diamond \gr(\F)$ is a submanifold. One has to examine the condition (\ref{eq_cleancondition}). For any $e_{1} \in \F^{-1}(L_{2})$ the intersection of the tangent spaces of both submanifolds at $(e_{1},\F(e_{1}),\F(e_{1}),0)$ takes the form 
\begin{equation}
\{ (x_{1}, (T_{e_{1}}\F)(x_{1}), (T_{e_{1}}\F)(x_{1})) \; | \; x_{1} \in (T_{e_{1}}\F)^{-1}( T_{\F(e_{1})}L_{2}) \} \times \{0\}.
\end{equation}
On the other hand, the tangent space to $(L_{2} \times \{0\}) \diamond \gr(\F)$ at the same point reads
\begin{equation}
\{ (x_{1}, (T_{e_{1}}\F)(x_{1}), (T_{e_{1}}\F)(x_{1})) \; | \; x_{1} \in T_{e_{1}}( \F^{-1}(L_{2})) \} \times \{0\}.
\end{equation}
These sets are equal, iff the condition on tangent spaces in $(i)$ of this proposition holds. The projection map $p: (L_{2} \times \{0\}) \diamond \gr(\F) \rightarrow (L_{2} \times \{0\}) \circ \gr(\F)$ is always a diffeomorphism. This shows that $\F^{-1}(L_{2})$ is a pullback involutive structure, iff $(i)$ holds. Similarly, the first part condition $(ii)$ is equivalent to the requirements on the supports in the conditions $(ii)$ of Proposition \ref{tvrz_RdiamondR} and Proposition \ref{tvrz_RcircR}. Finally, the second part ensures the constant dimension requirements.  
\end{proof}
\begin{example}
The condition $(i)$ is satisfied if the map $\F$ is transverse to the submanifold $L_{2}$. This is equivalent to the transversality of $\varphi$ to the submanifold $S_{2}$ together with the linear transversality condition $\im( \F_{m_{1}}) + (L_{2})_{\varphi(m_{1})} = (E_{2})_{\varphi(m_{1})}$. 
\end{example}
Let us turn our attention towards the pushforwards. The proof of the following proposition is very similar to the previous one and we leave it up to the interested reader. Note that the condition $(ii)$ is a lot simpler. This is because every $\F$ has to be fiber-wise injective. 
\begin{tvrz} \label{tvrz_push}
Let $L_{1} \subseteq E_{1}$ be an involutive structure supported on $S_{1}$. Let $\F: E_{1} \rightarrow E_{2}$ be a classical Courant algebroid morphism over $\varphi: M_{1} \rightarrow M_{2}$. 

Then $\gr(\F)_{\ast}(L_{1})$ is equal to the image $\F(L_{1})$ and it defines a pushforward involutive structure, iff either of the following two conditions holds:
\begin{enumerate}[(i)]
\item $\F(L_{1})$ is a submanifold and $T_{\F(e_{1})}( \F(L_{1})) = (T_{e_{1}}\F)( T_{e_{1}}L_{1})$ for all $e_{1} \in L_{1}$;
\item $\varphi(S_{1})$ is a submanifold and $T_{\varphi(m_{1})}(\varphi(S_{1})) = (T_{m_{1}}\varphi)(T_{m_{1}}S_{1})$ for all $m_{1} \in S_{1}$.
\end{enumerate}
\end{tvrz}
\begin{example} \label{ex_compfails}
We have promised to show an example of Courant algebroid relations $R$ and $R'$ which \textit{cannot} be composed. It is hardly useful, but an example it is.

Let $E_{1} = \R^{2} \times \R$ and $E_{2} = \R \times \R^{2}$ be a pair of vector bundles equipped with standard Euclidean fiber-wise metrics $\<\cdot,\cdot\>_{1}$ and $\<\cdot,\cdot\>_{2}$, respectively. Equip those with trivial anchors and brackets to make them into Courant algebroids. Define $\F: E_{1} \rightarrow E_{2}$ as 
\begin{equation}
\F((x,y),t) = (x, ( \cos(y)t, \sin(y)t)),
\end{equation}
for all $((x,y),t) \in E_{1}$. This is a fiber-wise injective vector bundle map over a surjective submersion $\varphi(x,y) = x$.  It is not difficult to see that it defines a classical Courant algebroid morphism from $E_{1}$ to $E_{2}$ over $\varphi$. We thus have a relation $\gr(\F): E_{1} \dra E_{2}$. Now, let us show that the composition $\gr(\F)^{T} \circ \gr(\F)$ is not a subbundle of $E_{1} \times \ol{E}_{1}$. Its base $\gr(\varphi)^{T} \circ \gr(\varphi)$ is just a fibered product 
\begin{equation}
\begin{split}
\gr(\varphi)^{T} \circ \gr(\varphi) = & \ \R^{2} \times_{\R} \R^{2} = \{ ((x,y),(x',y')) \in \R^{2} \times \R^{2} \; | \; \varphi(x,y) = \varphi(x',y') \} \\
= & \ \{ ((x,y),(x',y')) \in \R^{2} \times \R^{2} \; | \; x = x' \}.
\end{split}
\end{equation}
This is a submanifold of $\R^{2} \times \R^{2}$. However, $\gr(\F)^{T} \circ \gr(\F)$ fails to be a subbundle, as its fiber over $((x,y),(x,y'))$ can be identified with the vector subspace 
\begin{equation}
(\gr(\F)^{T} \circ \gr(\F))_{((x,y),(x,y'))} = \{ (t,t') \in \R \times \R \; | \; t (\cos(y),\sin(y)) = t' (\cos(y'), \sin(y')) \}.
\end{equation} 
Its dimension is $1$ for $y' - y \in \mathbb{Z} \{2\pi\}$ and $0$ otherwise. The dimension of the fibers of $\gr(\F)^{T} \circ \gr(\F)$ along $\gr(\varphi)^{T} \circ \gr(\varphi)$ is thus not even locally constant, hence it is not a subbundle. 
\end{example}

\subsection{Dorfman functor} \label{subsec_Dorfman}
This subsection is based on notions introduced in Example \ref{ex_LADorfman}. Let $A_{1}$ and $A_{2}$ be a pair of Lie algebroids over $M_{1}$ and $M_{2}$, respectively. 

By a \textbf{Lie algebroid relation} $K: A_{1} \dra A_{2}$ supported on $S$, one means a Lie subalgebroid $K \subseteq A_{1} \times A_{2}$ over a submanifold $S$. Such relations can be composed under completely the same conditions as Courant algebroid relations. Hence they also form a ``category'' which we shall denote as $\ol{\cLA}$. Similarly to the previous subsection, if one considers $K = \gr(\F)$ for a vector bundle map $\F: A_{1} \rightarrow A_{2}$ over $\varphi: M_{1} \rightarrow M_{2}$, one recovers the notion of a Lie algebroid morphism (see Remark \ref{rem_classicalCA}). Lie algebroids and such morphisms form an actual category of Lie algebroids $\cLA$.  

Now, for a given Lie algebroid $A$, let $\Df(A)$ denote the Courant algebroid defined on the vector bundle $A \oplus A^{\ast}$ using the Dorfman bracket associated to $A$. See Example \ref{ex_LADorfman} for details. We will argue that the map $A \mapsto \Df(A)$ can be interpreted as a functor\footnote{One has to be a little bit careful to call something a functor between ``categories''. See the actual theorem for a more precise statement.} $\Df: \ol{\cLA} \rightarrow \ol{\cCA}$. Note that its restriction to the subcategory $\cLA$ does not take values in the category $\cCA$.

\begin{theorem} \label{thm_Df}
Let $(A_{1},a_{1},[\cdot,\cdot]_{A_{1}})$ and $(A_{2}, a_{2}, [\cdot,\cdot]_{A_{2}})$ be a pair of Lie algebroids over $M_{1}$ and $M_{2}$, respectively. Let $K: A_{1} \dra A_{2}$ be a Lie algebroid relation supported on $S$. 

Then there exists a canonical Courant algebroid relation $R_{K}: \Df(A_{1}) \dra \Df(A_{2})$ supported on $S$. If $K': A_{2} \dra A_{3}$ is another Lie algebroid relation, Then $R_{K' \circ K} = R_{K'} \circ R_{K}$. In particular, if $K' \circ K$ happens to be a Lie algebroid relation, then $R_{K'} \circ R_{K}$ is a Courant algebroid relation. Moreover, one has $R_{\Delta(A_{1})} = \Delta( \Df(A_{1}))$. We thus call $\Df$ a \textbf{Dorfman functor}.

If $K = \gr(\F)$ for a Lie algebroid morphism $\F: A_{1} \rightarrow A_{2}$ over $\varphi$, then $R_{K} = \gr(\hat{\F})$ for a (unique) classical Courant algebroid morphism $\hat{\F}: \Df(A_{1}) \rightarrow \Df(A_{2})$ over $\varphi$, iff $\F$ is a fiber-wise bijective vector bundle map. 
\end{theorem}
\begin{proof}
Write $E_{1} = \Df(A_{1})$ and $E_{2} = \Df(A_{2})$. For every $(s_{1},s_{2}) \in S$, define 
\begin{equation} \label{eq_RKdefinice}
(R_{K})_{(s_{1},s_{2})} = \{ ((x_{1},\xi_{1}), (x_{2},\xi_{2})) \in (E_{1} \times \ol{E}_{2})_{(s_{1},s_{2})} \; | \; (x_{1},x_{2}) \in K_{(s_{1},s_{2})}, \; \xi_{1}(x_{1}) = \xi_{2}(x_{2}) \}.
\end{equation}
Instead of proving directly that $R_{K}$ is an involutive structure, we will construct a classical Courant algebroid isomorphism $\fPsi: \Df(A_{1} \times A_{2}) \rightarrow E_{1} \times \ol{E}_{2}$ over the identity and observe that 
\begin{equation}
R_{K} = \fPsi(K \oplus \an(K)).
\end{equation}
In Example \ref{ex_LADorfman}, we have shown that $K \oplus \an(K)$ is a Dirac structure in $\Df(A_{1} \times A_{2})$. Proposition \ref{tvrz_push} then immediately implies that $R_{K}$ is a Dirac structure in $E_{1} \times \ol{E}_{2}$. 

It suffices to define $\fPsi$ on generating sections. For $X_{i} \in \Gamma(A_{i})$ and $\xi_{i} \in \Gamma(A^{\ast}_{i})$, $i \in \{1,2\}$, set 
\begin{equation}
\fPsi((X_{1},X_{2}),(\xi_{1},\xi_{2})) := ((X_{1},\xi_{1}), (X_{2},-\xi_{2})).
\end{equation}
Note that the minus sign is essential due to the sign flip on $\ol{E}_{2}$. It is straightforward to prove that $\fPsi$ defines a classical Courant algebroid isomorphism over the identity. 

Directly from the definition (\ref{eq_RKdefinice}), one may show the inclusion $(R_{K'} \circ R_{K})_{(s_{1},s_{3})} \subseteq (R_{K' \circ K})_{(s_{1},s_{3})}$ for all $(s_{1},s_{3}) \in S' \circ S$. The equality now follows from the fact that these two subspaces have the same dimension. Indeed, we have shown in the previous paragraph that $(R_{K' \circ K})_{(s_{1},s_{3})}$ is maximally isotropic. But so is the composition $(R_{K'} \circ R_{K})_{(s_{1},s_{3})}$, see Proposition \ref{tvrz_coiso} and Example \ref{ex_coisored}. Note that it is \textit{vital} that the pairings on $E_{1}$ and $E_{2}$ have a split signature and there is thus no contradiction with Remark \ref{rem_compDiracfail}. Two maximally isotropic subspaces have to be of the same dimension and the equality $R_{K'} \circ R_{K} = R_{K' \circ K}$ follows. The identity $R_{\Delta(A_{1})} = \Delta(\Df(A_{1}))$ is obvious. 

Now, if $K = \gr(\F)$ for a Lie algebroid morphism $\F: A_{1} \rightarrow A_{2}$ over $\varphi: M_{1} \rightarrow M_{2}$, one finds
\begin{equation}
(R_{\gr(\F)})_{(m_{1},\varphi(m_{1}))} = \{ ((x_{1}, \F_{m_{1}}^{T}(\xi_{2})), (\F_{m_{1}}(x_{1}), \xi_{2})) \; | \; x_{1} \in (A_{1})_{m_{1}}, \; \xi_{2} \in (A^{\ast}_{2})_{\varphi(m_{1})} \}.
\end{equation}
This can be written as a graph of a vector bundle map, iff there exists an inverse of the linear map $\F_{m_{1}}: (A_{1})_{m_{1}} \rightarrow (A_{2})_{\varphi(m_{1})}$ for each $m_{1} \in M_{1}$. In other words, the bundle map $\F$ is fiber-wise bijective. If this is the case, one can define a smooth vector bundle map $\hat{\F}: \Df(A_{1}) \rightarrow \Df(A_{2})$ over $\varphi$, fiber-wise by formula $\hat{\F}_{m_{1}}( x_{1},\xi_{1}) = ( \F_{m_{1}}(x_{1}), \F_{m_{1}}^{-T}(\xi_{1}))$. Note that if $\F$ is fiber-wise bijective, the induced map $\F^{!}: A_{1} \rightarrow \varphi^{!}(A_{2})$ is a vector bundle isomorphism. The inverse of its transpose $(\F^{!})^{-T}: A_{1}^{\ast} \rightarrow \varphi^{!}(A_{2}^{\ast})$ can be then composed with the canonical vector bundle map $\varphi^{!}: \varphi^{!}(A_{2}^{\ast}) \rightarrow A_{2}^{\ast}$ and fiber-wise, this is exactly the linear map in the second component in $\hat{\F}$.
\end{proof}

\begin{rem}
This subsection is an example valid also for Courant algebroid relations defined in \cite{li2014dirac}. This is because Courant algebroids obtained by Dorfman functor always have a split signature and all relations are maximally isotropic. 
\end{rem}

\begin{example}
The assignment $M \mapsto TM$ may be viewed as a functor $T: \cMan^{\infty} \rightarrow \cLA$ from the category of smooth manifolds to the category of Lie algebroids. Indeed, if $\varphi: M \rightarrow M'$ is a smooth map, then $T(\varphi): TM \rightarrow TM'$ is easily seen to be a morphism of Lie algebroids over $\varphi$. The standard Courant algebroid on the generalized tangent bundle $\mathbb{T}M$ can be now viewed as the composed functor $\mathbb{T} = \Df \circ T$ from $\cMan^{\infty}$ to $\ol{\cCA}$ evaluated at $M \in \cMan^{\infty}$. In particular, for every smooth map $\varphi: M \rightarrow M'$, one obtains a Courant algebroid morphism $R_{\gr(T(\varphi))}: \gTM \rat \gTM'$ over $\varphi$. This is precisely Example 2.7 in \cite{bursztyn2008courant}.
\end{example}
\begin{example} \label{ex_paraHermitian}
Our next example comes from the para-Hermitian geometry \cite{Freidel:2017yuv, Svoboda:2018rci}. Recall that a \textbf{para-Hermitian manifold} is a triple $(P,\eta,K)$, where $P$ is an $2d$-dimensional smooth manifold, $\eta$ is a metric on $P$ of a split signature $(d,d)$, and $K: TP \rightarrow TP$ is a vector bundle map over $1_{P}$, such that $K^{2} = 1$ and it is anti-orthogonal with respect to $\eta$. Moreover, $K$ satisfies a certain integrability condition, which can be equivalently described as the involutivity of the pair of smooth regular distributions $T_{\pm}$ obtained as $\pm 1$ eigenbundles of $K$. 

Any involutive distribution $T_{+}$ can be viewed as a Lie algebroid $(T_{+},\ell_{+},[\cdot,\cdot]_{+})$, where the anchor $\ell_{+}: T_{+} \rightarrow TP$ is the inclusion of the subbundle and $[\cdot,\cdot]_{+}$ restricts from the Lie bracket on $\X(P)$. As $T_{\pm}$ are maximally isotropic with respect to $\eta$, there is a canonical vector bundle isomorphism $\rho_{+}: \Df(T_{+}) \rightarrow T_{+} \oplus T_{-} \equiv TP$ which can be used to induce a Courant algebroid structure $(TP, \fP_{+}, \eta, \dbl \cdot, \cdot \dbr_{+} )$, where $\fP_{+}: TP \rightarrow TP$ is the projector on the eigenbundle $T_{+}$.  

Let $\F_{+}$ be a smooth distribution induced by $T_{+}$. Let $i_{F}: F \rightarrow P$ be the injective immersion of one of its leaves $F \in \F_{+}$. By definition of the foliation corresponding to the involutive distribution, the map $T(i_{F}): TF \rightarrow T_{+}$ defines a fiber-wise bijective Lie algebroid morphism over $i_{F}$. It follows from Theorem \ref{thm_Df} that there is a canonical classical Courant algebroid morphism $\fPsi^{+}_{F}: \mathbb{T}F \rightarrow TP$ over $i_{F}: F \rightarrow P$, which is fiber-wise bijective. 

Naturally, a similar construction can be done for any leaf of the foliation $\F_{-}$ corresponding to the distribution $T_{-}$ and the Courant algebroid $(TP,\fP_{-},\eta, \dbl \cdot, \cdot \dbr_{-})$.

Note that they phrase it a little bit differently in \cite{Freidel:2017yuv, Svoboda:2018rci}. Instead of working with a single leaf, they consider a manifold $\F_{+} = \bigsqcup_{F \in \F_{+}} F$, a disjoint union of all leaves of the foliation. As it has uncountably many connected components, it is not a second-countable topological space. The set $\F_{+}$ is one-to-one with the manifold $P$. The collection of the above maps then defines a single morphism $\fPsi^{+}: \mathbb{T}\F_{+} \rightarrow TP$ which they call an isomorphism of Courant algebroids. However, strictly speaking, the smooth bijection $\F_{+} \rightarrow P$ does not have a smooth inverse and consequently, the fiber-wise inverse of $\fPsi^{+}$ is not a vector bundle map. 
\end{example}
\subsection{Reduction of Courant algebroids} \label{subsec_reduction}
For a comprehensive treatment of the Courant algebroid reductions by group actions, see \cite{Bursztyn2007726}. We consider only its simplified form which found its applications e.g. in the geometrical description of Kaluza--Klein reduction of supergravity \cite{Vysoky:2017epf} and Poisson--Lie T-duality \cite{Jurco:2017gii,Severa:2015hta}.

Let $\varpi: P \rightarrow M$ be a principal $G$-bundle, where $G$ is any connected Lie group. Suppose $q: E \rightarrow P$ is a vector bundle equipped with a structure of \textbf{$G$-equivariant Courant algebroid} $(E,\rho,\<\cdot,\cdot\>,[\cdot,\cdot],\Re)$. The additional structure is a linear map $\Re: \g \rightarrow \Gamma(E)$, where $\g = \Lie(G)$ is the Lie algebra of $G$, satisfying:
\begin{enumerate}[(i)]
\item $\rho \circ \Re = \#$, where $\#: \g \rightarrow \X(P)$ is the infinitesimal generator of the right action of $G$ on $P$;
\item $\Re([x,y]_{\g}) = [\Re(x),\Re(y)]$, where $[\cdot,\cdot]_{\g}$ is the Lie algebra bracket on $\g$;
\item the induced Lie algebra action $x \btr \psi := [\Re(x),\psi]$ of $\g$ on $\Gamma(E)$ integrates to a Lie group action $\mathfrak{R}$ of $G$ on $E$ making it into a $G$-equivariant vector bundle, see Section 3.1 of \cite{Mackenzie}. In particular, there is a unique vector bundle structure $q^{\natural}: E/G \rightarrow M$ on the quotient manifold, making the quotient map $\natural: E \rightarrow E/G$ into the vector bundle map over $\varpi$. 
\end{enumerate}
Now, as the bracket $[\cdot,\cdot]_{\g}$ is skew-symmetric and $[\cdot,\cdot]$ is not, we obtain the equation
\begin{equation}
0 = \Re([x,y]_{\g} + [y,x]_{\g}) = [\Re(x),\Re(y)] + [\Re(y),\Re(x)] = \D \<\Re(x),\Re(y)\>,
\end{equation}
for all $x,y \in \g$. Recall that $\D = \rho^{\ast} \circ \dr$. If $E$ is a transitive Courant algebroid, this implies that the function $\<\Re(x),\Re(y)\>$ is constant on every connected component of the base $P$. For simplicity, suppose that that regardless of the transitivity of $E$, it is constant on $P$. Hence the formula
\begin{equation} \label{eq_Rinducedpair}
(x,y)_{\g} = \<\Re(x),\Re(y)\>
\end{equation}
defines a symmetric bilinear form on $\g$. In fact, the axiom C3) implies that it is $\ad$-invariant. Note that it \textit{does not} have to be non-degenerate, let $(p_{0},q_{0},k_{0})$ denote its inertia.  
\begin{rem}
The devil is in the detail. If $E$ is not transitive, $\Re$ can be still used to induce a fiber-wise symmetric bilinear form on $P \times \g$. However, unlike the signature of a fiber-wise metric, the inertia of a fiber-wise bilinear form \textit{does not have to be locally constant}. In particular, the dimension of its kernel may not be locally constant. But the dimension of its kernel at $p$ is exactly the dimension of the intersection $K_{p} \cap K^{\perp}_{p}$, which would prevent $K \cap K^{\perp}$ from being a subbundle of $E$. See below for possible consequences. 
\end{rem}
Now, viewing $\Re$ as a fiber-wise injective vector bundle map from $P \times \g$ to $E$, one constructs a $G$-invariant subbundle $K = \Re(P \times \g)$ and its orthogonal complement $K^{\perp}$. It follows that the $C^{\infty}(M)$-module $\Gamma_{G}(K^{\perp})$ of its $G$-invariant sections is involutive with respect to $[\cdot,\cdot]$. However, the restriction of $\<\cdot,\cdot\>$ to $\Gamma_{G}(K^{\perp})$ is degenerate, so one has to take out its kernel $\Gamma_{G}(K \cap K^{\perp})$. The \textbf{reduced Courant algebroid} $(E',\rho',\<\cdot,\cdot\>',[\cdot,\cdot]')$ is then defined on the vector bundle 
\begin{equation}
E' = \frac{K^{\perp} / G}{(K \cap K^{\perp}) / G},
\end{equation}
where $\<\cdot,\cdot\>'$ and $[\cdot,\cdot]'$ are naturally induced using the identification $\Gamma(E') = \Gamma_{G}(K^{\perp}) / \Gamma_{G}(K \cap K')$. Let $\chi_{E'}: K^{\perp}/G \rightarrow E'$ denote the quotient map. The anchor map $\rho': E' \rightarrow TM$ is defined by the following commutative diagram:
\begin{equation} \label{eq_redanchor}
\begin{tikzcd}
& E \arrow{r}{\rho} \arrow{d}{\natural} & TP \arrow{d}{\natural_{T}} \arrow{dr}{T(\varpi)} & \\
K^{\perp}/G \arrow{d}{\chi_{E'}} \arrow[hookrightarrow]{r}& E/G \arrow[dashed]{r} & TP/G \arrow[dashed]{r} & TM \\
E' \arrow[dashed, bend right=15]{rrru}{\rho'} &  &  
\end{tikzcd},
\end{equation}
where all dashed arrows are canonically induced on quotients by the arrows above them. 

We will now show that there is a Courant algebroid morphism $Q(\Re): E \rat E'$ over $\varpi$. For $p \in P$, define its fiber at $(p,\varpi(p))$ to have the form
\begin{equation}
Q(\Re)_{(p,\varpi(p))} := \{ (e, \chi_{E'}(\natural(e))) \; | \; e \in K^{\perp}_{p} \} \subseteq (E \times \ol{E}')_{(p,\varpi(p))}. 
\end{equation}
This defines a vector subbundle of $E \times \ol{E}'$ and $\rk(Q(\Re)) = \rk(K^{\perp}) = \rk(E) - \dim(\g)$. Hence for $\g \neq 0$, $Q(\Re)$ is not a graph of a vector bundle map from $E$ to $E'$. The pairing $\<\cdot,\cdot\>'$ is induced from the one of $E$, which immediately implies that $Q(\Re)$ is isotropic. 
\begin{lemma} \label{lem_QRemaxiso}
$Q(\Re)$ is maximally isotropic, iff the induced bilinear form $(\cdot,\cdot)_{\g}$ is either positive semi-definite or negative semi-definite. 
\end{lemma}
\begin{proof}
Let $(p,q)$ and $(p',q')$ be the signature of the fiber-wise metric $\<\cdot,\cdot\>$ and $\<\cdot,\cdot\>'$, respectively. Let $(p_{0},q_{0},k_{0})$ be the inertia of $(\cdot,\cdot)_{\g}$. By definition, this is precisely the inertia of the restriction of $\<\cdot,\cdot\>|_{K}$. The inertia of $\<\cdot,\cdot\>|_{K^{\perp}}$ is $(p',q',k_{0})$. It follows from Proposition \ref{tvrz_inertias} that $p' = p - p_{0} - k_{0}$ and $q' = q - q_{0} - k_{0}$. $Q(\Re)$ is maximally isotropic, iff $\rk(Q(\Re)) = \min\{ p + q', q + p' \}$. Plugging into both sides then gives the condition
\begin{equation}
p + q - (p_{0} + q_{0} + k_{0}) = \min \{ p + q - q_{0} - k_{0}, p + q - p_{0} - k_{0} \} = p + q - k_{0} - \max \{p_{0},q_{0} \}.
\end{equation}
One can rewrite it as $p_{0} + q_{0} = \max \{p_{0},q_{0}\}$. This happens, iff either $q_{0} = 0$ or $p_{0} = 0$. 
\end{proof}
\begin{lemma}
$Q(\Re)^{\perp}$ is compatible with the anchor, hence $Q(\Re)$ is an almost involutive structure in $E \times \ol{E}'$ supported on $\gr(
\varpi)$. 
\end{lemma}
\begin{proof}
Recall that $Q(\Re)$ is a subbundle supported on $\gr(\varpi) \subseteq P \times M$. It follows immediately from (\ref{eq_redanchor}) that $Q(\Re)$ is compatible with the anchor. Next, it is not difficult to show that for each $p \in P$, the fiber of $Q(\Re)^{\perp}$ at $(p,\varpi(p))$ takes the form
\begin{equation}
Q(\Re)^{\perp}_{(p,\varpi(p))} = Q(\Re)_{(p,\varpi(p))} + K_{p} \times \{ 0_{\varpi(p)} \}. 
\end{equation}
It thus suffices to argue that for all $e \in K_{p}$, one has $\rho'(0_{\varpi(p)}) = (T_{p}\varpi)(\rho(e))$. By definition of the subbundle $K$, one can write $e = (\Re(x))(p)$ for a unique $x \in \g$. But then $\rho(e) = \#_{p}(x)$. This is a vertical tangent vector and the above equation holds. Hence $Q(\Re)^{\perp}$ is compatible with the anchor and the proof is finished. 
\end{proof}

\begin{tvrz}
$Q(\Re)$ is involutive, hence a Courant algebroid morphism $Q(\Re): E \rat E'$. 
\end{tvrz}
\begin{proof}
As $Q(\Re)$ is an almost involutive structure, we may employ Proposition \ref{tvrz_invongen}. It thus suffices to prove the involutivity on sections in the form $(\psi, \chi_{E'}(\natural(\psi)))$, where $\psi \in \Gamma_{G}(K^{\perp})$. One has
\begin{equation}
\begin{split}
[(\psi, \chi_{E'}(\natural(\psi))), (\psi', \chi_{E'}(\natural(\psi')))] = & \ ( [\psi,\psi']_{E}, [ \chi_{E'}(\natural(\psi)), \chi_{E'}(\natural(\psi'))]_{E'}) \\
= & \ ([\psi,\psi']_{E}, \chi_{E'}( \natural ([\psi,\psi']_{E}))),
\end{split}
\end{equation}
for all $\psi,\psi' \in \Gamma_{G}(K^{\perp})$. The right-hand side in $\Gamma(E \times \ol{E}', Q(\Re))$ and the conclusion follows. 
\end{proof}
Let us examine the conditions on pullbacks and pushforwards of involutive structures. For pullbacks, the situation is quite straightforward. 
\begin{tvrz}
Let $L' \subseteq E'$ be an involutive structure supported on $S' \subseteq M$. Then the pullback involutive structure $Q(\Re)^{\ast}(L')$ always exists and it can be identified with the inverse image of $L'$ under the vector bundle map $\chi_{E'} \circ \natural|_{K^{\perp}}$. It defines a $G$-invariant subbundle of $E$. 
\end{tvrz}
\begin{proof}
From Definition \ref{def_pushpull}, it follows that $Q(\Re)^{\ast}(L')$ consists of points of $K^{\perp}$ which project to $L'$ under $\natural_{E'} := \chi_{E'} \circ \natural|_{K^{\perp}}$. This is a fiber-wise surjective vector map from $K^{\perp}$ to $E'$ over the surjective submersion $\varpi: P \rightarrow M$, hence a surjective submersion. In particular, it is transverse to any submanifold of $E'$ and the inverse image $L := \natural_{E'}^{-1}(L')$ is a submanifold of $K^{\perp}$. Moreover, its tangent space  at $e \in L$ can be written as $T_{e}L = (T_{e}\natural_{E'})^{-1}( T_{\natural_{E'}(e)} L')$. Hence $L$ is a submanifold of $E$. By Theorem \ref{thm_GR}, it defines a subbundle and by construction, it is $G$-invariant. The proof of the fact that $Q(\Re)$ and $L' \times \{0\}$ compose cleanly is then completely analogous to the one of Proposition \ref{tvrz_pull}. This finishes the proof. 
\end{proof}

For pushforwards, the situation is a lot more complicated. The following proposition can be viewed as a generalization of Section 4 in \cite{Bursztyn2007726}. 
\begin{tvrz} \label{tvrz_invreduction}
Let $L \subseteq E$ be an involutive structure supported on $S \subseteq P$. Then the pushforward involutive structure $Q(\Re)_{\ast}(L)$ exists, iff the following conditions are satisfied:
\begin{enumerate}[(i)]
\item $\varpi(S)$ is a submanifold of $M$ and for each $s \in S$, one has $(T_{s}\varpi)(T_{s}S) = T_{\varpi(s)}(\varpi(S))$;
\item The dimension of $(K^{\perp} \cap L)_{s}$ and $(K^{\perp} \cap K \cap L)_{s}$ does not depend on $s \in S$. 
\end{enumerate}
$Q(\Re)_{\ast}(L)$ is precisely the image of $K^{\perp} \cap L$ under the map $\chi_{E'} \circ \natural|_{K^{\perp}}$.
\end{tvrz}
\begin{proof}
One has to find out when do the Courant algebroid relations $\{0\} \times L: \{0\} \dra E$ (supported on $\{\ast\} \times S$) and $Q(\Re): E \dra E'$ (supported on $\gr(\varpi)$) compose cleanly. The conditions $(ii)$ in Proposition \ref{tvrz_RdiamondR} and in Proposition \ref{tvrz_RcircR} are used to do so. First, one has 
\begin{equation}
\gr(\varpi) \circ ( \{\ast\} \times S) = \{\ast \} \times \varpi(S). 
\end{equation}
This is a submanifold, iff $\varpi(S)$ is a submanifold of $M$. Next, one finds 
\begin{equation}
\gr(\varpi) \diamond( \{ \ast\} \times S ) = \{ \ast \} \times \{ (s,s,\varpi(s)) \; | \; s \in S \}.
\end{equation}
This is always a submanifold diffeomorphic to $S$ and it is easy to see that $(\{\ast\} \times S) \times \gr(\varpi)$ and $\{\ast\} \times \Delta(P) \times M$ intersect cleanly. The map $\pi: \gr(\varpi) \diamond (\{\ast \} \times S) \rightarrow \gr(\varpi) \circ (\{\ast \} \times S)$ given by $\pi(\ast,s,s,\varpi(s)) = (\ast,\varpi(s))$ must be a surjective submersion. Obviously, it is surjective. It is a submersion, iff $(T_{s}\varpi)(T_{s}S) = T_{\varpi(s)}(\varpi(S))$ for all $s \in S$. We see that the condition $(i)$ of this proposition are equivalent to the conditions on the supports of the composed involutive structures. Now, write $\hat{s} = (\ast,s,s,\varpi(s))$. One has 
\begin{equation}
(Q(\Re) \diamond ( \{0\} \times L))_{\hat{s}} = \{0\} \times \{ (e,e, \natural_{E'}(e)) \; | \; e \in (K^{\perp} \cap L)_{s} \},
\end{equation}
where $\natural_{E'} = \chi_{E'} \circ \natural|_{K^{\perp}}$. This vector space has the same dimension for all $\hat{s} \in \gr(\varpi) \diamond (\{ \ast \} \times S)$, iff $(K^{\perp} \cap L)_{s}$ has the same dimension for all $s \in S$. This sorts out the second part of the condition $(ii)$ in Proposition \ref{tvrz_RdiamondR}. Finally, we have to ensure that the linear map $p_{\hat{s}}(0,e,e,\natural_{E'}(e)) = (0, \natural_{E'}(e))$, where $e \in (K^{\perp} \cap L)_{s}$, has the same rank for all $\hat{s} \in \gr(\varpi) \diamond ( \{\ast\} \times S)$. From the nullity-rank theorem, it suffices to prove that the dimension of the kernel of $\natural_{E'}$ restricted to $(K^{\perp} \cap L)_{s}$ does not depend on $s \in S$. But this kernel is the subspace $(K^{\perp} \cap K \cap L)_{s}$. 
\end{proof}

\begin{rem}
The condition $(i)$ is satisfied automatically whenever $S \subseteq P$ is a $G$-invariant submanifold. Indeed, then $S = \varpi^{-1}(\varpi(S))$ and $\varpi(S)$ is a submanifold of $M$ which can be identified with the quotient $S/G$. The quotient map $S \rightarrow S/G$ then corresponds to the map $\varpi|_{S}: S \rightarrow \varpi(S)$. In particular, it is a surjective submersion and $(T_{s}\varpi)(T_{s}S) = T_{\varpi(s)}( \varpi(S))$ for all $s \in S$. 
\end{rem}
\subsection{Morphism of reduced Courant algebroids} \label{subsec_tworeduced}
Let us now consider a particular example of a Courant algebroid reduction described in the previous subsection. Suppose $E$ is a $G$-equivariant Courant algebroid $(E,\rho,\<\cdot,\cdot\>,[\cdot,\cdot],\Re)$ over a principal $G$-bundle $\varpi: P \rightarrow M$. Furthermore, we assume that $\g = \Lie(G)$ is a quadratic Lie algebra, i.e. equipped with a non-degenerate symmetric bilinear and $\ad$-invariant form $\<\cdot,\cdot\>_{\g}$. Suppose that it coincides with $(\cdot,\cdot)_{\g}$ defined by (\ref{eq_Rinducedpair}). This implies that $K \cap K^{\perp} = 0$ and the reduced Courant algebroid is just $E' = K^{\perp} / G$.

Let $H \subseteq G$ be a closed and connected Lie subgroup of $G$ and let $\h = \Lie(H)$ be the corresponding Lie subalgebra of $\g$. Let $\Re_{0}: \h \rightarrow \Gamma(E)$ be the restriction of $\Re$. If follows that $(E,\rho,\<\cdot,\cdot\>,[\cdot,\cdot],\Re_{0})$ is an $H$-equivariant Courant algebroid. Let $K_{0} = \Re_{0}(P \times \h)$. We thus have another reduced Courant algebroid $E'_{0}$ over the base $N := P/H$ given by
\begin{equation}
E'_{0} = \frac{K_{0}^{\perp} / H}{(K_{0} \cap K_{0}^{\perp}) / H}.
\end{equation}
In this subsection, we will construct a Courant algebroid morphism $R(H): E'_{0} \rat E'$ over a certain smooth map $\varphi: N \rightarrow M$. Let us start with the following important observation:
\begin{lemma} 
Let $K'_{0} = \Re(P \times \h^{\perp})$, where $\h^{\perp} \subseteq \g$ is the orthogonal complement of $\h$ with respect to $\<\cdot,\cdot\>_{\g}$. Then $K'_{0}$ is an $H$-invariant subbundle of $E$ with respect to the action induced by $\Re_{0}$. Moreover, there is a canonical decomposition
\begin{equation} \label{eq_K0perp}
K_{0}^{\perp} = K^{\perp} \oplus K'_{0}.
\end{equation}
\end{lemma}
\begin{proof}
Let $\Phi \in C^{\infty}(P, \h^{\perp})$. To prove that $K'_{0}$ is $H$-invariant, it suffices to prove that $x \btr \Re(\Phi) \in \Gamma(K'_{0})$ for all $x \in \h$. To see this, note that $[\h,\h^{\perp}]_{\g} \subseteq \h^{\perp}$. This follows immediately from the fact that $\h$ is a Lie subalgebra and $\<\cdot,\cdot\>_{\g}$ is $\ad$-invariant. We can then write 
\begin{equation} \label{eq_xbtrRphi}
x \btr \Re(\Phi) = \Re( \Li{\#{x}}\Phi + [x,\Phi]_{\g}) \in \Gamma(K'_{0}), 
\end{equation}
since the above observation implies $\Li{\#{x}}\Phi + [x,\Phi]_{\g} \in C^{\infty}(P, \h^{\perp})$. Hence $K'_{0}$ is $H$-invariant. 

As $\Re_{0} = \Re|_{\h}$, one has $K_{0} \subseteq K$ and consequently $K^{\perp} \subseteq K_{0}^{\perp}$. Clearly $K'_{0} \subseteq K_{0}^{\perp}$. This proves the inclusion $K^{\perp} + K'_{0} \subseteq K^{\perp}_{0}$. The non-degeneracy of $\<\cdot,\cdot\>_{\g}$ implies $K^{\perp} \cap K'_{0} = 0$, hence the sum is in fact direct. Finally, one has 
\begin{equation}
\rk(K^{\perp}_{0}) = \rk(E) - \dim(\h) = \rk(E) - \dim(\g) + \dim(\g) - \dim(\h) = \rk(K^{\perp}) + \rk(K'_{0}).
\end{equation}
The equation (\ref{eq_K0perp}) follows and the proof is finished. 
\end{proof}
Before proceeding further, let us establish some notation. Let $\q = \h^{\perp} / (\h \cap \h^{\perp})$. Note that both $\h$ and $\h^{\perp}$ are invariant subspaces with respect to the adjoint action of subgroup $H$ on Lie algebra $\g$. In particular, there is a canonical action of $H$ on $\q$ which we denote as $\Ad$. Consequently, there is the associated vector bundle $P \times_{\Ad} \q$ over $N$.  

\begin{tvrz}
\begin{enumerate}[(i)]
\item There is a canonical surjective submersion $\varphi: N \rightarrow M$;
\item There is a fiber-wise bijective vector bundle map $\fPsi_{0}: K^{\perp}/H \rightarrow E'$ over $\varphi: N \rightarrow M$; 
\item There is a vector bundle isomorphism $\fPsi_{1}: K^{\perp}/H \oplus (P \times_{\Ad} \q) \rightarrow E'_{0}$ over the identity $1_{N}$.
\end{enumerate}
\end{tvrz}
\begin{proof}
By construction, $P$ can be viewed as a principal $H$-bundle $\varpi_{0}: P \rightarrow N$. The map $\varphi$ then completes the commutative triangle
\begin{equation}
\begin{tikzcd}
P \arrow{d}{\varpi_{0}} \arrow{rd}{\varpi}& \\
N \arrow[dashed]{r}{\varphi} & M 
\end{tikzcd}.
\end{equation}
Stated differently, each $H$-orbit of $P$ is a subset of the unique $G$-orbit of $P$. This defines a map $\varphi$ of the respective orbit spaces $N = P/H$ and $M = P/G$. It is a smooth surjective submersion as both $\varpi$ and $\varpi_{0}$ are smooth surjective submersions. This proves $(i)$. The vector bundle map $\fPsi_{0}$ in $(ii)$ is defined similarly as the map of quotients:
\begin{equation}
\begin{tikzcd}
K^{\perp} \arrow{d}{\natural_{0}} \arrow{rd}{\natural} & \\
K^{\perp}/H \arrow[dashed]{r}{\fPsi_{0}} & K^{\perp}/G \equiv E'
\end{tikzcd}
\end{equation}
It is easy to see that it is a fiber-wise bijective vector bundle map over $\varphi$. 

Now, note that $K_{0} \cap K_{0}^{\perp} \subseteq K'_{0}$. Indeed, one has $K_{0} \cap K_{0}^{\perp} = \Re_{0}(P \times \mathfrak{n}_{0})$, where $\mathfrak{n}_{0}$ is the kernel of the bilinear form $(\cdot,\cdot)_{\h} = \< \cdot,\cdot\>_{\g}|_{\h \times \h}$. Obviously $\mathfrak{n}_{0} \subseteq \h^{\perp}$ and the observation follows. It thus makes sense to consider a quotient map 
\begin{equation}
q'_{0}: K'_{0} / H \rightarrow \frac{K'_{0}/H}{(K_{0} \cap K_{0}^{\perp}) / H}
\end{equation}
The vector bundle on the right-hand side can be now shown isomorphic to $P \times_{\Ad} \q$. It follows from (\ref{eq_xbtrRphi}) that $\Re$ induces a vector bundle map $\Re: P \times_{\Ad} \h^{\perp} \rightarrow K'_{0} / H$. Let $\chi_{\q}: P \times_{\Ad} \h^{\perp} \rightarrow P \times_{\Ad} \q$ be the a map induced by the quotient map $\h^{\perp} \rightarrow \q$. It follows there is a vector bundle map $\fPsi_{\q}$ completing the commutative square
\begin{equation}
\begin{tikzcd}
P \times_{\Ad} \h^{\perp} \arrow{r}{\Re} \arrow{d}{\chi_{\q}} & K'_{0} / H \arrow{d}{q'_{0}} \\
P \times_{\Ad} \q \arrow[dashed]{r}{\fPsi_{\q}} & \frac{K'_{0} / H}{(K_{0} \cap K_{0}^{\perp}) / H}
\end{tikzcd}
\end{equation}
It is straightforward to see that $\fPsi_{\q}$ is a vector bundle isomorphism. To prove $(iii)$, note that the decomposition (\ref{eq_K0perp}) gives us a vector bundle map $\tI_{0}: K^{\perp}/ H \oplus K'_{0} / H \rightarrow K^{\perp}_{0} / H$. Finally, define $\fPsi_{1}$ to be the completion of the commutative diagram
\begin{equation}
\begin{tikzcd}
K^{\perp} / H \oplus K'_{0}/H \arrow{r}{\tI_{0}} \arrow{d}{1 \times (\fPsi_{\q}^{-1} \circ q'_{0})} & K_{0}^{\perp} / H \arrow{d}{\chi_{E'_{0}}} \\ 
K^{\perp} / H \oplus (P \times_{\Ad} \q) \arrow[dashed]{r}{\fPsi_{1}} & E'_{0}
\end{tikzcd}.
\end{equation}
Clearly, $\fPsi_{1}$ is a well-defined vector bundle map. Let us prove that it is fiber-wise injective. Let $\fPsi_{1}(\psi, \fPsi_{\q}^{-1}(q'_{0}(\psi'))) = 0$ for $\psi \in \Gamma_{H}(K^{\perp})$ and $\psi' \in \Gamma_{H}(K'_{0})$. This is equivalent to $\chi_{E'_{0}}( \psi + \psi') = 0$, that is $\psi + \psi' \in \Gamma_{H}(K_{0} \cap K_{0}^{\perp})$. Hence $\psi \in \Gamma_{H}(K'_{0} \cap K^{\perp}) = 0$ and consequently, $\psi' \in \Gamma_{H}(K_{0} \cap K_{0}^{\perp})$ and $q'_{0}(\psi') = 0$. This proves the claim. Finally, both vector bundles have the same rank, hence $\fPsi_{1}$ is an isomorphism. This finishes the proof.
\end{proof}
We are now ready to define the relation $R(H) \subseteq E'_{0} \times \ol{E}'$. Its support would be $\gr(\varphi)$, where $\varphi: N \rightarrow M$ is the map obtained in $(i)$ of the previous proposition. For each $n \in N$, set
\begin{equation}
R(H)_{(n,\varphi(n))} := \{ ( \fPsi_{1}(\hat{e},0), \fPsi_{0}(\hat{e})) \; | \; \hat{e} \in (K^{\perp} / H)_{n} \}.
\end{equation}
This defines a subbundle of $E'_{0} \times \ol{E}'$ supported on $\gr(\varphi)$. 
\begin{lemma}
$R(H)$ is isotropic. Let $(p_{0},q_{0})$ be the signature of $\<\cdot,\cdot\>_{\g}$ and let $(p_{\h},q_{\h},k_{\h})$ be the inertia of its restriction $(\cdot,\cdot)_{\h}$ to $\h$. Then $R(H)$ is maximally isotropic, iff 
\begin{equation} \label{eq_R(H)maxiso}
k_{\h} = \min\{ p_{0} - p_{\h}, q_{0} - q_{\h} \}.
\end{equation}
Note that $k_{\h} = \dim(\h \cap \h^{\perp})$. 
\end{lemma}
\begin{proof}
Let $\hat{e}, \hat{e}' \in (K^{\perp} / H)_{n}$. Fix $p \in \varpi_{0}^{-1}(n)$. Then there are unique elements $e,e' \in K^{\perp}_{p}$, such that $\hat{e} = \natural_{0}(e)$ and $\hat{e}' = \natural_{0}(e')$, respectively. Then one can write 
\begin{equation}
(\fPsi_{1}(\hat{e},0), \fPsi_{0}(\hat{e})) = ( \chi_{E'_{0}}( \natural_{0}(e)), \natural(e)), \; \; (\fPsi_{1}(\hat{e}',0), \fPsi_{0}(\hat{e}')) = ( \chi_{E'_{0}}( \natural_{0}(e')), \natural(e')). 
\end{equation}
Consequently, one obtains 
\begin{equation}
\begin{split}
\< (\fPsi_{1}(\hat{e},0), \fPsi_{0}(\hat{e})), (\fPsi_{1}(\hat{e}',0), \fPsi_{0}(\hat{e}')) \> = & \ \< \chi_{E'_{0}}( \natural_{0}(e)), \chi_{E'_{0}}(\natural_{0}(e')) \>_{E'_{0}} - \< \natural(e), \natural(e') \>_{E'} \\
= & \ \<e,e'\>_{E} - \<e,e'\>_{E} = 0. 
\end{split}
\end{equation}
This proves that $R(H)$ is isotropic. Now, let $(p,q)$ be the signature of the fiber-wise metric $\<\cdot,\cdot\>_{E}$, and let $(p',q')$ and $(p'_{0},q'_{0})$ denote the signature of the fiber-wise metric on $E'$ and $E'_{0}$, respectively. It follows from Proposition \ref{tvrz_inertias} that the signatures are related by
\begin{equation}
p' = p - p_{0}, \;\; q' = q - q_{0}, \; \; p'_{0} = p - p_{\h} - k_{\h}, \; \; q'_{0} = q - q_{\h} - k_{\h}. 
\end{equation} 
Note that $\rk(R(H)) = \rk(K^{\perp}) = p + q - p_{0} - q_{0}$. It follows that $R(H)$ is maximally isotropic, iff
\begin{equation}
\begin{split}
p + q - p_{0} - q_{0} = & \ \min\{ p - p_{0} + q - q_{\h} - k_{\h}, q - q_{0} + p - p_{\h} - k_{\h} \} \\
= & \ p + q - k_{\h}  - \max \{ p_{0} + q_{\h}, q_{0} + p_{\h} \}.
\end{split}
\end{equation}
This gives us the condition $p_{0} + q_{0} - k_{\h} = \max \{ p_{0} + q_{\h}, q_{0} + p_{\h} \}$. One can rewrite it as (\ref{eq_R(H)maxiso}). 
\end{proof}
\begin{example} \label{ex_qiszero}
The most interesting case is the one where $\q = 0$. Then $R(H)$ is a graph of a fiber-wise bijective vector bundle map. This happens whenever $\h^{\perp} = \h \cap \h^{\perp}$, that is $\h^{\perp} \subseteq \h$. In other words, the Lie subalgebra $\h$ is coisotropic with respect to $\<\cdot,\cdot\>_{\g}$. It follows from (\ref{eq_R(H)maxiso}) that in this case, $R(H)$ is maximally isotropic (in fact, Lagrangian). This is because one has $k_{\h} = \dim(\h^{\perp})$ and $p_{0} = p_{\h} + \dim(\h^{\perp})$, $q_{0} = q_{\h} + \dim(\h^{\perp})$.  
\end{example}
\begin{lemma}
$R(H)^{\perp}$ is compatible with the anchor. In other words, $R(H)$ is an almost involutive structure on $E'_{0} \times \ol{E}'$ supported on $\gr(\varpi)$. 
\end{lemma}
\begin{proof}
For each $n \in N$, fix $p \in \varpi_{0}^{-1}(n)$ and write a general element of $R(H)_{(n,\varphi(n))}$ as a pair $(\chi_{E'_{0}}(\natural_{0}(e)), \natural(e))$ for $e \in K^{\perp}_{p}$, see the proof of the previous lemma. Applying the anchor now gives 
\begin{equation}
(\rho'_{0} \times \rho')(\chi_{E'_{0}}(\natural_{0}(e)), \natural(e)) = ( (T_{p}\varpi_{0})(\rho(e)), (T_{p} \varpi)(\rho(e))).
\end{equation}
But by definition, we have $\varpi = \varphi \circ \varpi_{0}$, whence $T_{p}\varpi = T_{n}\varphi \circ T_{p} \varpi_{0}$. This shows that the right-hand side is in $T_{(n,\varphi(n))}(\gr(\varphi)) = \gr( T_{n} \varphi)$. Hence $R(H)$ is compatible with the anchor. 

Next, it is not difficult to show that for each $n \in N$, the fiber of the orthogonal complement at $(n,\varphi(n))$ reads
\begin{equation}
R(H)^{\perp}_{(n,\varphi(n))} = R(H)_{(n,\varphi(n))} + \{ ( \chi_{E'_{0}}( \natural_{0}(e')), 0_{\varphi(n)}) \; | \; e' \in (K'_{0})_{p} \},
\end{equation}
where $p \in \varpi_{0}^{-1}(n)$ is fixed. Applying the anchor on the elements of the second summand gives 
\begin{equation} \label{eq_thoR(H)perp}
(\rho'_{0} \times \rho')(\chi_{E'_{0}}(\natural_{0}(e')), 0_{\varphi(n)}) = ( (T_{p}\varpi_{0})(\rho(e')), 0_{\varphi(n)})
\end{equation}
As $e' \in (K'_{0})_{p}$, we may write $e' = \Re(x)(p)$ for some $x \in \h^{\perp}$, hence $\rho(e') = \#_{p}(x)$. But then 
\begin{equation}
(T_{n}\varphi)( (T_{p}\varpi_{0})(\rho(e'))) = (T_{p} \varpi)( \#_{p}(x)) = 0_{\varpi(p)} = 0_{\varphi(n)}. 
\end{equation}
This shows that the right-hand side of (\ref{eq_thoR(H)perp}) is in $\gr(T_{n}\varphi)$ and the compatibility of $R(H)^{\perp}$ with the anchor follows. We have already proved that $R(H)$ is isotropic, hence we can conclude that $R(H)$ is an almost involutive structure in $E'_{0} \times \ol{E}'$. 
\end{proof}
\begin{tvrz}
The subbundle $R(H)$ is an involutive structure in $E'_{0} \times \ol{E}'$, hence it defines a Courant algebroid morphism $R(H): E'_{0} \rat E'$ over $\varphi$. 
\end{tvrz}
\begin{proof}
Thanks to the previous lemma, we know that $R(H)$ is an almost involutive structure. We may thus employ Proposition \ref{tvrz_invongen} and verify its involutivity on the sections which restrict to local generators. Now, to any section $\psi' \in \Gamma(E')$, there is a unique section $\varphi^{\ast}(\psi') \in \Gamma(K^{\perp}/H)$, such that $\varphi^{\ast}(\psi') \sim_{\fPsi_{0}} \psi'$. Less formally, $\psi'$ can be interpreted as a $G$-invariant section of $K^{\perp}$. By definition, it is also $H$-invariant, hence defines a section $\varphi^{\ast}(\psi')$ of $K^{\perp}/H$. 

It suffices to consider the sections of the form $(\fPsi_{1}(\varphi^{\ast}(\psi'),0), \psi') \in \Gamma(E'_{0} \times \ol{E}', R(H))$. It then follows from the definitions of the involved brackets and maps that
\begin{equation}
[(\fPsi_{1}(\varphi^{\ast}(\psi'),0), \psi'), (\fPsi_{1}(\varphi^{\ast}(\phi'),0), \phi') ] = ( \fPsi_{1}( \varphi^{\ast}([\psi',\phi']_{E'}),0), [\psi',\phi']_{E'}),
\end{equation}
for all $\psi',\phi' \in \Gamma(E')$. The right-hand side in $\Gamma(E'_{0} \times \ol{E}', R(H))$ and the statement follows. 
\end{proof}
\begin{example}[\textbf{Poisson--Lie T-duality}] \label{ex_PLT}
The Courant algebroid morphism of this subsection is central for the geometrical interpretation of Poisson--Lie T-duality \cite{Severa:2015hta, Severa:2018pag} and also the related work \cite{Jurco:2017gii, Jurco:2019tgt}. One adds the following requirements:
\begin{enumerate}[(i)]
\item $E$ is an exact Courant algebroid admitting a $G$-equivariant isotropic splitting, that is there exists a vector bundle map $\sigma: TP \rightarrow E$, such that $\rho \circ \sigma = 1$, $\sigma(TP)$ is isotropic in $E$ and the map $\sigma$ is $G$-equivariant;
\item $\<\cdot,\cdot\>_{\g}$ has a split signature and the subalgebra $\h$ is Lagrangian, that is $\h = \h^{\perp}$. 
\end{enumerate}
As noted in Example \ref{ex_qiszero}, in such a situation, we have $R(H) = \gr(\fPsi_{H})$ for a classical Courant algebroid morphism $\fPsi_{H}: E'_{0} \rightarrow E'$ over $\varphi$. Note that in this case, $E'_{0}$ is exact.

Now, one can consider any other closed Lie subgroup $H' \subseteq G$, such that the corresponding Lie subalgebra $\h' = \Lie(H')$ is Lagrangian. There is thus another reduced Courant algebroid $E'_{1}$ over $N' = P / H'$ together with a classical Courant algebroid morphism $\fPsi_{H'}: E'_{1} \rightarrow E'$ over $\varphi'$.  

We then construct a Courant algebroid relation $R_{H,H'}: E'_{0} \dra E'_{1}$ of the two exact Courant algebroids. Define it as the composition $R_{H,H'} := \gr(\fPsi_{H'})^{T} \circ \gr(\fPsi_{H})$. 

Let us show that the involved relations compose cleanly. As a subset of $E'_{0} \times \ol{E}'_{1}$, one has 
\begin{equation}
R_{H,H'} = \{ (e'_{0}, e'_{1}) \in E'_{0} \times \ol{E}'_{1} \; | \; \fPsi_{H}(e'_{0}) = \fPsi_{H}(e'_{1}) \} \equiv E'_{0} \times_{E'} \ol{E}'_{1}. 
\end{equation}
Both vector bundle maps are fiber-wise bijective and over surjective submersions, hence they are surjective submersions. The corresponding fibered product is thus a closed submanifold of $E'_{0} \times \ol{E}'_{1}$. Moreover, one has 
\begin{equation}
\gr(\fPsi_{H'})^{T} \diamond \gr(\fPsi_{H}) = \{ (e'_{0}, \fPsi_{H}(e'_{0}), \fPsi_{H'}(e'_{1}), e'_{1} ) \; | \; (e'_{0},e'_{1}) \in R_{H,H'} \}
\end{equation}
This is a submanifold diffeomorphic to $R_{H,H'}$ and it is not difficult to see that the $\gr(\fPsi_{H}) \times \gr(\fPsi_{H'})^{T}$ and $E'_{0} \times \Delta(E') \times \ol{E}'_{1}$ intersect cleanly. Moreover, the map $p: \gr(\fPsi_{H'})^{T} \diamond \gr(\fPsi_{H}) \rightarrow R_{H,H'}$ is a diffeomorphism. The claim now follows from Proposition \ref{tvrz_RdiamondR} and Proposition \ref{tvrz_RcircR}. Theorem \ref{thm_composition} says that $R_{H,H'}: E'_{0} \dra E'_{1}$ is a Courant algebroid relation supported on $\gr(\varphi')^{T} \circ \gr(\varphi) = N \times_{M} M'$. This relation \textbf{is the Poisson--Lie T-duality}.  

Now, let $L \subseteq E'_{0}$ be an involutive structure supported on $S \subseteq N'$. One can find the conditions on the existence of the pushforward involutive structure $L' := (R_{H,H'})_{\ast}(L) \subseteq E'_{1}$ in the sense of Definition \ref{def_pushpull}. The only requirements are on the base manifold, namely 
\begin{enumerate}[(i)]
\item the base $S' = \varphi'^{-1}(\varphi(S))$ of $L'$ must be a submanifold of $N'$;
\item for every $(s,s') \in S \times_{M} S'$, one has $T_{s'}S' = (T_{s'}\varphi')^{-1}( (T_{s}\varphi)(T_{s}S))$. 
\end{enumerate}
Note that one has $L' = \fPsi_{H'}^{-1}( \fPsi_{H}(L))$. Pullbacks can be discussed easily as $(R_{H,H'})^{\ast}(L') = (R_{H',H})_{\ast}(L')$ for any involutive structure of $L' \subseteq E'_{1}$. 

Note that as $\varphi'$ is a surjective submersion, it suffices to assume that $\varphi(S)$ is a submanifold of $M$, such that for all $s \in S$, one has $T_{\varphi(s)}(\varphi(S)) = (T_{s}\varphi)(T_{s}S)$. This is equivalent to the assumption that the pushforward $\gr(\fPsi_{H})_{\ast}(L)$ exists. In particular, one can achieve this by considering $S = \varphi^{-1}(S_{0})$ for a submanifold $S_{0} \subseteq M$. Then $(R_{H,H'})_{\ast}$ maps the space of involutive structures in $E'_{0}$ supported on $\varphi^{-1}(S_{0})$ bijectively to the space of involutive structures in $E'_{1}$ supported on $\varphi'^{-1}(S_{0})$. 
\end{example}
\begin{rem}
Previous paragraphs together with Example \ref{ex_qiszero} suggest that the generalization of Poisson--Lie T-duality can be worked out for the case where $\h \subseteq \g$ and $\h' \subseteq \g$ are assumed to be merely coisotropic with respect to the bilinear form $\<\cdot,\cdot\>_{\g}$ (of any signature). In this case $E'_{0}$ and $E'_{1}$ do not have to be exact Courant algebroids. This is certainly an idea for the future investigation. 
\end{rem}
\begin{example} \label{ex_RHcomQRe0}
Let us now give an explicit example of two Courant algebroid relations which are maximally isotropic (Dirac structures), whereas their composition is not. We can consider the relations $Q(\Re_{0}): E \dra E'_{0}$ and $R(H): E'_{0} \rightarrow E'$. It is not difficult to see that they compose cleanly, and for each $p \in P$, one finds
\begin{equation} \label{eq_RHcomQRe0}
(R(H) \circ Q(\Re_{0}))_{(p,\varpi(p))} = (Q(\Re))_{(p,\varpi(p))} \oplus (K_{0} \cap K_{0}^{\perp})_{p} \times 0_{\varpi(p)}.
\end{equation}
In particular, see that $R(H) \circ Q(\Re_{0}) = Q(\Re)$, iff the restricted fiber-wise bilinear form $(\cdot,\cdot)_{\h}$ is non-degenerate. Now, let $G = H \times K$ be a direct product of compact Lie groups $H$ and $K$. Then $\g = \h \oplus \frk$, where $\g = \Lie(G)$, $\h = \Lie(H)$ and $\frk = \Lie(K)$ are the respective Lie algebras. Let $\<\cdot,\cdot\>_{\g}$ be a product bilinear form that restricts to the Killing form of $\h$ and to the opposite of the Killing form of $\frk$. In particular, the bilinear form $(\cdot,\cdot)_{\h}$ is negative-definite, hence non-degenerate. The signature of $\<\cdot,\cdot\>_{\g}$ is $(\dim(\frk),\dim(\h))$. It follows from Lemma \ref{lem_QRemaxiso} that $Q(\Re_{0})$ is maximally isotropic. Moreover, the equation (\ref{eq_R(H)maxiso}) holds as $k_{\h} = 0$ and $q_{0} = q_{\h}$. Hence also $R(H)$ is maximally isotropic. Finally, it follows from (\ref{eq_RHcomQRe0}) that 
\begin{equation}
R(H) \circ Q(\Re_{0}) = Q(\Re).
\end{equation}
But Lemma \ref{lem_QRemaxiso} shows that for $\dim(\frk) > 0$, $Q(\Re)$ \textit{is not} maximally isotropic.
\end{example}
\section{Generalized isometries} \label{sec_geniso}
The following notion is based on unpublished notes kindly provided to us by Pavol Ševera. Note that the compatibility of generalized metrics with Courant algebroid automorphisms was also considered in \cite{rubio2016lie}. Recall that a \textbf{generalized metric} on a quadratic vector bundle\footnote{A vector bundle $E$ together with a fiber-wise metric $\<\cdot,\cdot\>$.} $(E,\<\cdot,\cdot\>)$ over $M$ is a vector bundle endomorphism $\tau: E \rightarrow E$ over $1_{M}$ satisfying $\tau^{2} = 1$, such that the formula $\gm(\psi,\psi') = \<\psi,\tau(\psi')\>$, where $\psi,\psi' \in \Gamma(E)$, defines a positive-definite fiber-wise metric on $E$. Equivalently, this corresponds to the choice of a maximal positive subbundle $V_{+} \subseteq E$ with respect to $\<\cdot,\cdot\>$. $V_{+}$ plays the role of a $+1$ eigenbundle of $\tau$. For details, see Section 3 of \cite{Jurco:2016emw}. Note that on every quadratic vector bundle, there exists a generalized metric, see Corollary \ref{cor_genmetric}. 

Now, let $(E_{1},\rho_{1},\<\cdot,\cdot\>_{1},[\cdot,\cdot]_{1})$ and $(E_{2}, \rho_{2}, \<\cdot,\cdot\>_{2}, [\cdot,\cdot]_{2})$ be a pair of Courant algebroids over $M_{1}$ and $M_{2}$, respectively. Let $\tau_{1}$ and $\tau_{2}$ be a generalized metric on $E_{1}$ and $E_{2}$, respectively. One can then define a single vector bundle endomorphism $\tau := \tau_{1} \times \tau_{2}$ on $E_{1} \times E_{2}$. It is not difficult to see that $\tau$ defines a generalized metric on the product Courant algebroid $E_{1} \times E_{2}$. 

\begin{definice}
Let $R: E_{1} \dra E_{2}$ be a Courant algebroid relation. Suppose $\tau$ is the vector bundle endomorphism defined above. 

We say that \textbf{$R$ is a generalized isometry with respect to $\tau_{1}$ and $\tau_{2}$}, if $\tau(R) = R$. 
\end{definice}
\begin{rem}
One can ask whether it would not be more natural to consider the map $\tau' = \tau_{1} \times (-\tau_{2})$ which forms a generalized metric on the product Courant algebroid $E_{1} \times \ol{E}_{2}$. However, it turns out that this is not a particularly good idea. 

Indeed, consider any quadratic vector bundle $(E,\<\cdot,\cdot\>)$ together with a generalized metric $\tau$ preserving some isotropic subbundle $R \subseteq E$. $\tau$ provides a decomposition $E = V_{+} \oplus V_{-}$ onto its $\pm 1$ eigenbundles $V_{\pm}$. Let $R_{\pm} = p_{\pm}(R)$, where $p_{\pm}: E \rightarrow V_{\pm}$ are the projections. Note that $R_{\pm}$ are subbundles. One has $\ker(p_{+}|_{R}) = R \cap V_{-} = 0$, as $R$ is isotropic and $\<\cdot,\cdot\>$ is negative definite on $V_{-}$. Hence $R_{+}$ is a subbundle of $V_{+}$ isomorphic to $R$. The same argument works for $R_{-}$. We may thus view $R$ as a subbundle of $R_{+} \oplus R_{-}$ which has a trivial intersection with both $R_{+} \oplus 0$ and $0 \oplus R_{-}$. As $\rk(R) = \rk(R_{+}) = \rk(R_{-})$, there is a unique vector bundle isomorphism $\F: R_{+} \rightarrow R_{-}$, such that $R = \gr(\F)$. Every element of $R$ can be then written as $(e,\F(e))$ for $e \in R_{+}$. By assumption, we have $\tau(e,\F(e)) \in R$. But $\tau(e,\F(e)) = (e,-\F(e))$, which forces $\F(e) = 0$ for all $e \in R_{+}$. Hence $\F = 0$ and consequently $R = 0$. 

This observation shows that if we would consider $\tau'$ instead of $\tau$, the only generalized isometry would be a trivial relation\footnote{Which, to make things worse, is  usually not a Courant algebroid relation.} $0_{S}: E_{1} \dra E_{2}$. 
\end{rem}
\begin{example}
Let $\F: E_{1} \rightarrow E_{2}$ be a classical Courant algebroid morphism over $\varphi: M_{1} \rightarrow M_{2}$, see Subsection \ref{subsec_graph}. Let us examine when $\gr(\F)$ defines a generalized isometry of $\tau_{1}$ and $\tau_{2}$.  Let $(e_{1},\F(e_{1})) \in \gr(\F)$. Then $\tau(e_{1},\F(e_{1})) = (\tau_{1}(e_{1}), \tau_{2}(\F(e_{1})))$ is in $\gr(\F)$ again, iff $\tau_{2}(\F(e_{1})) = \F(\tau_{1}(e_{1}))$. This has to hold for all $e_{1} \in E_{1}$ and we obtain the condition
\begin{equation} \label{eq_Fintertwines}
\tau_{2} \circ \F = \F \circ \tau_{1}. 
\end{equation}
In terms of the induced positive-definite fiber-wise metrics $\gm_{1}$ and $\gm_{2}$, this can be equivalently restated as the condition
\begin{equation}
\gm_{2}( \F(e_{1}), \F(e'_{1})) = \gm_{1}(e_{1},e_{1}),
\end{equation}
for all $e_{1},e'_{1} \in (E_{1})_{m_{1}}$ and all $m_{1} \in M_{1}$. This justifies the name generalized isometry. 
\end{example}
Now, let us show that if $R$ is a Courant algebroid morphism, there are some serious restrictions on $R$ for it to be a generalized isometry. 
\begin{tvrz} \label{tvrz_genisoCAmorph}
Let $R: E_{1} \rat E_{2}$ be a Courant algebroid morphism over $\varphi: M_{1} \rightarrow M_{2}$, such that $R$ is a generalized isometry of $\tau_{1}$ and $\tau_{2}$. Let $p_{i}: E_{1} \times \ol{E}_{2} \rightarrow E_{i}$ be the projections, $i \in \{1,2\}$. 

Then $K_{1} = p_{1}(R)$ is a subbundle of $E_{1}$. There exists a fiber-wise injective vector bundle map $\F: K_{1} \rightarrow E_{2}$ over $\varphi$, such that for each $m_{1} \in M_{1}$, the fiber of $R$ over $(m_{1},\varphi(m_{1}))$ has the form
\begin{equation} \label{eq_Rformisometry}
R_{(m_{1},\varphi(m_{1}))} = \{ (e_{1}, \F(e_{1})) \; | \; e_{1} \in (K_{1})_{m_{1}} \}.
\end{equation}
The subbundle $K_{1}$ is invariant with respect to $\tau_{1}$ and 
\begin{equation} \label{eq_genisointertw}
\F \circ \tau_{1}|_{K_{1}} = \tau_{2} \circ \F. 
\end{equation}
\end{tvrz}
\begin{proof}
First, let us argue that $K_{1}$ is a subbundle. The restriction $p_{1}|_{R}: R \rightarrow E_{1}$ is a vector bundle map over a diffeomorphism $\pi_{1}|_{\gr(\varphi)}: \gr(\varphi) \rightarrow M_{1}$, where $\pi_{1}: M_{1} \times M_{2} \rightarrow M_{1}$ is the projection. It thus suffices to show that $p_{1}|_{R}$ is fiber-wise injective. One has
\begin{equation} 
\ker(p_{1}|_{R}) = R \cap (0_{M_{1}} \times E_{2}).
\end{equation}
Let us argue that this intersection must be trivial. Suppose $(0,e_{2}) \in R$. By assumption then also $(0, \tau_{2}(e_{2})) \in R$. But by definition, $R$ is isotropic. Hence  
\begin{equation}
0 = \< (0,e_{2}), (0,\tau_{2}(e_{2})) \> = - \<e_{2}, \tau_{2}(e_{2}) \>_{2} = - \gm_{2}(e_{2},e_{2}). 
\end{equation}
But $\gm_{2}$ is positive-definite, whence $e_{2} = 0$. Hence $K_{1}$ is a subbundle and $\rk(K_{1}) = \rk(R)$. Now, one can view $R$ as the subbundle of the direct sum $E_{1} \oplus \varphi^{!}(E_{2})$, where $\varphi^{!}(E_{2})$ is the pullback vector bundle of $E_{2}$ by $\varphi: M_{1} \rightarrow M_{2}$. Let $p'_{2}: E_{1} \oplus \varphi^{!}(E_{2}) \rightarrow \varphi^{!}(E_{2})$ be the projection. Using the same arguments as above, it follows that $K_{2} := p'_{2}(R) \subseteq \varphi^{!}(E_{2})$ is a subbundle and $\rk(K_{2}) = \rk(R)$. We have $R \subseteq K_{1} \oplus K_{2}$, $R \cap (K_{1} \oplus 0) = R \cap (0 \oplus K_{2}) = 0$ and $\rk(R) = \rk(K_{1}) = \rk(K_{2})$. There is thus a vector bundle isomorphism $\F^{!}: K_{1} \rightarrow K_{2}$, such that $R = \gr(\F^{!})$ as a subbundle of $K_{1} \oplus K_{2}$. Let $\F = \varphi^{!} \circ \F^{!}$ where  we view $\F^{!}$ as a map from $K_{1}$ to $\varphi^{!}(E_{2})$ and $\varphi^{!}: \varphi^{!}(E_{2}) \rightarrow E_{2}$ is the canonical fiber-wise bijective vector bundle map. Then $\F$ is a fiber-wise injective vector bundle map over $\varphi$ and $R$ has the form (\ref{eq_Rformisometry}). 

Next, let $e_{1} \in K_{1}$. There is thus $e_{2} \in E_{2}$, such that $(e_{1},e_{2}) \in R$. Hence also $\tau(e_{1},e_{2}) = (\tau_{1}(e_{1}),\tau_{2}(e_{2})) \in R$ and consequently $\tau_{1}(e_{1}) \in K_{1}$. This proves that $K_{1}$ is invariant with respect to $\tau_{1}$. The equation (\ref{eq_genisointertw}) then follows in the same way as (\ref{eq_Fintertwines}). 
\end{proof}
\begin{example} \label{ex_genisored}
Let us now examine the Courant algebroid morphism $Q(\Re): E \rat E'$ over $\varpi: P \rightarrow M$ obtained in Subsection \ref{subsec_reduction}. Suppose $\tau$ and $\tau'$ is a generalized metric on $E$ and $E'$, respectively. Proposition \ref{tvrz_genisoCAmorph} immediately tells us necessary conditions for $Q(\Re)$ to define a generalized isometry of $\tau$ and $\tau'$. First, one has $K_{1} = p_{1}(Q(\Re)) = K^{\perp}$. This is a subbundle of $E$, which is good. 

On the other hand, we see that the map $\F: K_{1} \rightarrow E'$ over $\varphi$ has the form $\F = \chi_{E'} \circ \natural|_{K^{\perp}}$. This map is fiber-wise injective, iff $\chi_{E'}: K^{\perp}/ G \rightarrow E'$ is the identity. In other words, $Q(\Re)$ can be a generalized isometry, only if $K \cap K^{\perp} = 0$. Equivalently, this means that the fiber-wise bilinear form $(\cdot,\cdot)_{\g}$ defined by (\ref{eq_Rinducedpair}) must be non-degenerate. 

Suppose this is the case. We can thus write $E = K^{\perp} \oplus K$. Write $\tau$ in the formal block form with respect to this decomposition as
\begin{equation}
\tau = \begin{pmatrix}
\tau_{0} & \tau_{2} \\
0 & \tau_{1}
\end{pmatrix},
\end{equation}
where the zero in the bottom-left corner makes the subbundle $K^{\perp}$ invariant with respect to $\tau$. It is an easy exercise to show that $\tau$ is a generalized metric, iff $\tau_{2} = 0$ and $\tau_{0}$ and $\tau_{1}$ is a generalized metric on $K^{\perp}$ and $K$, respectively. Now, note that $\F = \natural|_{K^{\perp}}$ and the equation (\ref{eq_genisointertw}) becomes $\natural|_{K^{\perp}} \circ \tau_{0} = \tau' \circ \natural|_{K^{\perp}}$. This is equivalent for $\tau_{0}$ to be $G$-equivariant with respect to the action $\frR$ of $G$ on $K^{\perp}$. It also shows that $\tau'$ is uniquely determined by $\tau_{0}$ and it is automatically a smooth generalized metric on $E'$. There is no restriction on the generalized metric $\tau_{1}$.  

Compare this to the assumptions we have made in Section 6.3 of our geometrical description of Kaluza--Klein reduction in \cite{Vysoky:2017epf}. The only difference is that we have required $\tau$ to be a $G$-equivariant map on $E$, which restricts $\tau_{1}$. We can thus view Kaluza--Klein reduction of supergravity as an example of a generalized isometry. 
\end{example}
\begin{example}
Consider the other extreme case of the Courant algebroid reduction, namely the isotropic $K$. In this case $K \cap K^{\perp} = K$. Let $\tau$ be any $G$-equivariant generalized metric on $E$. In particular, its eigenbundle $V_{+}$ is $G$-invariant. One may thus define $V'_{+} = \chi_{E'}( (V_{+} \cap K^{\perp}) / G)$. It follows from the proof of Proposition \ref{tvrz_coiso} that $V'_{+}$ is a maximal positive subbundle with respect to $\<\cdot,\cdot\>'$, hence defines a generalized metric $\tau'$ on $E'$. 

However, the Courant algebroid morphism $Q(\Re): E \rightarrow E'$ cannot be a generalized isometry. This is due to $\F = \chi_{E'} \circ \natural|_{K^{\perp}}$ having a non-trivial kernel, namely the subbundle $K$. This example shows that not every natural construction with generalized metrics is a generalized isometry. 
\end{example}
\begin{tvrz} \label{tvrz_genisocomp}
Let $(E_{1},\tau_{1})$, $(E_{2},\tau_{2})$ and $(E_{3},\tau_{3})$ be a triple of Courant algebroids equipped with generalized metrics. Suppose $R: E_{1} \dra E_{2}$ and $R': E_{2} \dra E_{3}$ is a pair of cleanly composing generalized isometries of the respective generalized metrics. 

Then $R' \circ R: E_{1} \dra E_{3}$ is a generalized isometry of $\tau_{1}$ and $\tau_{3}$. The transpose relation $R^{T}: E_{2} \dra E_{1}$ is a generalized isometry of $\tau_{2}$ and $\tau_{1}$. 

For any Courant algebroid $(E,\tau)$ equipped with a generalized metric, the relation $\Delta(E) = \gr(1_{E}): E \dra E$ is a generalized isometry of $\tau$ and $\tau$. 
\end{tvrz}
\begin{proof}
Let $(e_{1},e_{3}) \in R' \circ R$. There is thus $e_{2} \in E_{2}$, such that $(e_{1},e_{2}) \in R$ and $(e_{2},e_{3}) \in R'$. By assumption, we have $(\tau_{1}(e_{1}),\tau_{2}(e_{2})) \in R$ and $(\tau_{2}(e_{2}), \tau_{3}(e_{3})) \in R'$. But this proves that $(\tau_{1}(e_{2}), \tau_{3}(e_{3})) \in R' \circ R$. Hence $R' \circ R$ is a generalized isometry of $\tau_{1}$ and $\tau_{3}$. The claim about the transpose relation is obvious. Finally, if $(e,e) \in \Delta(E)$, then $(\tau(e),\tau(e)) \in \Delta(E)$, and the last claim of the proposition follows.
\end{proof}
\begin{example}
Let $R(H): E'_{0} \rat E'$ be the Courant algebroid morphism over $\varphi: N \rightarrow M$ described in Subsection \ref{subsec_tworeduced}. Let $\tau'_{0}$ and $\tau'$ be a generalized metric on $E'_{0}$ and $E'$, respectively. Using the notation of Proposition \ref{tvrz_genisoCAmorph}, we have $K_{1} = \fPsi_{1}( K^{\perp}/H \oplus 0)$ and the map $\F: K_{1} \rightarrow E'$ reads $\F( \fPsi_{1}(\hat{e},0)) = \fPsi_{0}(\hat{e})$ for all $\hat{e} \in K^{\perp}/H$. It is fiber-wise bijective and no issues in the likes of Example \ref{ex_genisored} occur. Now, the isomorphism $\fPsi_{1}$ and $\tau'_{0}$ induce a generalized metric $\hat{\tau}$ on the direct sum $K^{\perp} / H \oplus (P \times_{\Ad} \q)$ which must have the block form 
\begin{equation}
\hat{\tau} = \begin{pmatrix}
\hat{\tau}_{0} & 0 \\
0 & \hat{\tau}_{1} 
\end{pmatrix},
\end{equation}
where $\hat{\tau}_{0}$ and $\hat{\tau}_{1}$ is a generalized metric\footnote{One has to specify the pairings on these quadratic vector bundles. They are assumed to be the ones induced $E'_{0}$ using the isomorphism $\fPsi_{1}$.} on $K^{\perp} / H$ and $P \times_{\Ad} \q$, respectively. The intertwining property (\ref{eq_genisointertw}) then becomes $\tau' \circ \fPsi_{0} = \fPsi_{0} \circ \hat{\tau}_{0}$. This shows that $\hat{\tau}_{0}$ is uniquely determined by $\tau'$. In other words, note that $K^{\perp}/H$ may be identified with the pullback bundle $\varphi^{!}(E')$. The above condition then says that $\hat{\tau}_{0}$ has to be the canonical vector bundle map $(\tau')^{!}: \varphi^{!}(E) \rightarrow \varphi^{!}(E)$ induced from $\tau'$ via the universal property of the pullback. There is no restriction on $\hat{\tau}_{1}$. 

Note that one can find a simple criterion on $\hat{\tau}_{0}$ to be of this form. The original vector bundle $K^{\perp}$ can be canonically identified with the pullback $\varpi_{0}^{!}(K^{\perp} / H)$. $\hat{\tau}_{0}$ thus induces the vector bundle map $\hat{\tau}_{0}^{!}: K^{\perp} \rightarrow K^{\perp}$. This map must be $G$-equivariant. 
\end{example}

\begin{example}[\textbf{Poisson--Lie T-duality II}]
Recall the Courant algebroid relation $R_{H,H'}: E'_{0} \dra E'_{1}$ obtained in Example \ref{ex_PLT}. Let $\tau'_{0}$ and $\tau'_{1}$ be a generalized metric on $E'_{0}$ and $E'_{1}$, respectively. 

Recall that on any exact Courant algebroid $E'_{0}$ over $N$, the choice of a generalized metric $\tau'_{0}$ corresponds to a Riemannian metric $g'_{0}$ on $N$ and a closed $3$-form $H'_{0} \in \Omega^{3}(N)$ representing the Ševera class of $E'_{0}$. These may serve as target backgrounds for a two-dimensional $\sigma$-model. Let $g'_{1}$ and $H'_{1}$ be a Riemannian metric and a $3$-form on $N'$ constructed from $\tau'_{1}$.

One says that the backgrounds $(g'_{0},H'_{0})$ and $(g'_{1},H'_{1})$ are \textbf{Poisson--Lie T-dual}, iff the Courant algebroid relation $R_{H,H'}: E'_{0} \dra E'_{1}$ defines a generalized isometry of $\tau'_{0}$ and $\tau'_{1}$. There is an obvious way how to construct such pairs. Fix a generalized metric $\tau'$ on $E'$. By previous example, there is a unique generalized metric $\tau'_{0}$ on $E'_{0}$ making $R(H): E'_{0} \rat E'$ into a generalized isometry of $\tau'_{0}$ and $\tau'$. Repeat this for $H'$ to obtain $\tau'_{1}$. Finally, it follows immediately from Proposition \ref{tvrz_genisocomp} that $R_{H,H'} = R(H')^{T} \circ R(H)$ is a generalized isometry of $\tau'_{0}$ and $\tau'_{1}$. This is precisely the construction in described in \cite{Severa:2015hta,Severa:2017kcs}. 
\end{example}
\section{Relations and connections} \label{sec_connections}
Let $(E,\rho,\<\cdot,\cdot\>,[\cdot,\cdot])$ be a Courant algebroid over $M$. Recall that a \textbf{Courant algebroid connection} $\cD$ is an $\R$-bilinear map $\cD: \Gamma(E) \times \Gamma(E) \rightarrow \Gamma(E)$ satisfying the Leibniz rules
\begin{equation} \label{eq_connLeibniz}
\cD(f\psi,\psi') = f \cD(\psi,\psi'), \; \; \cD(\psi,f\psi') = f \cD(\psi,\psi') + \Li{\rho(\psi)}(f)\psi',
\end{equation}
for all $\psi,\psi' \in \Gamma(E)$ and $f \in C^{\infty}(M)$, together with the metric compatibility condition
\begin{equation}
\Li{\rho(\psi)} \<\psi',\psi''\> = \< \cD(\psi,\psi'), \psi''\> + \< \psi', \cD(\psi,\psi'') \>,
\end{equation}
for all $\psi,\psi',\psi'' \in \Gamma(E)$. We usually write $\cD_{\psi} := \cD(\psi,\cdot)$ for the corresponding covariant derivative along the section $\psi$. See \cite{alekseevxu} and \cite{2007arXiv0710.2719G} or our review in \cite{Jurco:2016emw}.

Now, let $(E_{1},\rho_{1},\<\cdot,\cdot\>_{1},[\cdot,\cdot]_{1})$ and $(E_{2},\rho_{2},\<\cdot,\cdot\>_{2},[\cdot,\cdot]_{2})$ be a pair of Courant algebroids over $M_{1}$ and $M_{2}$, respectively. Let $\cD^{1}$ and $\cD^{2}$ be a Courant algebroid connection on $E_{1}$ and $E_{2}$, respectively. Let $R: E_{1} \dra E_{2}$ be a Courant algebroid relation from $E_{1}$ to $E_{2}$. We would like to establish some kind of compatibility condition of $\cD^{1}$ and $\cD^{2}$ with $R$. The idea is very similar to the previous section. First, one combines $\cD^{1}$ and $\cD^{2}$ into a single connection $\cD$ on $E_{1} \times \ol{E}_{2}$. Namely, for all $\psi_{1},\phi_{1} \in \Gamma(E_{1})$ and all $\psi_{2},\phi_{2} \in \Gamma(E_{2})$, one defines 
\begin{equation}
\cD_{(\psi_{1},\phi_{1})}(\psi_{2},\phi_{2}) := ( \cD^{1}_{\psi_{1}} \phi_{1}, \cD^{2}_{\psi_{2}} \phi_{2}).
\end{equation}
This sets $\cD$ on generators and one extends it to all sections using the Leibniz rules (\ref{eq_connLeibniz}). It follows easily that $\cD$ defines a Courant algebroid connection on both $E_{1} \times \ol{E}_{2}$ and $E_{1} \times E_{2}$.
\begin{definice}
We say that the connections \textbf{$\cD^{1}$ and $\cD^{2}$ are $R$-related} and write $\cD^{1} \sim_{R} \cD^{2}$, if for all $\psi,\psi' \in \Gamma(E_{1} \times \ol{E}_{2};R)$, one has $\cD_{\psi}\psi' \in \Gamma(E_{1} \times \ol{E}_{2};R)$. $\cD$ is the connection on $E_{1} \times \ol{E}_{2}$ constructed in the previous paragraph. 
\end{definice}

\begin{rem} \label{rem_conncomp}
Note that in principle, this condition is very similar to the involutivity of the subbundle $R$, except that the $\R$-bilinear operation $[\cdot,\cdot]$ is now replaced by $\cD$. The situation is now a lot simpler due to the $C^{\infty}$-linearity of $\cD$ in the first argument. In particular, there holds an analogue of Proposition \ref{tvrz_invongen} allowing one to prove everything on local generators. 
\end{rem}
\begin{lemma}
Suppose $\cD^{1} \sim_{R} \cD^{2}$. Let $\psi_{1} \sim_{R} \psi_{2}$ and $\phi_{1} \sim_{R} \phi_{2}$. Then $\cD^{1}_{\psi_{1}} \phi_{1} \sim_{R} \cD^{2}_{\psi_{2}} \phi_{2}$. 
\end{lemma}
This follows easily from the definitions. There is an analogue of Proposition \ref{tvrz_genisocomp}. Note that its proof is significantly more involved as we no longer deal with vector bundle maps. However, we can make its proof very brief thanks to Remark \ref{rem_conncomp}.  
\begin{tvrz} \label{tvrz_relconncomp}
Let $(E_{1},\cD^{1})$, $(E_{2},\cD^{2})$ and $(E_{3},\cD^{3})$ be a triple of Courant algebroids equipped with Courant algebroid connections. Suppose $R: E_{1} \dra E_{2}$ and $R': E_{2} \dra E_{3}$ is a pair of cleanly composing Courant algebroid connection, such that $\cD^{1} \sim_{R} \cD^{2}$ and $\cD^{2} \sim_{R'} \cD^{3}$. 

Then $\cD^{1} \sim_{R' \circ R} \cD^{3}$, one has $\cD^{2} \sim_{R^{T}} \cD^{1}$, and for any Courant algebroid $(E,\cD)$ equipped with a Courant algebroid connection $\cD$, one has $\cD \sim_{\Delta(E)} \cD$.
\end{tvrz}
\begin{proof}
The proof of $\cD \sim_{\Delta(E)} \cD$ is the same as the discussion in Example \ref{ex_CArel} (i). If $\cD^{1} \sim_{R} \cD^{2}$ and $\cD^{2} \sim_{R'} \cD^{3}$, the proof of $\cD^{1} \sim_{R' \circ R} \cD^{3}$ is completely analogous to the one of Theorem \ref{thm_composition} (and all the preceding lemmas) where one replaces all Courant brackets by connections. The fact that $\cD^{2}$ and $\cD^{1}$ are $R^{T}$-related is obvious. 
\end{proof}
\begin{example}
Let $\F: E_{1} \rightarrow E_{2}$ be a classical Courant algebroid morphism over $\varphi$, see Subsection \ref{subsec_graph}. Similarly to one of the statements of Theorem \ref{thm_grF}, the condition $\cD^{1} \sim_{\gr(\F)} \cD^{2}$ is equivalent to the following condition:

Let $\psi_{1},\psi'_{1} \in \Gamma(E_{1})$ be any two sections. Let $(\psi_{\mu})_{\mu=1}^{\rk(E_{2})}$ be any local frame for $E_{2}$ over $U$. By construction, there are unique smooth functions $f^{\mu},g^{\nu} \in C^{\infty}(\varphi^{-1}(U))$, such that on $\varphi^{-1}(U)$, one can write $\F^{!}(\psi_{1}) = f^{\mu} \psi_{\mu}^{!}$, $\F^{!}(\psi'_{1}) = g^{\nu} \psi_{\nu}^{!}$. Then on $\varphi^{-1}(U)$, the equation
\begin{equation} \label{eq_Frelatedconn}
\F^{!}(\cD^{1}_{\psi_{1}} \psi'_{1}) = f^{\mu}g^{\nu} ( \cD^{2}_{\psi_{\mu}} \psi_{\nu})^{!} + \Li{\rho_{1}(\psi_{1})}(g^{\nu}) \psi^{!}_{\nu}
\end{equation}
must hold. Similarly to Remark \ref{rem_graphoverdiff}, if $\psi_{1} \sim_{\F} \psi_{2}$, $\phi_{1} \sim_{\F} \phi_{2}$, then $\cD_{\psi_{1}}\phi_{1} \sim_{\F} \cD_{\psi_{2}}\phi_{2}$. Moreover, if $\varphi$ is a diffeomorphism and one defines $\F(\psi_{1}) = \F \circ \psi_{1} \circ \varphi^{-1}$ for every $\psi_{1} \in \Gamma(E_{1})$, then (\ref{eq_Frelatedconn}) is equivalent to the expected condition
\begin{equation}
\F( \cD^{1}_{\psi_{1}} \psi'_{1}) = \cD^{2}_{\F(\psi_{1})} \F(\psi'_{1}),
\end{equation}
which has to be valid for all $\psi_{1},\psi'_{1} \in \Gamma(E_{1})$.
\end{example}
Now, recall that for every Courant algebroid connection $\cD$ on $(E,\rho,\<\cdot,\cdot\>,[\cdot,\cdot])$, one can define its torsion $3$-form $T_{\cD} \in \Omega^{3}(E)$. See e.g. \cite{2007arXiv0710.2719G}. For all $\psi,\psi',\psi'' \in \Gamma(E)$, set 
\begin{equation} \label{eq_torsion}
T_{\cD}(\psi,\psi',\psi'') := \< \cD_{\psi}\psi' - \cD_{\psi'}\psi - [\psi,\psi'], \psi''\> + \< \cD_{\psi''}\psi, \psi'\>.
\end{equation}
It follows from axioms C1), C3) and C4) that it is is $C^{\infty}$-linear in every input and completely skew-symmetric. Recall that for every Courant algebroid relation $R: E_{1} \dra E_{2}$, we have introduced the concept of $R$-related covariant tensors, see Definition \ref{def_tensRrelated} and subsequent Lemma \ref{lem_tensRrelated}. 
\begin{tvrz} \label{tvrz_torzRrelated}
Let $(E_{1},\cD^{1})$ and $(E_{2},\cD^{2})$ be pair of Courant algebroids equipped with Courant algebroid connections. Let $R: E_{1} \dra E_{2}$ be a Courant algebroid relation over $S$. 

Then if $\cD^{1} \sim_{R} \cD^{2}$, then also $T_{\cD^{1}} \sim_{R} T_{\cD^{2}}$. 
\end{tvrz}
\begin{proof}
First, note that the torsion $3$-form of the induced connection $\cD$ on $E_{1} \times \ol{E}_{2}$ can be written as $T_{\cD} = p_{1}^{\ast}(T_{\cD^{1}}) - p_{2}^{\ast}(T_{\cD^{2}})$. This is easy to see from definitions, the minus sign coming from the sign flip of the pairing on $\ol{E}_{2}$. In view of Lemma \ref{lem_tensRrelated}, it then suffices to prove that $T_{\cD}(\psi,\psi',\psi'')|_{S} = 0$ for all $\psi,\psi',\psi'' \in \Gamma(E_{1} \times \ol{E}_{2}; R)$. But that follows immediately from (\ref{eq_torsion}) using the assumption on $\cD$ and the fact that $R$ is isotropic and involutive. 
\end{proof}
\begin{example}
Let us consider a Courant algebroid morphism $R(H): E'_{0} \rat E'$ over $\varphi$ we have defined in Subsection \ref{subsec_tworeduced}. Let $\cD^{0}$ and $\cD'$ be Courant algebroid connections on $E'_{0}$ and $E'$, respectively. Let us examine when $\cD^{0} \sim_{R(H)} \cD'$. 

First, one can induce an $\R$-bilinear map $\hat{\cD}^{0}$ on the direct sum $K^{\perp}/H \oplus (P \times_{\Ad} \q)$ using the isomorphism $\fPsi_{1}$. Now, let $\psi',\phi' \in \Gamma(E')$ and consider the sections $\psi := (\fPsi_{1}( \varphi^{\ast}(\psi'),0), \psi')$ and $\phi := (\fPsi_{1}(\varphi^{\ast}(\phi'),0),\phi')$ in $\Gamma(E'_{0} \times \ol{E}', R(H))$. Then $\cD_{\psi}\phi \in \Gamma(E'_{0} \times \ol{E}', R(H))$, iff 
\begin{equation}
\hat{\cD}^{0}_{(\varphi^{\ast}(\psi'),0)} (\varphi^{\ast}(\phi'),0) = ( \varphi^{\ast}( \cD'_{\psi}\phi'), 0 ). 
\end{equation}
This uniquely determines $\cD^{0}$ in the $\fPsi_{1}(K^{\perp}/H \oplus 0)$ corner. In particular, if $\q = 0$, to a given Courant algebroid connection $\cD'$ on $E'$, there is a unique $\cD^{0}$ on $E'_{0}$ such that $\cD^{0} \sim_{R(H)} \cD'$. This is crucial for the compatibility of supergravity with Poisson--Lie T-duality, see \cite{Jurco:2019tgt,Severa:2018pag}. 
\end{example}
To conclude this section, recall that to every Courant algebroid connection $\cD$, there is a corresponding generalized Riemann tensor $R_{\cD} \in \T_{4}(E)$. We have defined it for a general Courant algebroid in \cite{Jurco:2016emw}, inspired by the double field theory paper \cite{Hohm:2012mf}. First, define $R_{\cD}^{(0)}$ by
\begin{equation}
R_{\cD}^{(0)}(\phi',\phi,\psi,\psi') = \< \cD_{\psi}(\cD_{\psi'}\phi) - \cD_{\psi'}( \cD_{\psi}\phi) - \cD_{[\psi,\psi']} \phi, \phi' \>,
\end{equation}
for all $\psi,\psi',\phi,\phi' \in \Gamma(E)$. Note that $R^{(0)}_{\cD}$ is not a tensor. Next, define the $\R$-bilinear map $\fK$ by formula $\< \fK(\psi,\psi'), \phi\> = \< \cD_{\phi}\psi, \psi'\>$, for all $\psi,\psi',\phi \in \Gamma(E)$. The tensor $R_{\cD}$ is then defined as
\begin{equation}
R_{\cD}(\phi',\phi,\psi,\psi') = \frac{1}{2}( R^{(0)}_{\cD}(\phi',\phi,\psi,\psi') + R^{(0)}(\psi',\psi,\phi,\phi') + \< \fK(\phi,\phi'), \fK(\psi,\psi')\>),
\end{equation} 
for all $\psi,\psi',\phi,\phi' \in \Gamma(E)$. Albeit it may seem quite strange, the resulting tensor has very nice symmetries, See Proposition 4.10 and Theorem 4.13 of \cite{Jurco:2016emw}. 
\begin{tvrz} \label{tvrz_Riemannirelated}
Let $(E_{1},\cD^{1})$ and $(E_{2},\cD^{2})$ be pair of Courant algebroids equipped with torsion-free Courant algebroid connections. Let $R: E_{1} \dra E_{2}$ be a Courant algebroid relation over $S$. 

Then if $\cD^{1} \sim_{R} \cD^{2}$, then also $R_{\cD^{1}} \sim_{R} R_{\cD^{2}}$. 
\end{tvrz} 
\begin{proof}
Let $\cD$ be the induced connection on $E_{1} \times \ol{E}_{2}$. Directly from the definitions, one finds $R_{\cD} = p_{1}^{\ast}(R_{\cD^{1}}) - p_{2}^{\ast}( R_{\cD^{2}})$. In view of Lemma \ref{lem_tensRrelated}, it suffices to prove that $R_{\cD}(\phi',\phi,\psi,\psi')|_{S} = 0$ for all $\psi,\psi',\phi,\phi' \in \Gamma(E_{1} \times \ol{E}_{2};R)$. It is easy to see that $R^{(0)}(\phi',\phi,\psi,\psi')|_{S} = 0$ for any $\cD^{1}$ and $\cD^{2}$. We claim that for torsion-free connections, one has $\fK(\psi,\psi') \in \Gamma(E_{1} \times \ol{E}_{2};R)$ for all $\psi,\psi' \in \Gamma(E_{1} \times \ol{E}_{2};R)$. It suffices to show that $\< \fK(\psi,\psi'), \phi \>|_{S} = 0$ for all $\phi \in \Gamma(E_{1} \times \ol{E}_{2};R^{\perp})$. Note that as $T_{\cD} = p_{1}^{\ast}(T_{\cD^{1}}) - p_{2}^{\ast}( T_{\cD^{2}})$, the connection $\cD$ is torsion-free, iff both $\cD^{1}$ and $\cD^{2}$ are. Hence by assumption, we have $T_{\cD} = 0$. It follows that one can write 
\begin{equation}
\< \fK(\psi,\psi'), \phi \>|_{S} = \< [\psi,\psi'] - \cD_{\psi}\psi' + \cD_{\psi'}\psi, \phi \>|_{S} = 0, 
\end{equation}
where we have used the involutivity and isotropy of $R$ together with the assumption $\cD^{1} \sim_{R} \cD^{2}$. This implies that the restriction of the last term in $R_{\cD}(\phi',\phi,\psi,\psi')$ to $S$ gives zero as $R$ is isotropic. Note that if $\cD$ is not torsion-free, one can only prove that $\fK(\psi,\psi') \in \Gamma(E_{1} \times \ol{E}_{2}; R^{\perp})$. 
\end{proof}
\begin{example} \label{ex_splitLagRiemanni}
If the fiber-wise metric on $E_{1} \times \ol{E}_{2}$ has a split signature and $R = R^{\perp}$, the above proposition holds for arbitrary Courant algebroid connections. This happens e.g. when $R = \gr(\F)$ for a fiber-wise bijective vector bundle map $\F: E_{1} \rightarrow E_{2}$. The statement $R_{\cD^{1}} \sim_{\gr(\F)} R_{\cD^{2}}$ then turns into the expected property of generalized Riemann tensors:
\begin{equation}
R_{\cD^{2}}( \F(f'),\F(f),\F(e),\F(e')) = R_{\cD^{1}}(f',f,e,e'),
\end{equation}
for all $e,e',f,f' \in (E_{1})_{m_{1}}$ and all $m_{1} \in M_{1}$. 
\end{example} 
Let us show that the vanishing torsion of $\cD^{1}$ and $\cD^{2}$ is a necessary assumption. 
\begin{example}
Let $E_{2} = (\g, [\cdot,\cdot]_{\g}, \<\cdot,\cdot\>_{\g})$ be a quadratic Lie algebra viewed as a Courant algebroid over a point. Suppose there exists its Lie subalgebra $ \h \subseteq \g$, such that $\h \cap \h^{\perp} = 0$. In other words, the restriction $\<\cdot,\cdot\>_{\h}$ of $\<\cdot,\cdot\>_{\g}$ to $\h$ is non-degenerate. Let $E_{1} = (\h,[\cdot,\cdot]_{\h}, \<\cdot,\cdot\>_{\h})$. It is easy to see that the inclusion $\mathrm{i}: \h \rightarrow \g$ is a classical Courant algebroid morphism. 

Now, define the Courant algebroid connections $\cD^{1}_{x}y = [x,y]_{\h}$ and $\cD^{2}_{u}v = [u,v]_{\g}$ for all $x,y \in \h$ and $u,v \in \g$. Obviously, $\cD^{1} \sim_{\gr(\mathrm{i})} \cD^{2}$. In general, these connections are not torsion-free. For all $x,y,z \in \h$, one has 
\begin{equation}
T_{\cD^{1}}(x,y,z) = 2 \< [x,y]_{\h}, z \>_{\h} =: 2 \chi_{\h}(x,y,z),
\end{equation}
where $\chi_{\h} \in \Lambda^{3} \h^{\ast}$ is the well-known Cartan $3$-form corresponding to $(\h, [\cdot,\cdot]_{\h}, \<\cdot,\cdot\>_{\h})$. The same holds for $T_{\cD^{2}}$. Let us evaluate $\fK$ on $\Gamma(\h \times \ol{\g} ; \gr(\mathrm{i}))$, that is on elements of $\h \times \ol{\g}$ of the form $(x,\mathrm{i}(x))$ for $x \in \h$. For every $x,y \in \h$, one finds
$\fK((x,\mathrm{i}(x)), (y,\mathrm{i}(y))) = ( [x,y]_{\h}, \mathrm{i}( [x,y]_{\h})) \in \Gamma(\h \times \ol{\g}, \gr(\mathrm{i}))$. 

This shows that the proposition can work even for connections with a non-zero torsion. Not always, though. We can still break things. Consider a modified connection $\ol{\cD}^{2}$ defined by 
\begin{equation}
\ol{\cD}^{2}_{u}v = [u,v]_{\g} + \fk(u,v),
\end{equation}
where $\fk: \g \times \g \rightarrow \g$ is some bilinear map satisfying $\<\fk(u,v),v\>_{\g} = 0$ for all $u,v \in \g$ and $\fk(\mathrm{i}(x),\mathrm{i}(y)) = 0$ for all $x,y \in \h$. This ensures that $\ol{\cD}^{2}$ is a Courant algebroid connection satisfying $\cD^{1} \sim_{\gr(\mathrm{i})} \ol{\cD}^{2}$. Let $\mathrm{j}: \h^{\perp} \rightarrow \g$ be the inclusion and let $\pi_{\h}: \g \rightarrow \h$ be the projection. Let $\fk_{0}: \h \times \h \rightarrow \h^{\perp}$ be a skew-symmetric bilinear map. For all $u,v,w \in \g$, define 
\begin{equation}
\< \fk(u,v), w \>_{\g} := \<u, \mathrm{j}( \fk_{0}( \pi_{\h}(v), \pi_{\h}(w))) \>_{\g}. 
\end{equation}
This $\fk$ satisfies the restrictions imposed above. $\ol{\fK}$ corresponding to $\ol{\cD}^{2}$ then reads 
\begin{equation}
\ol{\fK}((x,\mathrm{i}(x)), (y, \mathrm{i}(y))) = ([x,y]_{\h}, \mathrm{i}([x,y]_{\h}) + \mathrm{j}( \fk_{0}(x,y))),
\end{equation}
for all $x,y \in \h$. For $\fk_{0} \neq 0$, the right hand side is not an element of $\gr(\mathrm{i})$. Moreover, one finds 
\begin{equation}
\< \ol{\fK}((x,\mathrm{i}(x)), (y, \mathrm{i}(y))), \ol{\fK}((x,\mathrm{i}(x)), (y, \mathrm{i}(y))) \> = - \< \fk_{0}(x,y), \fk_{0}(x,y) \>_{\h^{\perp}}.
\end{equation}
If we can choose $\fk_{0}$ and $x,y \in \h$ so that $\fk_{0}(x,y)$ is not an isotropic vector with respect to $\<\cdot,\cdot\>_{\h^{\perp}}$, we obtain $R_{\cD}((x,\mathrm{i}(x)),(y,\mathrm{i}(y)),(x,\mathrm{i}(y)),(x,\mathrm{i}(y))) \neq 0$, that is our counterexample. 

For illustration, let $\g = \mathfrak{o}(4)$ and $\<\cdot,\cdot\>_{\g}$ be its Killing form. It is negative-definite as $\g$ is compact. Let $\h \cong \mathfrak{o}(3)$ be its Lie subalgebra induced by one of the obvious inclusions $\gSO(3) \rightarrow \gSO(4)$. Both restrictions $\<\cdot,\cdot\>_{\h}$ and $\<\cdot,\cdot\>_{\h^{\perp}}$ have to be negative-definite too. In particular, one may choose an arbitrary non-zero $\fk_{0}$. Note that the signatures of $\<\cdot,\cdot\>_{\g}$ and $\<\cdot,\cdot\>_{\h}$ are $(0,6)$ and $(0,3)$, respectively. It follows that the signature of $\<\cdot,\cdot\>$ on $\h \times \ol{\g}$ is $(6,3)$ and $\dim(\gr(\mathrm{i})) = 3$, hence in fact, it is maximally isotropic.  This shows that the assumption on the split signature in Example \ref{ex_splitLagRiemanni} cannot be relaxed. 
\end{example}
Note that in general, there is no analogue of Proposition \ref{tvrz_Riemannirelated} for the generalized Ricci tensor $\Ric_{\cD}$, or for the generalized divergence $\Div_{\cD}: \Gamma(E) \rightarrow C^{\infty}(M)$. See \cite{Jurco:2016emw} for the definitions. This is because they are defined using traces of vector bundle maps which (in general) do not interplay well with Courant algebroid relations.
\appendix
\section{Linear algebra supplement}\label{sec_supplement}
In this appendix, we will throw in some linear algebra we have used in this paper. Naturally, there is nothing really new among these lines. 

We will work exclusively with finite-dimensional real vector spaces. By a quadratic vector space $(V,\<\cdot,\cdot\>)$ we mean $V$ endowed with a non-degenerate symmetric bilinear form $\<\cdot,\cdot\>$. For any its subspace $W \subseteq V$, by $W^{\perp}$ we mean its orthogonal complement with respect to $\<\cdot,\cdot\>$. 

\begin{definice}
Let $(V,\<\cdot,\cdot\>)$ be a quadratic vector space. We say that $\tau \in \End(V)$ is a \textbf{compatible involution}, if $\tau^{2} = 1$ and $\fg(x,y) = \<x,\tau(y)\>$ defines a scalar product $\fg$ on $V$.
\end{definice}
One can show that the choice of a compatible involution $\tau$ is completely equivalent to the choice of a maximal positive subspace $V_{+} \subseteq V$ with respect to $\<\cdot,\cdot\>$. One can find $V_{+}$ as $+1$ eigenspace of $\tau$. Conversely, starting from $V_{+}$, one can argue that $V_{-} := (V_{+})^{\perp}$ is maximal negative subspace with respect to $\<\cdot,\cdot\>$. Hence $V = V_{+} \oplus V_{-}$ and $\tau$ can be defined so that $V_{\pm}$ are its $\pm 1$ eigenspaces.
\begin{tvrz}
On every quadratic vector space $(V,\<\cdot,\cdot\>)$, there exists a compatible involution. 
\end{tvrz}
\begin{proof}
The existence of a maximal positive subspace is a standard statement, see \cite{lam2005introduction}. However, we provide the proof which can be easily generalized to vector bundles. Fix any scalar product $\fg_{0}$ on $V$. Define a linear map $\sigma \in \End(V)$ by the formula $\<x,y\> = \fg_{0}( x, \sigma(y))$, for all $x,y \in V$. 

The map $\sigma$ is $\fg_{0}$-symmetric, that is $\fg_{0}(x, \sigma(y)) = \fg_{0}(\sigma(x), y)$ for all $x,y \in V$. It follows that its square $\sigma^{2}$ is $\fg_{0}$-symmetric and positive definite with respect to $\fg_{0}$, that is $\fg_{0}(x, \sigma^{2}(x)) > 0$ for all non-zero $x \in V$. There is thus its unique square root $\eta = (\sigma^{2})^{\frac{1}{2}}$ which is also $\fg_{0}$-symmetric and positive definite with respect to $\fg_{0}$. Define $\tau := \eta^{-1} \sigma$. We claim that $\tau$ is the compatible involution. Note that all involved maps commute, being constructed from the $\fg_{0}$-symmetric map $\sigma$. Hence 
\begin{equation}
\tau^{2} = \eta^{-1} \sigma \eta^{-1} \sigma = (\eta^{2})^{-1} \sigma^{2} = (\sigma^{2})^{-1} \sigma^{2} = 1. 
\end{equation}
Finally, one has $\fg(x,y) \equiv \<x,\tau(y)\> = \<x, \sigma( \eta^{-1}(y)) \> = \fg_{0}(x, \eta^{-1}(y))$. As $\eta^{-1}$ is $\fg_{0}$-symmetric and positive definite with respect to $\fg_{0}$, this proves that $\fg$ is a scalar product. 
\end{proof}
We obtain an immediate corollary of this statement. See Section \ref{sec_geniso} for definitions.
\begin{cor} \label{cor_genmetric}
On every quadratic vector bundle $(E,\<\cdot,\cdot\>)$, there exists a generalized metric.
\end{cor}
\begin{proof}
On every vector bundle, there exists a positive definite fiber-wise metric $\fg_{0}$. The vector bundle map $\tau$ can be then fiber-wise defined using the same formulas as its vector space cousin in the previous proposition. The only non-trivial statement is then its smoothness. Without going into too much detail, this at its core follows from the fact that for each $\alpha \in \R$, the map $\fA \mapsto \fA^{\alpha}$ is smooth on the open cone of symmetric positive definite matrices. 
\end{proof}

\begin{tvrz}[\textbf{Coisotropic reduction}] \label{tvrz_coiso}
Let $(V,\<\cdot,\cdot\>)$ be a quadratic space. Let $C \subseteq V$ be a coisotropic subspace, that is $C^{\perp} \subseteq C$. Then there is a canonical quadratic space structure $(V', \<\cdot,\cdot\>')$ on the quotient vector space $V' = C / C^{\perp}$. 

If $(p,q)$ is the signature of $\<\cdot,\cdot\>$ and $k = \dim(C^{\perp})$, then the signature of $\<\cdot,\cdot\>'$ is $(p-k,q-k)$. 

Let $\natural: C \rightarrow V'$ denote the quotient map. If $L \subseteq V$ is (maximally) isotropic with respect to $\<\cdot,\cdot\>$, then $L' = \natural(L \cap C)$ is (maximally) isotropic with respect to $\<\cdot,\cdot\>'$. One has $L'^{\perp}  = \natural(L^{\perp} \cap C)$. 
\end{tvrz}
\begin{proof}
Set $\< \natural(x), \natural(y)\>' := \<x,y\>$ for all $x,y \in C$. It is an easy exercise to see that $\<\cdot,\cdot\>'$ is a well-defined non-degenerate symmetric bilinear form. Let $\tau$ be a compatible involution on $V$ and let $V_{\pm} \subseteq V$ be the corresponding $\pm 1$ eigenbundles. Next, note that the inertia of the restricted bilinear form $\<\cdot,\cdot\>|_{C}$ is $(p',q',k)$, where $(p',q')$ is the signature of $\<\cdot,\cdot\>'$. We have $p' + q' + k = \dim(C) = n - k$, where $n = \dim(V)$. Consider the subspace $V'_{+} = V_{+} \cap C$ positive with respect to $\<\cdot,\cdot\>|_{C}$. For its dimension, we get
\begin{equation}
\begin{split}
\dim(V'_{+}) = & \ \dim(V_{+} \cap C) = \dim(V_{-}^{\perp} \cap C) = \dim((V_{-} + C^{\perp})^{\perp}) = n - \dim(V_{-} + C^{\perp}) \\
= & \ n - (\dim(V_{-}) + \dim(C^{\perp}) - \dim(V_{-} \cap C^{\perp})) = n - q - k + \dim(V_{-} \cap C^{\perp}) \\
= & \ p - k + \dim(V_{-} \cap C^{\perp}). 
\end{split}
\end{equation}
But $V_{-} \cap C^{\perp} = 0$ as $C^{\perp}$ is isotropic and $V_{-}$ is negative with respect to $\<\cdot,\cdot\>$. Hence $\dim(V'_{+}) = p - k$. Defining the negative subspace $V'_{-} = V_{-} \cap C$, one can similarly show $\dim(V'_{-}) = q - k$. We thus obtain the estimates $p' \geq p - k$ and $q' \geq q - k$. On the other hand, one has $\dim(V') = n - 2k$. Hence $p' + q' = n - 2k$. This can only happen if $p' = p - k$ and $q' = q - k$. 

Let $L \subseteq V$ be isotropic. It is easy to see that $L' = \natural(L \cap C)$ is isotropic. If $L$ is maximally isotropic, then $L^{\perp} = L \sqcup L_{0}$, where $L_{0} = \{ v \in L^{\perp} \; | \; \<v,v\> \neq 0 \}$. As $C^{\perp}$ is isotropic, it follows that $L^{\perp} \cap C^{\perp} = L \cap C^{\perp}$. Using the nullity-rank theorem, one finds 
\begin{equation}
\begin{split}
\dim(L') = & \ \dim(L \cap C) - \dim(\ker(\natural|_{L \cap C})) = \dim(L \cap C) - \dim(L \cap C^{\perp}) \\
= & \ \dim( (L^{\perp} + C^{\perp})^{\perp}) - \dim(L \cap C^{\perp}) = n - \dim(L^{\perp} + C^{\perp}) - \dim(L \cap C^{\perp}) \\
= & \ n - ( \dim(L^{\perp}) + \dim(C^{\perp}) - \dim(L^{\perp} \cap C^{\perp})) - \dim(L \cap C^{\perp}) \\
= & \ \dim(L) - k + \dim(L^{\perp} \cap C^{\perp}) - \dim(L \cap C^{\perp}) = \dim(L) - k \\
= & \ \min \{p,q\} - k = \min\{ p',q'\}. 
\end{split}
\end{equation}
This shows that $L'$ is maximally isotropic. 

Finally, let $L \subseteq V$ be any isotropic subspace. Obviously, $\natural(L^{\perp} \cap C) \subseteq L'^{\perp}$. One finds
\begin{equation}
\begin{split}
\dim(\natural(L^{\perp} \cap C)) = & \ \dim(L^{\perp} \cap C) - \dim(L^{\perp} \cap C^{\perp}) = \dim((L + C^{\perp})^{\perp}) - \dim((L + C)^{\perp}) \\
= & \ \dim(L + C) - \dim(L + C^{\perp}) \\
= & \ \dim(C) - \dim(C^{\perp}) - \dim(L \cap C) + \dim(L \cap C^{\perp}) \\
= & \ \dim(V') - \dim(L') = \dim(L'^{\perp}). 
\end{split}
\end{equation}
This proves the last claim of the proposition. 
\end{proof}
The following example is a vector space version of Theorem \ref{thm_composition}. 
\begin{example} \label{ex_coisored}
Let $(V_{1},\<\cdot,\cdot\>_{1})$, $(V_{2},\<\cdot,\cdot\>_{2})$ and $(V_{3},\<\cdot,\cdot\>_{3})$ be quadratic vector spaces. 

Let $R \subseteq V_{1} \times \ol{V}_{2}$ and $R' \subseteq V_{2} \times \ol{V}_{3}$ be isotropic subspaces, where Cartesian products are assumed to be equipped with product bilinear forms and overlines indicate the flipped sign. 

Let $V = V_{1} \times \ol{V}_{2} \times V_{2} \times \ol{V}_{3}$ and $C = V_{1} \times \Delta(V_{2}) \times V_{3}$, where $\Delta(V_{2})$ denotes the diagonal embedding of $V_{2}$ into $V_{2} \times \ol{V}_{2}$. It follows that $C$ is a coisotropic subspace as $C^{\perp} = 0 \times \Delta(V_{2}) \times 0$. It is easy to see that the quotient quadratic space $C / C^{\perp}$ can be identified with $V_{1} \times \ol{V}_{3}$ and the quotient map $\natural$ coincides with the obvious projection $p: V \rightarrow V_{1} \times \ol{V}_{3}$. 

Now, let $L = R \times R'$. This is an isotropic subspace of $V$. It immediately follows that $L' = p(L \cap C)$ is isotropic in $V_{1} \times \ol{V}_{3}$. But $L'$ is the well-known composition
\begin{equation}
R' \circ R = \{ (x_{1},x_{3}) \in V_{1} \times \ol{V}_{3} \; | \; (x_{1},x_{2}) \in R \text{ and } (x_{2},x_{3}) \in R' \text{ for some } x_{2} \in V_{2} \}. 
\end{equation}
This shows that quadratic vector spaces together with their relations (isotropic subbundles of the products) form a nice category. Now, note that Proposition \ref{tvrz_coiso} also claims that if $L = R \times R'$ is maximally isotropic, then so is $R' \circ R$. Beware that even if $R$ and $R'$ are maximally isotropic, their product $R \times R'$, in general, is \textit{not}.
\end{example} 
\begin{tvrz} \label{tvrz_inertias}
Let $(V,\<\cdot,\cdot\>)$ be a quadratic vector space. For any subspace $P \subseteq V$, let $\In(P) \in \Z_{\geq 0}^{3}$ denote the inertia of the restricted bilinear form $\<\cdot,\cdot\>|_{P}$. Then
\begin{equation} \label{eq_inertias}
\In(V) = \In(P) + \In(P^{\perp}) + (k_{0},k_{0},-2k_{0}),
\end{equation}
where $k_{0} = \dim(P \cap P^{\perp})$. 
\end{tvrz}
\begin{proof}
Write $\In(P) = (n_{+}(P),n_{-}(P),n_{0}(P))$. It is easy to see that one has $n_{0}(P) = n_{0}(P^{\perp}) = k_{0}$. Moreover, we have $n_{0}(V) = 0$. This shows that the third component of the equation (\ref{eq_inertias}) is obvious. We will prove the equation by induction on $n = \dim(V)$. For $n = 0$, the statement is trivial. 

Next, fix $n \geq 1$ and assume that the statement holds for all quadratic vector spaces of dimension strictly lower than $n$. Moreover, note that it suffices to prove just one of the two non-trivial components of the equation. Indeed, suppose that it holds for the $(+)$ case. Then 
\begin{equation}
\begin{split}
n_{-}(V) = & \ n - n_{+}(V) = n - (n_{+}(P) + n_{+}(P^{\perp}) + k_{0}) \\
 = & \ n  - (\dim(P) - n_{-}(P) + \dim(P^{\perp}) - n_{-}(P^{\perp}) - k_{0}) \\
 = & \ n_{-}(P) + n_{-}(P^{\perp}) + k_{0}. 
\end{split}
\end{equation}
Hence it automatically holds also for the $(-)$ component. 

First, assume that $n_{\pm}(P^{\perp}) = 0$. This happens precisely when $P$ is coisotropic and $k_{0} = \dim(P^{\perp})$. The equation (\ref{eq_inertias}) then turns into $n_{\pm}(V) = n_{\pm}(P) + \dim(P^{\perp})$. But we have already shown this in the proof of Proposition \ref{tvrz_coiso}. 

Hence we can assume that one of the two numbers $n_{\pm}(P^{\perp})$ is non-zero. Without the loss of generality, suppose that $n_{+}(P^{\perp}) > 0$. There thus exists an $n_{+}(P^{\perp})$-dimensional positive subspace $Q \subseteq P^{\perp}$. One has $Q \cap Q^{\perp} = 0$ and $\dim(Q^{\perp}) < n$. It follows that $(Q^{\perp}, \<\cdot,\cdot\>|_{Q^{\perp}})$ is a quadratic vector space of dimension strictly lower than $n$. We have $P \subseteq Q^{\perp}$ and the induction hypothesis thus implies the equation
\begin{equation}
n_{+}(Q^{\perp}) = n_{+}(P) + n_{+}(P^{\perp} \cap Q^{\perp}) + k_{0}
\end{equation}
Now, note that $n_{+}(P^{\perp} \cap Q^{\perp}) = 0$. If there would be a positive vector $v \in P^{\perp} \cap Q^{\perp}$, the subspace $Q \oplus \R \{v \}$ would be  a positive subspace with respect to $\<\cdot,\cdot\>|_{P^{\perp}}$ properly containing $Q$, which would contradict $\dim(Q) = n_{+}(P^{\perp})$. We thus obtain the equation $n_{+}(Q^{\perp}) = n_{+}(P) + k_{0}$. On the other hand, as $Q$ is positive, we get $n_{+}(Q) = n_{+}(P^{\perp})$. As $Q \cap Q^{\perp} = 0$, we find
\begin{equation}
\dim(V) = \dim(Q) + \dim(Q^{\perp}) = n_{+}(P) + n_{+}(P^{\perp}) + k_{0}.
\end{equation}
This is the $(+)$ component of (\ref{eq_inertias}) and by above comments, the remaining two follow automatically. Note that the other possibility $n_{-}(P^{\perp}) > 0$ would just lead us to proving the $(-)$ component instead. This finishes the proof. 
\end{proof}
\section*{Acknowledgments}
I would like to thank Branislav Jurčo, Josef Svoboda, David Li-Bland, Eckhard Meinrenken, Pavol Ševera and Vít Tuček for helpful discussions. The author is grateful for a financial support from MŠMT under grant no. RVO 14000. 

\bibliography{bib}	
\end{document}